\documentclass[11pt, reqno]{amsart}
\usepackage{amsmath,amssymb,latexsym,esint,mathrsfs,mathtools}
\usepackage{amsthm}
\usepackage{verbatim,wasysym}

\usepackage[left=2.6cm,right=2.6cm,top=2.6cm,bottom=2.5cm]{geometry}
\usepackage[english]{babel} %!
\usepackage[T1]{fontenc}
\usepackage{setspace} %!
\usepackage{appendix}
\usepackage{etoolbox} %Cambiare aspetto section, subsection etc
\usepackage{microtype}
\usepackage{color,enumitem,graphicx}

\usepackage[%
colorlinks=true,
citecolor=red, %magenta
linkcolor=blue,
urlcolor=green!45!black, 
linktocpage,
pdfpagelabels, 
bookmarksnumbered,
bookmarksopen]
{hyperref}

\usepackage{lmodern}

\usepackage{csquotes} %Suggerito da .log; da usare con babel

\usepackage[hang]{footmisc}
\usepackage{xpatch}
\usepackage{footnotebackref} %!
\makeatletter
\xpretocmd{\@adminfootnotes}{\let\@makefntext\BHFN@OldMakefntext}{}{}
\renewcommand\@makefntext[1]{%
 \@ifundefined{@makefnmark}
 {}
 {%
 \renewcommand\@makefnmark{%
 \mbox{%
 \textsuperscript{%
 \normalfont
 \hyperref[\BackrefFootnoteTag]{\@thefnmark}%
 }%
 }\,%
 }%
 \BHFN@OldMakefntext{#1}%
 }%
}
\makeatother

\patchcmd{\section}{\scshape}{\normalfont \Large \bfseries}{}{}

\patchcmd{\subsection}{\normalfont}{\normalfont \large}{}{}
\patchcmd{\subsection}{-.5em}{.5\linespacing}{}{}

\patchcmd{\subsubsection}{\normalfont}{\noindent \normalfont \large \bfseries \textit}{}{}
\patchcmd{\subsubsection}{-.5em}{.5\linespacing}{}{}

\DeclareRobustCommand{\SkipTocEntry}[5]{}
\protected\def\ignorethis#1\endignorethis{}
\let\endignorethis\relax

\let\oldtocsection=\tocsection
\let\oldtocsubsection=\tocsubsection
\let\oldtocsubsubsection=\tocsubsubsection

\renewcommand{\tocsection}[2]{\hspace{0em}\oldtocsection{#1}{#2}}
\renewcommand{\tocsubsection}[2]{\hspace{1em}\oldtocsubsection{#1}{#2}}
\renewcommand{\tocsubsubsection}[2]{\hspace{2em}\oldtocsubsubsection{#1}{#2}}

\newtheorem{lemma}{Lemma}[section]
\newtheorem{corollary}[lemma]{Corollary}
\newtheorem{proposition}[lemma]{Proposition}
\newtheorem{theorem}[lemma]{Theorem}

\theoremstyle{definition}

\newtheorem{remark}[lemma]{Remark}
\newtheorem{example}[lemma]{Example}
\newtheorem{definition}[lemma]{Definition}

\numberwithin{equation}{section}

\usepackage{orcidlink} %logo orcid
\newcommand{\Omegabar}{\overline{\Omega}}
\newcommand{\Btilde}{\widetilde{\mathscr{B}}}
\newcommand{\R}{\mathbb{R}}
\newcommand{\N}{\mathbb{N}}
\newcommand{\mc}[1]{\mathcal{#1}}
\newcommand{\diam}{{\rm diam}}
\newcommand{\osc}{{\rm osc}}
\newcommand{\Tr}{{\rm Tr}}
\newcommand{\Dom}{{\rm Dom}}

\def\eps{\varepsilon}
\def\abs#1{|#1|}

\def\norm#1{\|#1\|}

\usepackage[bibstyle =numeric, backend=biber, giveninits=true, backref=true, sortcites=true]{biblatex}
\DeclareFieldFormat{journaltitle}{#1} % Rimuove il corsivo dal titolo della rivista
\DeclareFieldFormat[article]{title}{\mkbibemph{#1}} % Corsivo per i titoli degli articoli
\DeclareFieldFormat[book]{title}{\mkbibquote{#1}} % Virgolette per i titoli dei libri

\renewbibmacro*{finentry}{%
	\setunit{\adddot\space}% % Aggiungi il punto prima del backref
	\iflistundef{pageref}
	{}
	{\printlist[pageref][-\value{listtotal}]{pageref}}}
	
\renewbibmacro*{pageref}{%
	\iflistundef{pageref}
	{}
	{}}

\DeclareFieldFormat{pages}{#1} % Stampa solo il numero di pagina senza "pp."
\renewbibmacro{in:}{%
	\ifentrytype{article}{}{\printtext{\bibstring{in}\intitlepunct}}}

\DeclareFieldFormat[article]{volume}{#1} % Solo il numero di volume
\DeclareFieldFormat[article]{number}{\mkbibparens{#1}} % Numero tra parentesi
\DeclareFieldFormat[article]{year}{\mkbibparens{#1}} % Anno tra parentesi

\renewbibmacro*{journal+issuetitle}{%
	\usebibmacro{journal}%
	\setunit*{\addspace}%
	\printfield{volume}%
	\setunit{\addspace}%
	\printfield{year}%
	\setunit{\addcomma\addspace}%
	\printfield{pages}%
	\newunit}

\DeclareFieldFormat{labelnumberwidth}{\mkbibbrackets{#1}} % Rende i numeri delle citazioni tra parentesi
\DeclareFieldFormat{labelnumber}{#1} % Formatta il numero della citazione senza modifiche
\renewbibmacro*{note+pages}{%
	\printfield{note}%
	\setunit{\bibpagespunct}}
\AtBeginBibliography{}
\title[Quantitative and exact concavity principles for parabolic and elliptic equations]
{Quantitative and exact concavity principles \\ for parabolic and elliptic equations}
\author[M. Gallo]{Marco Gallo \orcidlink{0000-0002-3141-9598}}
\author[R. Moraschi]{Riccardo Moraschi \orcidlink{0009-0001-6049-5362}}
\author[M. Squassina]{Marco Squassina \orcidlink{0000-0003-0858-4648}}

\address[M. Gallo, M. Squassina]{\newline\indent Dipartimento di Matematica e Fisica
	\newline\indent
	Università Cattolica del Sacro Cuore
	\newline\indent
	Italy, Brescia, BS, Via della Garzetta 48, 25133}
\email{\href{mailto:marco.gallo1@unicatt.it}{marco.gallo1@unicatt.it}}
\email{\href{mailto:marco.squassina@unicatt.it}{marco.squassina@unicatt.it}}

\address[R. Moraschi]{\newline\indent Dipartimento di Matematica
	\newline\indent
Università degli Studi di Pavia
	\newline\indent
	Italy, Pavia, PV, Via Ferrata 1, 27100}
\email{\href{mailto:r.moraschi1@campus.unimib.it}{r.moraschi1@campus.unimib.it}}

\thanks{\emph{Fundings.} 
All the authors are members of INdAM-GNAMPA.
M. Gallo is supported by INdAM-GNAMPA Projects
``Metodi variazionali per problemi dipendenti da operatori frazionari isotropi e anisotropi'', codice CUP \#E5324001950001\#, and 
``Aspetti geometrici e qualitativi di equazioni ellitiche e paraboliche'', codice CUP \#E53C25002010001\#.
R. Moraschi was partially funded by PRIN 2022 project 2022R537CS $NO^3$, CUP F53D23002810006, granted by the MUR and funded by the European Union - Next Generation EU%
\bigskip
}

\subjclass[2020]{%
26B25, %Convexity of real functions of several variables, generalizations
35B09, %Positive solutions to PDEs
35B25, %Perturbations in context of PDEs
35B99, %Qualitative properties of solutions to partial differential equations (None of the above, but in this section)
35E10, %Convexity properties of solutions to PDEs and systems of PDEs with constant coefficients
35J60, %Nonlinear elliptic equations
35K20, %Initial-boundary value problems for second-order parabolic equations
35K91%Semilinear parabolic equations with Laplacian, bi-Laplacian or poly-Laplacian
}

\begin{document}

\begin{abstract}%ACV: max 250 parole. Queste sono circa 137.
Goal of this paper is to study classes of Cauchy-Dirichlet problems which include parabolic equations of the type
$$u_t -\Delta u= a(x,t)f(u)\quad\hbox{in $\Omega\times(0,T)$}$$
with $\Omega\subset\mathbb{R}^n$ bounded, convex domain and $T\in(0,+\infty]$. 
Under suitable assumptions on $a$ and $f$, we show logarithmic or power concavity (in space, or in space-time) of the solution $u$; under some relaxed assumptions on $a$, we show moreover that $u$ enjoys concavity properties up to a controlled error. 
The results include relevant examples like the torsion $f(u)=1$, the Lane-Emden equation $f(u)=u^q$, $q\in(0,1)$, the eigenfunction $f(u)=u$, the logarithmic equation $f(u)=u\log(u^2)$, and the saturable nonlinearity $f(u)=\frac{u^2}{1+u}$. 
The logistic equation $f(x,u)=a(x)u-u^2$ can be treated as well.

Some exact results give a different approach, as well as generalizations, to \cite{IsSa13, IsSa16}. 
Moreover, some quantitative results are valid also in the elliptic framework $-\Delta u=a(x)f(u)$ and refine \cite{BuSq19, GaSq24}.

%Some exact results provide a different approach to the existing literature, as well as several generalizations. Moreover, some quantitative results are also valid and new in the elliptic setting‌  $-\Delta u=a(x)f(u)$.(Hyers–Ulam, [35.35])
\end{abstract}

\maketitle

\vspace{-1em}

\begin{center}
\begin{minipage}{10cm}
		\small
		\tableofcontents
\end{minipage}
\end{center}

\vspace{1em}

%%%%%%%%%%%%%%%%%%%%%%%%%%%%%%%%%%%%%%%%%%%%%%%%%%%%%%%%%%%%%%%%%%%%%%%%
%%%%%%%%%%%%%%%%%%%%%%%%%%%%%%%%%%%%%%%%%%%%%%%%%%%%%%%%%%%%%%%%%%%%%%%%
\section{Introduction}

In the 1970s and 1980s, significant progress was made in understanding the convexity and concavity properties of solutions to certain elliptic differential equations on convex domains and Dirichlet boundary conditions. In the pivotal papers by Brascamp and Lieb \cite{BrascampLieb} and Makar-Limanov \cite{MakarLimanov}, the authors showed that solutions of the eigenfunction problem and of the torsion problem are concave up to a composition, respectively, with the logarithm function or the square root function. The seminal papers by Korevaar \cite{Korevaar} and Kennington \cite{kennington} gave a new boost to this research, by introducing general \emph{concavity maximum principles} to deal with general classes of equations. We refer to \cite{AAGS2023,GaSq24} and references therein for a general overview on the topic. 
Let us mention, among various mathematical applications, that in \cite{SWYY85} the authors exploit the $\log$-concavity of the eigenfunction to show notable results about the eigenvalue gap, i.e. the difference between the first and the second Dirichlet eigenvalue.

\smallskip

In the framework of parabolic equations and Cauchy-Dirichlet problems
$$
\begin{cases}
 u_t - \Delta u = b(x,u,t) & \quad \hbox{in $\Omega \times (0,T)$},\\
u=0 & \quad \hbox{on $\partial \Omega \times [0,T)$}, \\
u=u_0& \quad \hbox{on $\Omegabar \times \{0\}$},
\end{cases}
$$
$T \in (0, +\infty]$, several questions arise. 
A first problem is to study very particular examples of $b$ and discuss if, under suitable assumptions on $u_0$, the solution $u(x,t)$ becomes logarithmic or power concave (i.e. $\log(u)$ or $u^\alpha$ for some $\alpha>0$ is concave) after some amount of time ($t>0$ or possibly $t\gg 0$): these kinds of results are strongly related to the particular shape of $b$ which allows, for instance, explicit representation of the solution, and have been investigated for example in \cite{LeeVazquez, lee2008parabolic, Ish24}. 
We highlight that these results allow, under suitable assumptions, to get interesting properties also in the elliptic case, by considering the limit of $u(x,t)$ as $t\to +\infty$ which %eventually 
coincides with the solution of the corresponding stationary problem.

A second question is the following: what is the optimal (in some sense) concavity on $u_0$ which is preserved by the heat flow? In this problem, the logarithmic concavity plays a very special role. 
This and similar questions have been addressed in a series of papers \cite{IshigeSalaniTakatsu2024, IshigeSalaniHeatFlow, IsSaTa22, IsSaTa20}. 
See also \cite{Daskalopoulos-Hamilton-Lee2001, ChauWeinkove} for similar questions regarding power concavity and porous medium flow. 

In this paper we are interested in a different question: under which assumptions on $b$ and $u_0$ can we obtain that $u(x,t)$ enjoys some concavity properties for each $t>0$?
To the authors' knowledge, only few papers deal with this problem. 

A first result can be found in Korevaar \cite{Korevaar}, where the author adapts %authors adapt 
the elliptic concavity maximum principles to study classes of parabolic equations and show that, for example, the parabolic eigenfunctions $u(\cdot, t)$ are spatially logarithmic concave for each $t \in [0,T]$ if $u_0$ is so. 
Let us mention that in \cite{AnCl11} a stronger notion of $\log$-concavity is proved and it is exploited to prove a sharp result on the eigenvalue gap; it is notable to observe that the technique adopted in \cite{AnCl11} relies on the introduction and analysis of some suitable multivariable function in a spirit relatable to the one introduced in \cite{Korevaar, kennington} and used in the present paper.
Some partial generalizations of the results in \cite{Korevaar} can be found in \cite{PorruSerra} (see also \cite{GrecoKawohl}), by allowing the inclusion of other interesting models. 
Other results are contained in \cite{IshigeSalaniTakatsu2024} (see e.g. Proposition 4.2 therein) %in there) 
by means of \emph{spatial log-concave envelopes} as viscosity subsolutions of the problem and corresponding comparison principles; similar techniques go back to \cite{AlvarezLasryLions1997} (see also \cite{ColesantiSalani2003}).

A second result, i.e. a parabolic generalization for equations in the spirit of the torsion and Lane-Emden problems, which allows to obtain power concavity (stronger than logarithmic one), appears first in \cite{kenningtonparabolic}: here the author studies the case $a=1$ and $f(u)=(1-u)^q$, $q \in (0,1)$, obtaining that a suitable transformation of $u$ is concave and thus $u$ is quasi-concave (i.e., the superlevel sets of $u$ are convex). 
To gain the claim the author exploits again some concavity maximum principle and he strongly uses that $f(0) \neq 0$ coupled with some comparison principle; moreover, he relies on the \emph{information at infinity} given by the stationary problem. 
Other results have been more recently obtained in \cite{IsSa13, IsSa16} (see also \cite{ILS20}):
here the authors rely on the use of \emph{parabolic concave envelopes} as viscosity subsolutions.
In such a way the authors are able to study the concavity of nondecreasing solutions with controlled behaviour on the boundary, in the case of nonnegative nonlinearities satisfying a priori some comparison principle, set in smooth domains: for example, they cover the cases $f=1$, or $a=1$ and $f(u)=u^q$, $q \in (0,1)$. Moreover, the authors study a \emph{rescaled} concavity in space and time, namely they show that $(x,t) \mapsto u(x,t^{\beta})$ is power concave for suitable values of $\beta$.

\smallskip

Aim of this paper is twofold. 
First, we aim to recover and slightly generalize the abovementioned results, in a unified way, by means of parabolic concavity maximum principles. In particular, to recover the results previously obtained through different techniques, we need a finer analysis, especially regarding the boundary behaviour of the solution (including $t \to +\infty$). 
A second goal is to derive \emph{quantitative} estimates of concavity maximum principles, which will require a careful study of the objects in play. 
For the sake of clarity in presentation, we describe first some applications of the exact results (which are indeed a consequence of the quantitative estimates).

\smallskip

%To fix notations, we say that $u$ is \emph{$\alpha$-concave} if $\alpha>0$ and $u^{\alpha}$ is concave, or $\alpha=0$ and $\log(u)$ is concave (see \eqref{eq_def_alpha_concav}). Moreover, a solution of a PDE is meant \textit{positive} when $u>0$ in $\Omega\times (0,T]$.

\smallskip
 \textbf{Notation.} 
We say that $u$ is \emph{$\alpha$-concave} if $\alpha>0$ and $u^{\alpha}$ is concave, or $\alpha=0$ and $\log(u)$ is concave (see \eqref{eq_def_alpha_concav}). Moreover, a solution of a PDE is meant \textit{positive} when $u>0$ in $\Omega\times (0,T]$.
For the sake of simplicity, to say that $u\colon \Omegabar \times [0,T] \to \R$ verifies $u(\cdot, t) \in C^k(\Omegabar)$ for each $t \in (0,T]$ (or $t>0$ if $T=+\infty$) and $u(x, \star) \in C^h([0,T])$ for each $x \in \Omega$, we will simply write %-- with a slight abuse of notation -- 
$u \in C^k_x(\Omegabar) \cap C^h_t([0,T])$. 
Finally in what follows we use the convention, for the sake of simplicity, $1/\gamma=+\infty$ when $\gamma=0$.

We refer to the beginning of Section \ref{sec_prelimin} for other notations.

\smallskip
%In this paper we recover these results, obtaining among others the following \tr{statements}. 
%
%
%\kt{Specifically, our approach allows us to recover the classical results established in \cite{Korevaar}, extending the analysis through the inclusion of a variable term $a(x,t)$, see Theorem \ref{thm_general_korevaar} and \ref{thm_populat} . Since the introduction of this term necessitates specific concavity requirements, a core objective of our study is also to quantitatively estimate the "lack of concavity" of $\log(u)$ in terms of the corresponding lack of concavity of the coefficient $a(x,t)$ (see Theorem \ref{approximate concavity parabolic application a(x)u introduction} below). To achieve this, we incorporate novel quantitative estimates into the general framework originally developed in \cite{Korevaar}.}
%
%
%\tor{M: Specificare cosa è nuovo rispetto alla letteratura precedente, e in che senso il risultato è nuovo (almeno per l'autofunzione e per $s \log(s)$). Se il risultato non è nuovo, bisogna specificare che è semplicemente ottenuto con metodi diversi; se questo è il caso, bisogna evidenziare quali sono i miglioramenti provenienti dalla nostra metodologia [ad esempio, che magari permettono un'analisi quantitativa]}

We start with some examples which also cover results in \cite{Korevaar, PorruSerra, GrecoKawohl} and \cite{IshigeSalaniTakatsu2024}.

\begin{theorem}[Weighted eigenfunctions and logarithmic equations]
\label{thm_general_korevaar}
	Let $\Omega \subset \R^n$, $n\geq 2$, be a bounded, smooth, strongly convex domain. 
Let $T>0$, $a\colon\Omega\times(0,T]\to \R$, $u_0\in C^1(\Omegabar)$ 
with $u_0=0$ on $\partial\Omega$ and $u \in C^2_x(\Omegabar) \cap C^1(\Omegabar\times[0,T])$ be a positive solution of$$		
\begin{cases}
u_t - \Delta u = a(x,t) f(u) &\quad \hbox{in $\Omega \times (0,T]$}, \\
u=0 & \quad\text{on }\partial \Omega\times [0,T],	 \\
D_x u \cdot \nu >0 & \quad\text{on }\partial \Omega\times [0,T], \\
u =u_0 & \quad \text{on } \Omegabar \times \{0\},
\end{cases}
$$
where $\nu$ is the interior normal to $\partial \Omega$. 
Assume moreover that $u_0>0$ is $\log$-concave in $\Omega$ and one of the following:
\begin{enumerate}[label=\textit{\roman*})]
\item $f(s)=s$, and $a(\cdot, t)$ is concave for each $t \in (0,T]$; %\tor{M: giusto? ricontrollare}; \kt{R: ok}
\item $f(s)=s \log^q(1+s)$ or $f(s)=\frac{s^{q+1}}{1+s^q}$, $q \in (0,1)$, and $a(\cdot,t)$ is $\frac{1}{1-q}$-concave for each $t \in (0,T]$; %\tor{M: giusto? ricontrollare};\kt{R: ok}
 \item $f(s)=s \log(1+s)$, $f(s)=s \log(s)$ or $f(s)=\frac{s^2}{1+s}$, and $a$ is a positive constant. 
\end{enumerate}
Then $\log(u(\cdot, t))$ is concave for each $t \in (0,T)$. Moreover, if $a \in C^{2,1}(\Omega\times(0,T])$ and $u \in C^{2,1}(\Omega\times(0,T])$, then $\log(u(\cdot, t))$ is strictly concave for each $t \in (0,T)$; as a consequence, for each $t \in (0,T)$ the function $u(\cdot, t)$ has a unique, nondegenerate, critical point.
\end{theorem}

 Theorem \ref{thm_general_korevaar} is a particular case of the more general Corollary \ref{corol_concav_strict}; we highlight that the method allows a dependence of $a$ %in 
 both %the dependence of $a$ 
in $x$ and $t$, and to choose %and the fact that we do not require 
$a$ sign-changing %to be 
%nonnegative 
in point (\textit{i}) (which allows to cover a broader range of applicative problems, see also Theorem \ref{thm_populat} below). 
Moreover, we highlight the information on strict concavity that we obtain.
A generalization of Theorem \ref{thm_general_korevaar} in the sense of perturbed concavity is presented in Theorem \ref{approximate concavity parabolic application a(x)u introduction} below and Corollary \ref{cor_korevaar_perturbed}.
We refer to Remark \ref{rem_Thm_kor} for some technical comments on Theorem \ref{thm_general_korevaar}.

As an additional example, we may consider logistic models of population dynamics in heterogeneous environments with lethal boundary. 
That is, the equation is governed by a logistic law $b(x,s) = a(x) s - s^2$ where $\{a>0\}$ is the favorable habitat while $\{a<0\}$ is the hostile one, see \cite{CaCo89, BeCo95, OSS02} (see also Remark \ref{rem_bangbang} for bang-bang weights). 

%\tor{Specificare cosa è nuovo rispetto alla letteratura precedente, e in che senso il risultato è nuovo. Se il risultato non è nuovo, bisogna specificare che è semplicemente ottenuto con metodi diversi; se questo è il caso, bisogna evidenziare quali sono i miglioramenti provenienti dalla nostra metodologia [ad esempio, che magari permettono un'analisi quantitativa]}
%\kt{R: se lasciamo il pezzo prima del teorema 1.1, qui non serve specificare altro. Potremmo solo dire che korevaar non si era accorto di questa applicazione, ma mi sembra eccessivo.}
\begin{theorem}[Logistic model]
\label{thm_populat}
	Let $\Omega \subset \R^n$, $n\geq 2$, be a bounded, smooth, strongly convex domain. 
Let $T>0$, $a\colon\Omega\times(0,T]\to \R$, %} \tor{M: controllare, $a$ può dipendere dal tempo o meno?}, 
$u_0\in C^1(\Omegabar)$ with $u_0=0$ on $\partial\Omega$ and $u \in C^2_x(\Omegabar) \cap C^1(\Omegabar\times[0,T])$ be a positive solution of
\begin{equation*}
\begin{cases}
u_t - \Delta u = a(x,t) u - u^2 &\quad \hbox{in $\Omega \times (0,T]$}, \\
u=0 & \quad\text{on }\partial \Omega\times [0,T],	 \\
D_x u \cdot \nu >0 & \quad\text{on }\partial \Omega\times [0,T], \\
u =u_0 & \quad \text{on } \Omegabar \times \{0\},
\end{cases}
\end{equation*}
where $\nu$ is the interior normal to $\partial \Omega$. 
Assume moreover that $u_0>0$ is $\log$-concave in $\Omega$ and $a(\cdot,t)$ is concave in $\Omega$ for each $t\in (0,T]$.
Then $\log(u(\cdot, t))$ is concave for each $t \in (0,T)$. Moreover, if $a\in C^{2,1}(\Omega\times(0,T])$ and $u \in C^{2,1}(\Omega \times (0,T])$, then $\log(u(\cdot, t))$ is strictly concave for each $t \in (0,T)$.
\end{theorem}

Compared with \cite{Korevaar, PorruSerra, GrecoKawohl}, Theorems \ref{thm_general_korevaar} and \ref{thm_populat} point out some additional models and contain some minor generalizations (e.g. the presence of the weight $a(x,t)$), which are not substantial but allow us to give a more complete picture.
We continue by presenting some examples, inspired by \cite{kenningtonparabolic}, and which cover models presented in \cite{IsSa13, IsSa16}.

\begin{theorem}[Weighted Lane-Emden]
\label{thm_main_weighted_le}
	Let $\Omega \subset \R^n$, $n\geq 2$, be a bounded, convex domain which satisfies the interior sphere property.
	 Let $q \in [0,1)$ and $\gamma \in [0,1]$. 
	Let $u\in C^2_x(\Omega)\cap C^1_x(\Omegabar) \cap C^1_t((0,+\infty)) \cap C(\Omegabar\times[0,+\infty))$ be a positive solution of 
$$		
\begin{cases}
u_t-\Delta u=a(x) t^{\gamma} u^q &\quad\text{in }\Omega\times (0,+\infty), \\
u=0 & \quad\text{on }\partial \Omega\times [0,+\infty),	 \\
u =0 & \quad \text{on } \Omegabar \times \{0\}.
\end{cases}
$$
Suppose $a\colon \Omega \to \R$ to be measurable and that there exists $m>0$ such that $a(x)\geq m$ for all $x\in \Omega$. 
If $\beta \in [1, \frac{1}{\gamma}) \cap [1,2]$ and $a$ is $\theta$-concave with $\theta \geq \frac{1}{1-\beta \gamma}$, then $u(\cdot, \star^{\beta})$ is $\frac{(1-q)\theta}{2\theta +\beta \gamma \theta+1}$-concave. 
If $\beta \in [1, \min\{\frac{1}{\gamma},2\}]$ and $a$ is constant, then $u(\cdot, \star^{\beta})$ is $\frac{1-q}{2+\beta \gamma }$-concave.
\end{theorem}

We notice that, if $q=0$, the condition $a(x) \geq m>0$ above can be relaxed into $a(x) \geq 0$, see Proposition \ref{prop_weight_tors}; we highlight in Theorem \ref{thm_main_weighted_le} the possibility of the nonlinearity to vanish where $u=0$, which is a nontrivial improvement compared to \cite{kenningtonparabolic}.
Theorem \ref{thm_main_weighted_le} is a particular case of the more general Corollary \ref{concavity parabolic with boundary conditions}.
We highlight that, whereas \cite[Theorem 3]{IsSa13} and \cite[Theorem 4.1]{IsSa16} rely on a priori assumptions on the solution $u$, in Theorem \ref{approximate concavity parabolic with boundary conditions} and Corollary \ref{concavity parabolic with boundary conditions} we focus instead on %our Theorem \ref{approximate concavity parabolic with boundary conditions} 
furnishing operative conditions to check on the nonlinearity. %imposes conditions solely on the nonlinearity
Notice that Theorem \ref{thm_main_weighted_le} somehow unifies \cite[Theorem 5 and Theorem 6]{IsSa13} (see also \cite[Theorem 5.3 and Corollary 6.1]{IsSa16}), even if such generalization is not %directly 
deducible from the arguments in \cite{IsSa13,IsSa16}.

A quantitative version of Theorem \ref{thm_main_weighted_le} can be found in Theorem \ref{thm_approximate concavity parabolic application a(x)u^gamma} below and Corollary \ref{corollario con concavità}. 
We refer to Remark \ref{rem_techn_kennin} for more technical comments on Theorem \ref{thm_main_weighted_le}.

%\tr{We actually recover also the result in \cite{kenningtonparabolic}, see Proposition \ref{prop_kenn88} and Remark \ref{rem_Ken88}. }

%Our approach allows to recover also the quasiconcavity obtained by \cite{kenningtonparabolic}, see Proposition \ref{prop_kenn88} and the subsequent remark. 
To mention another consequence of our general setting, we propose here the concavity of solutions for nonhomogeneous nonlinearities. 
Let us mention that our approach allows to recover also the quasiconcavity obtained by \cite{kenningtonparabolic} (see Proposition \ref{prop_kenn88} and the subsequent remark), and to treat gradient-type equations (see e.g. Remark \ref{rem_gradient_type}).

\begin{theorem}[Sum of powers]
\label{thm_sum_powers}
 	Let $\Omega \subset \R^n$, $n\geq 2$, be a bounded, convex domain which satisfies the interior sphere property.
	 Let $q\in (\frac{1}{3},1)$ and $p\in (\frac{3q-1}{2},q)$.
	Let $u\in C^2_x(\Omega)\cap C^1_x(\Omegabar) \cap C^1_t((0,+\infty)) \cap C(\Omegabar\times [0,+\infty))$ be a positive solution of 
	$$		
\begin{cases}
u_t-\Delta u= a(x) u^p+u^q &\quad\text{in }\Omega\times (0,+\infty), \\
u=0 & \quad\text{on }\partial \Omega\times [0,+\infty),	 \\
u =0 & \quad \text{on } \Omegabar \times \{0\}.
\end{cases}
$$
Suppose $a\colon \Omega \to \R$ to be measurable and that there exists $m>0$ such that $a(x)\geq m$ for all $x\in \Omega$; assume moreover $a$ to be $\theta$-concave with $\theta \in [\frac{1-q}{2(q-p)}, +\infty]$.
Then we have that $u(\cdot, \star^{2})$ is $\frac{1-q}{2}$-concave. 
\end{theorem}

We postpone the discussion of the technical difficulties of Theorems \ref{thm_general_korevaar}--\ref{thm_sum_powers} (and of the corresponding generalizations Corollaries \ref{corol_concav_strict} and \ref{concavity parabolic with boundary conditions}), see Remark \ref{rem_proof_techni}.

\smallskip

As aforementioned, the second and main goal of this paper is to deal with quantitative versions of the concavity maximum principles. Recently, quantitative versions of symmetry results have been considered for example in \cite{Ciraolo-Cozzi-Perugini-Pollastro, Gat25, ABMP24, DSPV24};
we mention also \cite{CiraoloFigalliMaggi2018, DengSunWei2025} where the error from bubble functions is quantified in Sobolev critical equations.
In the realm of concavity properties, instead, similar results have been recently investigated in the elliptic case in \cite{BuSq19} (see also \cite{ABCSnonautonomousellipticequation, GaSq24}); such estimates are made in terms of some parameter of the weight $a(x)$: more precisely, in these papers the authors show that, if $a$ is close to a constant function, then some power of $u$ is close to a concave function, up to some explicit error.
We mention also the recent \cite{HamelNadirashvili2025} where a concavity %$\log$-concavity 
result is obtained for problems with small perturbations, %small perturbations of the eigenfunction problem, 
by exploiting the strong $\log$-concavity of the eigenfunction.% itself.

%\smallskip

In this paper, we deal with estimations of the \emph{(space) concavity function} of a solution $v$ (see \eqref{defn concavity function})
$$ \mathcal{C}^*_v(x_1,x_3,\lambda, t) :=v(x_2, t)-\lambda v(x_3, t)-(1-\lambda)v(x_1,t), $$
or the \emph{(spacetime) concavity function}
$$\mathcal{C}_v(x_1,x_3,t_1,t_3,\lambda):= v(x_2,t_2)-\lambda v(x_3,t_3)-(1-\lambda)v(x_1,t_1),$$
where $x_1, x_3 \in \Omega $, $t_1, t_3 \in [0,+\infty)$, $\lambda \in [0,1]$ and $x_2=\lambda x_3+(1-\lambda)x_1$ and $t_2=\lambda t_3+(1-\lambda)t_1$; notice that $\mc{C}_v \geq 0$ if and only if $v$ is concave. 
Our goal is, as already anticipated, to generalize the abovementioned results also to the parabolic setting, which seems to be new in the literature.
As a consequence of our general results, we get in particular the following statement. Here $f^- := \max\{0,-f\}$ stands for the negative part of a function $f$, while $\osc(a):= \sup(a) - \inf(a)$ stands for the oscillation of $a$.

\begin{theorem}[Quantitative concavity, I]\label{approximate concavity parabolic application a(x)u introduction}
Let $\Omega \subset \R^n$, $n\geq 2$, be a bounded, smooth, strongly convex domain. 
Let $T>0$, $a\colon \Omega \times (0,T] \to \R$ bounded,
%\tr{locally bounded} \tor{M: serve richiedere bounded?}\kt{R: mi sembra lo avessimo messo per sicurezza. A posteriori, le stime di regolarità sulla u richiedendo qualcosa su $a$. Riguardando tutto potremmo capire se si può togliere ma nel dubbio possiamo lasciarlo},
 $u_0\in C^1(\Omegabar)$ with $u_0=0$ on $\partial\Omega$ and $u \in C^2_x(\Omega) \cap C^1(\Omegabar\times[0,T])$ be a positive solution of
$$		
\begin{cases}
u_t - \Delta u = a(x,t) u &\quad \hbox{in $\Omega \times (0,T]$}, \\
u=0 & \quad\text{on }\partial \Omega\times [0,T],	 \\
D_x u \cdot \nu >0 & \quad\text{on }\partial \Omega\times [0,T], \\
u =u_0 & \quad \text{on } \Omegabar \times \{0\},
\end{cases}
$$
where $\nu$ is the interior normal to $\partial \Omega$. 
Assume moreover that $u_0>0$ is $\log$-concave in $\Omega$. 
 Then
$$\inf_{\Omega \times \Omega \times [0,1] \times [0,T]} \mc{C}^*_{\log(u)} \geq -e T \sup_{\Omega \times \Omega \times [0,1] \times (0,T]} \left(\mc{C}^*_{a(\cdot, \star)}\right)^-.$$
\end{theorem}

In Theorem \ref{approximate concavity parabolic application a(x)u introduction} one may consider also other nonlinearities, like the ones in Theorem \ref{thm_general_korevaar} (see Corollary \ref{cor_korevaar_perturbed}).
Moreover, if $u_0$ is not assumed to be $\log$-concave, then an additional error related to $\mc{C}_{\log(u_0)}$ appears (see Remark \ref{rem_u0_nonconcav}).

\begin{theorem}[Quantitative concavity, II]
\label{thm_approximate concavity parabolic application a(x)u^gamma}
Let $\Omega \subset \R^n$, $n\geq 2$, be a bounded, convex domain which satisfies the interior sphere property.
 Let $q \in [0,1)$. 
Let $u\in C^2_x(\Omega)\cap C^1_x(\Omegabar) \cap C^1_t((0,+\infty)) \cap C(\Omegabar\times [0,+\infty))$ be a positive solution of 
$$		
\begin{cases}
u_t-\Delta u=a(x) u^q &\quad\text{in }\Omega\times (0,+\infty), \\
u=0 & \quad\text{on }\partial \Omega\times [0,+\infty),	 \\
u =0 & \quad \text{on } \Omegabar \times \{0\}.
\end{cases}
$$
Suppose $a\colon\Omega \to \R$ to be measurable and $0<m \leq a \leq M$.
 Then $u$ is increasing in time and 
\begin{equation}\label{eq_quantitativ_oscil}
\min_{\Omegabar\times \Omegabar\times[0,+\infty]\times [0,+\infty]\times[0,1]}\mathcal{C}_{u^{\frac{1-q}{2}}} \geq 
- \norm{u(\cdot, \infty)}_{L^{\infty}(\Omega)}^{\frac{1-q}{2}} \frac{\osc(a^2)}{m^2}, 
\end{equation}
where $u(\cdot, \infty)$ is the solution of the corresponding stationary problem.

If moreover $\theta \geq 1$ and $m^\theta\geq M^{\theta}/2$, then
\begin{equation}\label{eq_quantitative_theta}
	\min_{\Omegabar\times \Omegabar\times[0,+\infty]\times [0,+\infty]\times[0,1]}\mathcal{C}_{u^{\frac{\theta(1-q)}{2\theta+1}}}
\geq
%-\frac{2\theta }{2\theta+1} \frac{1}{m} \norm{u(\cdot, \infty)}_{L^{\infty}(\Omega)}^{\frac{(\theta-1)(1-q)}{2\theta+1}}
%\sup\limits_{\Omega \times \Omega\times[0,1]}
% (\mathcal{C}^-_{a^{\theta}})^{1/\theta}.
-\frac{\theta }{\theta+1} \frac{1}{m} \norm{u(\cdot, \infty)}_{L^{\infty}(\Omega)}^{\frac{\theta(1-q)}{2\theta+1}}
\sup\limits_{\Omega \times \Omega\times[0,1]}
 (\mathcal{C}^-_{a^{\theta}})^{1/\theta}.
\end{equation}
\end{theorem}

We refer to Remark \ref{rmk quantitative concavity introduction} for some additional comments.
We highlight that estimate \eqref{eq_quantitativ_oscil} in Theorem \ref{thm_approximate concavity parabolic application a(x)u^gamma} is slightly finer if compared with the ones obtained in the elliptic setting in \cite{BuSq19, GaSq24}. 
Moreover, the estimate \eqref{eq_quantitative_theta} does not involve the oscillation of $a$ and is totally new, even for the elliptic setting. %\tr{; see Theorem \ref{thm_elliptic_case_intr} below.}
%As just mentioned in Remark \ref{rmk quantitative concavity introduction}, 
Indeed, some results of this paper can be easily adapted to the elliptic framework, allowing to improve and generalize some previous results \cite{BuSq19, GaSq24}. 
By way of example, with the same techniques, we obtain the following statement. We leave the details to the interested reader.

\begin{theorem}[$\theta$-concavity estimate]	
\label{thm_elliptic_case_intr}
Let $\Omega \subset \R^n$, $n\geq 2$, be a bounded, convex domain which satisfies the interior sphere property, and let $q \in [0,1)$. Let $u\in C^2(\Omega)\cap C^1(\Omegabar)$ be a positive solution of
$$\begin{cases}
-\Delta u = a(x) u^q & \quad \hbox{in $\Omega$}, \\
u=0 & \quad \hbox{on $\partial \Omega$}.
\end{cases}
$$
Suppose $a\colon \Omega \to \R$ to be measurable and $0<m\leq a \leq M$ and $1\leq\theta \leq\log(2)/\log(\frac{M}{m})$. 
 Then
\[	\min_{\Omegabar\times \Omegabar\times[0,1]}\mathcal{C}_{u^{\frac{\theta(1-q)}{1+2\theta}}}
\geq
%-\frac{2\theta}{1+2\theta} \frac{1}{m}||u||_\infty^{\frac{(\theta-1)(1-q)}{1+2\theta}}
%\sup\limits_{\Omega \times \Omega\times[0,1]}
%(\mathcal{C}^-_{a^{\theta}})^{1/\theta}.
-\frac{\theta}{1+\theta} \frac{1}{m}||u||_\infty^{\frac{\theta(1-q)}{1+2\theta}}
\sup\limits_{\Omega \times \Omega\times[0,1]}
(\mathcal{C}^-_{a^{\theta}})^{1/\theta}.
\]
\end{theorem}

\smallskip

As already mentioned, the main tool of this paper is composed by \emph{concavity} and \emph{harmonic concavity maximum principles}, which allow to give an estimate on the concavity function $\mc{C}_v$ of a solution of a general PDE in terms of the concavity/harmonic concavity function of the nonlinear term (see Definition \ref{concavity function definition}). 
The equations treated do not rely on the boundary conditions and can have a very general form (see Theorems \ref{thm_conc_princ_eigenf}, \ref{Concavity perturbative maximum principle for parabolic equations}), and in particular include the equations satisfied by $v=\log(u)$ or $v=u^{\alpha}$, $\alpha \in (0,1)$, where $u$ is the solution of the equation of interest.
%
%To further treat equations like the logarithmic one, we need to implement a correction in time of the type $e^{\mu t}$, $\mu$ large, which allows to \tb{decrease the monotonicity \tr{in $u$} of the nonlinearity involved\tr{, recovering the nonincreasing monotonicity which is needed for the argument}} (see \tr{\eqref{eq_drift_trick}, \eqref{eq_ipos_sigma_trans} and \eqref{eq_deriv_eq_B_Btilde_2}} in Theorem \ref{thm_conc_princ_eigenf}). %\tor{M: Spieghiamo ancora meglio?}. \kt{R: no va bene}
%
We highlight that the concavity principle approach is flexible since it can be used, for example, to deal with \emph{weak} solutions of quasilinear equations, see \cite{GaSq24, BMSplaplacian, Lin1994} (see also Remark \ref{rem_bangbang} for another nonsmooth setting), and thus possibly useful for treating $p$-Laplace parabolic equations.

We refer to Remark \ref{rem_proof_techni} for some technical comments on the proofs and corresponding difficulties and novelties.

\medskip

The paper is organized as follows. 
In Section \ref{sec_remarks} we collect some remarks on the main theorems and on the proofs, while in Section \ref{sec_prelimin} we recall some basic facts on parabolic equations, and we introduce the main assumptions that we will require on the nonlinearity. 
Moreover, we recall notations and properties about concavity functions; more technical results are instead collected in Appendix \ref{sec_prop_conc_fun}.
Section \ref{sec_concav_max_princip} is devoted to the proof of the main concavity principles, which take into account only the PDE and not the boundary/initial conditions; these theorems will lead both to exact and quantitative results. 
In particular, we show both concavity principles in space (in presence of nonlinearities which are concave but not necessarily decreasing, and set in $(0,T]$) and in space-time (in the case of harmonic concave and decreasing nonlinearities set in $(0,+\infty)$); the latter will require the use of the harmonic concavity function.
In Section \ref{sec_gen_eigen} we take into account the boundary conditions and the initial condition, employing a suitable analysis of the solution and of the concavity function near the boundary of the parabolic cylinder; this analysis exploits also comparison principles and Hopf lemmas which are collected in Appendix \ref{sec_max_priniciples}. Some applications and examples follow, including the proofs of the main theorems presented in the introduction.
When the nonlinearity does not depend on $u$, finally, we present some further results (both simplifications and improvements of the previous ones) in Section \ref{sec_torsion}.

%%%%%%%%%%%%%%%%%%%%%%%%%%%%%%%%%%%%%%%%%%%%%%%%%%%%%%%%%%%%%%%%%%%%%%%%
%%%%%%%%%%%%%%%%%%%%%%%%%%%%%%%%%%%%%%%%%%%%%%%%%%%%%%%%%%%%%%%%%%%%%%%%
\section{Some remarks on the theorems}
\label{sec_remarks}

We collect in this section some technical comments on the main theorems and on the proofs and difficulties of the present paper.
 
\begin{remark}[On Theorems \ref{thm_general_korevaar} and \ref{thm_populat}]
\label{rem_Thm_kor}
We observe the following facts.
\begin{itemize}
%\item Theorem \ref{thm_general_korevaar} is a particular case of the more general Corollary \ref{corol_concav_strict}; \tb{we highlight \tr{that the method allows a dependence of $a$ in both} %the dependence of $a$ 
%in $x$ and $t$.} 
%\tor{Evidenziare: è nuovo rispetto a...? E' tecnicamente più difficile perché...? Non mi pare che crei difficoltà tecniche... e sicuramente se uno utilizza il drift, la dipendenza dal tempo già compare...} \tb{Such generalization is an extension of \cite{korevaar1987convex}.} \kt{R: qui non calcherei troppo la mano in quanto se uno prende i risultati di Korevaar, può ottenere questo teorema. Korevaar non specifica chiaramente la classe generale di equazioni a cui può essere applicata mentre noi si. Non compaiono difficoltà tecniche. In realtà, se lasciamo il pezzo prima del teorema, questo punto non serve}. 
%We also mention that in this paper we actually deduce generalizations of Theorem \ref{thm_general_korevaar} also in the sense of perturbed concavity, see Theorem \ref{approximate concavity parabolic application a(x)u introduction} below and Corollary \ref{cor_korevaar_perturbed}.
%
%\item We highlight that \tb{we do not require $a$ to be nonnegative in point (\textit{i})} \tor{Evidenziare: è nuovo rispetto a...? E' tecnicamente più difficile perché...?}; this allows us to cover a broader range of problems, including applicative ones (see also Theorem \ref{thm_populat} below). \kt{R: qui non specificherei altro. Non ha difficoltà tecniche e non è una generalizzione dato che nessuno aveva messo $a(x,t)$ prima.}
%
\item Once concavity is obtained, the strict concavity of solutions is usually gained, in elliptic PDEs, through the use of a constant rank theorem. 
Parabolic versions of the constant rank theorem have been obtained in a space version in \cite{BianGuan2009, CaffareliGuanMa2007}, and in a spacetime version in \cite{chen2014, ChenHu2013} (see Proposition \ref{prop_const_rank}). 
We mention that, in our case, the conclusion for the weighted eigenfunction (case (\textit{i})) can be obtained in a more straightforward way by assuming $a$ strictly concave in $x$ (and no additional regularity), see Remark \ref{rem_direct_strict_conc}; similar properties appear in \cite{Bitterlin92} in the setting of periodic solutions.

\item The condition $D_x u \cdot \nu >0$ already appears in the papers \cite{GrecoKawohl,Korevaar}; we point out the requirement also in $t=0$ (that is, $u_0 \in C^1(\Omegabar)$ and $D_x u_0 \cdot \nu >0$) which, as highlighted in \cite[Section 6]{GrecoKawohl}, is essential but erroneously neglected in previous papers; in the same section, the authors provide an approximation process in order to avoid this restriction in zero (in particular, allowing $u_0 \in C(\Omegabar)$). 
Moreover, in \cite[Remark 2, page 7]{GrecoKawohl} the authors comment the validity of such condition for general classes of nonlinearities.
For example, assume $a \geq 0$ and $D_x u_0 \cdot \nu >0$. Then $f(s) \geq 0$ is sufficient to ensure this condition thanks to the parabolic Hopf Lemma \ref{Hopf lemma in parabolic case}; this covers almost every case of Theorem \ref{thm_general_korevaar}. 
To treat $f(s)=s \log(s)$, assuming moreover $M \geq  a \geq m >0$, we see that it is sufficient to choose $u_0$ such that $u_0 \geq w$ where $w$ is a positive solution to the stationary problem $-\Delta w = f(w)$ which, by \cite[Theorem 2]{Vazquez2014} verifies $D_x w \cdot \nu >0$; then a %the 
comparison principle 
and $u=w=0$ on $\partial \Omega$ imply $D_x u \cdot \nu \geq D_x w \cdot \nu >0$.

\item The requirement on the smoothness of $\Omega$ (not required in Theorem \ref{thm_main_weighted_le} below) is here related to the control of the concavity function of $\log(u)$ near the boundary, by means of the positivity of the curvatures. 
Anyway this condition can be relaxed through some standard approximation of the domain $\Omega$, see e.g. \cite{BMSplaplacian, GaSq24}.

\item We highlight that the logarithmic equation $f(s)=s \log(s)$, being superlinear and sign-changing, is a problem not easy to tackle in the elliptic framework, see \cite{Gallo-Mosconi-Squassina}; in the parabolic framework, on the other hand, it can be addressed through a suitable drift %trick 
in time (see \eqref{eq_drift_trick}). 
Notice moreover that the result is obtained on $[0,T]$ and, generally, the solutions of this problem do not converge to the one of the stationary problem (see Remark \ref{rem_converg_logarithm}), thus it is not straightforward to obtain the elliptic result from the parabolic one. 
\end{itemize}
\end{remark}

\begin{remark}[On Theorems \ref{thm_main_weighted_le} and \ref{thm_sum_powers}]
\label{rem_techn_kennin}
We observe the following facts.
\begin{itemize}
%\item The theorem is a particular case of the more general Corollary \ref{concavity parabolic with boundary conditions}, where\tr{, differently from %\cite{...} %IsSa13,IsSa16,ILS20 
%%\tor{Scrivi il riferimento},} \kt{
%\cite[Theorem 3]{IsSa13}} we focus on giving explicit sufficient assumptions on $f$ (rather than on $u$) to ensure the result;
%moreover, our method allows to keep track of the error in concavity and deduce some quantitative results, see Theorem \ref{thm_approximate concavity parabolic application a(x)u^gamma} below and Corollary \ref{corollario con concavità}.
\item We see that the exponent $\alpha=\frac{(1-q)\theta}{2\theta +\beta \gamma \theta+1}$ is consistent with the results in \cite{kennington, GaSq24}, where $\alpha=\frac{(1-q)\theta}{2\theta+1}$ if the problem is elliptic (and thus $\gamma=0$), and in \cite{IsSa13, IsSa16}, where $\alpha = \frac{\theta}{2\theta+\beta\gamma \theta+1}$ if the nonlinearity does not depend on $u$ (i.e. $q=0$).
%\item If $q=0$, the condition $a(x) \geq m>0$ can be relaxed into $a(x) \geq 0$, see Proposition \ref{prop_weight_tors}.
\item The initial condition $u(\cdot, 0)=0$ appears also in \cite{IsSa13, kenningtonparabolic, IsSa16}, and we mention that, for such a sublinear problem, nontrivial solutions generally may exist (see Remark \ref{rem_esist_soluz}). 
See also Remark \ref{rem_u0_nonzero} for some comments on the case $u(\cdot, 0)\neq 0$.
\item If $u$ is nondecreasing in time (and this is the case of our setting, see Theorem \ref{approximate concavity parabolic with boundary conditions}), then if $u(\cdot,\star^{\beta_1})$ is $\alpha$-concave then $u(\cdot,\star^{\beta_2})$ is $\alpha$-concave for all $\beta_2\leq \beta_1$, see \cite[Section 2, Property (e)]{IsSa13} (see also \eqref{eq_harmonic_beta}).
\end{itemize}
\end{remark}

\begin{remark}[On Theorems \ref{approximate concavity parabolic application a(x)u introduction} and \ref{thm_approximate concavity parabolic application a(x)u^gamma}]
\label{rmk quantitative concavity introduction}
We observe the following facts.
\begin{itemize}
%\item In Theorem \ref{approximate concavity parabolic application a(x)u introduction} one may consider also other nonlinearities, like the ones in Theorem \ref{thm_general_korevaar} (see Corollary \ref{cor_korevaar_perturbed}).
%Moreover, if $u_0$ is not assumed to be $\log$-concave, \tr{then} an additional error related to $\mc{C}_{\log(u_0)}$ appears (see Remark \ref{rem_u0_nonconcav}).
\item 
In Theorem \ref{thm_approximate concavity parabolic application a(x)u^gamma} we observe that $ \norm{u(\cdot, \infty)}_{L^{\infty}(\Omega)}$ could be bounded in terms of known quantities of the problems, i.e. $n$, $|\Omega|$, $q$, $m$ and $M$, see e.g. \cite[proof of Corollary 1.4(b)]{Ciraolo-Cozzi-Perugini-Pollastro}.
%Notice that, choosing $\theta=1$ in \eqref{eq_quantitative_theta}, we can drop the dependence on the $L^{\infty}$-norm \tor{M: non più vero} and the statement becomes
%\begin{equation*}
%	\min_{\Omegabar\times \Omegabar\times[0,+\infty]\times [0,+\infty]\times[0,1]}\mathcal{C}_{u^{\frac{1-q}{3}}}
%	\geq
%	- \tr{\frac{1}{2}} %\frac{2}{3}
%	 \frac{1}{m} 
%\sup\limits_{\Omega \times \Omega\times[0,1]}
%\mathcal{C}^-_{a}.
%\end{equation*}
\item In Theorem \ref{thm_approximate concavity parabolic application a(x)u^gamma} the condition $m^\theta\geq M^{\theta}/2$ comes into play in showing quantitative concavity properties of products (see Proposition \ref{concavità prodotto}), but we believe it is technical.
%\item \tr{In Theorem \ref{thm_approximate concavity parabolic application a(x)u^gamma} w}e highlight that estimate \eqref{eq_quantitativ_oscil} is slightly finer if compared with the ones obtained in the elliptic setting in \cite{BuSq19, GaSq24}. 
%Moreover, the estimate \eqref{eq_quantitative_theta} does not involve the oscillation of $a$ and is totally new, even for the elliptic setting\tr{; see Theorem \ref{thm_elliptic_case_intr} below.}
\end{itemize}
\end{remark}

\begin{remark}[On the proofs]
\label{rem_proof_techni}
As mentioned in the introduction, differently from \cite{IsSa13, IsSa16, IshigeSalaniTakatsu2024}, our proofs rely on concavity maximum principles, which allow also to keep track of the concavity error.
Such maximum principles require that the minimum of the concavity function is in the \emph{interior} of the domain: to ensure such a property we will exploit both the information at the boundary $\partial \Omega\times [0,+\infty)$ %of $\Omega$ 
and the information at $\Omega \times \{0\}$; %$t=0$; 
when treating equations in $(0,+\infty)$, we will exploit also the information on the stationary problem (i.e. $t=+\infty$) (see Section \ref{subsec_harmon_boundar}). 
When $a(x,t)$ depends on $t$ in a coercive way (for instance, $a(x,t)=a(x) t^{\gamma}$) we also implement a truncation argument in order to be able to compare the problem with a stationary one at infinity (see Remark \ref{rem_troncamento}).
In such a way, we control all the information on $\partial(\Omega \times (0,+\infty))$ and we can apply the concavity maximum principles. 
We notice that, when dealing with $T<+\infty$, the nonlinearity is concave and concavity is discussed only in the spatial sense, no a priori information on $u$ in $t=T$ is needed. 
To further treat equations like the logarithmic one, we need to implement a correction in time of the type $e^{\mu t}$, $\mu$ large, which allows to decrease the monotonicity in $u$ of the nonlinearity involved, recovering the nonincreasing monotonicity which is needed for the argument (see \eqref{eq_drift_trick}, \eqref{eq_ipos_sigma_trans} and \eqref{eq_deriv_eq_B_Btilde_2} in Theorem \ref{thm_conc_princ_eigenf}). %\tr{; we notice that the idea is already hidden in \cite{PorruSerra}}. %\tor{M: Spieghiamo ancora meglio?}. \kt{R: no va bene

To gain the above results, we make some technical assumptions: on the nonlinearity $b(x,s,t)$ we require some conditions to ensure that $u(x, \cdot)$ is monotone nondecreasing, and it converges to the stationary problems (see conditions \hyperlink{H2}{\((H2)\)}, \hyperlink{H3}{\( (H3) \)} and (\textit{i})--(\textit{iv}) in Remark \ref{stability_parabolic condition u^gamma}). 
Moreover, we need to ensure the validity of some comparison principles (notice that $b(x,\cdot,t)$, in this setting, cannot be generally assumed locally Lipschitz) and that $u$ can be compared through such a tool with some suitable known subsolution, whose behaviour near $t=0$ and near $\partial \Omega$ is known (see conditions \hyperlink{H1}{\((H1)\)}-\hyperlink{H2}{\((H2)\)}); in particular we highlight that, compared with \cite{kenningtonparabolic}, the nonlinearities of interest in this paper vanish at zero, thus a finer analysis is needed. 
We notice again that, when logarithmic transformations are considered and the concavity is discussed only in the spatial sense, then the analysis on the boundary essentially relies on the strong convexity of the domain. 
Finally, in order to adapt the above proofs to some quantitative version and keep track of all the interested quantities in play, a finer and technical analysis of the relations among concavity and harmonic concavity functions is needed (see Appendix \ref{sec_prop_conc_fun}).
\end{remark}

%%%%%%%%%%%%%%%%%%%%%%%%%%%%%%%%%%%%%%%%%%%%%%%%%%%%%%%%%%%%%%%%%%%%%%%%
%%%%%%%%%%%%%%%%%%%%%%%%%%%%%%%%%%%%%%%%%%%%%%%%%%%%%%%%%%%%%%%%%%%%%%%%
\section{Preliminaries}
\label{sec_prelimin}

In what follows, we will denote by $\Omega \subset \R^n$ an open, bounded and convex domain. By smooth we will mean $\partial\Omega\in C^{2,\alpha}$, while by strongly convex we will mean that the second fundamental form with respect to its interior normal is positive
definite everywhere. 
We moreover recall that $\Omega$ satisfies the \textit{interior sphere condition} if for all $x\in \partial \Omega$, there exist $y\in \Omega$ and $r>0$ such that $B_r(y)\subseteq \Omega$ and $x\in \partial B_r(y)$.
We %also 
set 
$ d_\Omega(x):= \text{dist}(x,\partial \Omega)$ and 
$$\Omega_\rho:=\left\{x\in \Omega: d_{\Omega}(x)> \rho\right\}$$
for $\rho>0$.
We also denote $f^+:=\max\{0,f\}$ and $f^-:=\max\{0, -f\}$ the positive and negative parts of a function $f\colon \Omega \to \R$, so that $f=f^+-f^-$, and by 
$$\osc(f):=\sup_{\Omega} f-\inf_{\Omega} f$$
the oscillation of $f$, whenever well defined.

In the whole paper, $\cdot$ represents spatial dependence and $\star$ time dependence, e.g. we write $u(\cdot, \star)$ for $u\colon \Omegabar \times [0,T] \to \R$. 
%We will also use the following notation.

We finally denote 
$$Du:=(D_x u,\partial_t u) =(D_x u, u_t)$$
where $D_x$ is the spatial gradient and $\partial_t$ is the time derivative of $u$; in this regard, we will often decompose vectors as $p=(\tilde{p}, p^{n+1}) \in \R^n \times \R = \R^{n+1} $.

%%%%%%%%%%%%%%%%%%%%%%%%%%%%%%%%%%%%%%%%%%%%%%%%%%%%
%%%%%%%%%%%%%%%%%%%%%%%%%%%%%%%%%%%%%%%%%%%%%%%%%%%%
\subsection{On parabolic equations}

In this subsection we recall some facts on the parabolic problem
	\begin{equation}\label{eqz problema limite, parabolico}
		\begin{cases}
			u_t-\Delta u=b(x,u,Du, t) &\text{in }\Omega\times (0,+\infty),\\
			u>0 & \text{in } \Omega\times (0,+\infty),\\
			u=0 & \text{on }\partial \Omega\times (0,+\infty),\\ 
			u=u_0& \text{on } \Omega \times \{0\}, 
		\end{cases}
	\end{equation}
where $b\colon\Omega\times \R\times \R^{n+1}\times \R\to \R$.

\begin{remark} \label{rem_esist_soluz}
We briefly comment existence and regularity results for problem \eqref{eqz problema limite, parabolico} when $b(x,s,p,t)=a(x) f(s)$ and $u_0=0$. 
	Assumed $a\in C(\Omegabar)\cap C^\sigma(\Omega)$, $\sigma \in (0,1]$, $a\not\equiv 0,$ $a\geq 0$, $f\in C([0,+\infty))$, nondecreasing, $f(0)=0$ and $f>0$ on $(0,+\infty)$ we have that \eqref{eqz problema limite, parabolico} admits a nontrivial solution if (and only if) (see \cite[Theorem 1.5]{Fujita1968}, \cite[Corollary 2.6]{LaisterRobinsonSierzega})
	\[ \int_{0}^{\varepsilon}\frac{1}{f(s)}ds<+\infty \quad\text{for some }\varepsilon>0.\]
	Observe that $f(s)= s^q$ with $q\in [0,1)$ satisfies such condition. For regularity results we refer again to \cite{LaisterRobinsonSierzega, Fujita1968}.
	For local and global existence in the case $u_0 \neq 0$ see e.g. \cite[Lemma 3.5 and Theorem 2.1]{Kajikiya2018}. 
	We further mention the existence of weighted eigenfunctions with periodic conditions, see \cite{Daners2000, GodoyKaufmann2001} and references therein; see also Remark \ref{rem_converg_logarithm} for some comments on the logarithmic equation.
\end{remark}

To give concavity information on the boundary of $(0,+\infty)$, and in particular for $t=+\infty$, we need to be able to compare the evolutive problem with the stationary one. We give thus the following definition.
\begin{definition}\label{defn problema limite}
	Let $\Omega \subset \R^n$ be open and consider $u\colon\Omega\times (0,+\infty)\to \R$ such that \eqref{eqz problema limite, parabolico} holds.
	We will say that $b$ satisfies the \emph{parabolic stability condition} if 
\begin{equation}\label{eq_stab_par_cond}
\lim\limits_{t\to +\infty} ||u(\cdot,t)-v||_{L^\infty(\Omega)}=0 
\end{equation}
	 for some $v\colon\Omega\to \R$ solution of 
	\begin{equation}\label{eqz problema limite, ellittico}
		\begin{cases}
			-\Delta v=b(x,v,D_xv,\infty) &\text{in }\Omega,\\
			v=0 & \text{on }\partial \Omega;
		\end{cases}
	\end{equation}
	here $b(x,s,\tilde{p},\infty):=\lim\limits_{t\to+\infty}b(x,s,\tilde{p}, 0,%p,
t)$ is assumed to exist for all $(x,s,p)\in \Omega\times (0,+\infty)\times \R^{n+1}$, where $p=(\tilde{p},p^{n+1})$.
\end{definition}

\begin{remark}\label{stability_parabolic condition u^gamma}
	We briefly comment the validity of the parabolic stability condition: 
	in \cite[Theorem 1.1]{AkagiKajikiya2015} it is proved that such property is satisfied by the power $b(x,s,p,t)=|s|^{q-1}s$
	where $0<q<1$.	
	Other examples can be deduced by \cite{Kajikiya2018}: suppose indeed that $b(x,s,p,t)=b(x,s)$ satisfies the following facts
	\begin{enumerate}[label=\textit{\roman*})]
		\item $b \in C(\Omegabar\times \R) \cap C_{loc, x}^{\sigma}(\R)$ for some $\sigma\in (0,1)$ and $b(x, \cdot)$ is odd for each $x \in \Omega$;
		\item there exists $C>0$ such that \[ |b(x,s)|\leq C(|s|^q+1)\quad \text{for }s\in\R, \ x\in \Omegabar ,\]
		where $q>1$ if $n=1,2$ and $q\in (1,\frac{n}{n-2})$ if $n\geq 3$;
		\item it holds \[ \limsup_{|s|\to +\infty} \left(\max_{x\in \Omegabar}\frac{b(x,s)}{s}\right) <\lambda_1,\]
		where $\lambda_1$ is the first eigenvalue of the Laplacian.
	\end{enumerate} 
	Then, for any $u_0 \in H^1_0(\Omega)$, \cite[Theorem 2.1]{Kajikiya2018} states that the orbit of $u(\cdot,t)$ for all $t\in(0,+\infty)$ is relatively compact in $C^{1,\theta}(\Omegabar)$ for some $\theta\in (0,1)$, and the limit points are stationary solutions.%
\footnote{We are confident that this relative compactness to stationary solutions -- which is weaker than the parabolic stability condition, and does not require (\textit{iv}) below -- is sufficient to prove Theorem \ref{approximate concavity parabolic with boundary conditions}. On the other hand, in order to make use of the notation $v(\cdot, \infty)$ and keep the presentation simple, we prefer not to improve the generality of the result.}
	Suppose moreover that 
	\begin{itemize}
	\item[$iv)$] $s\in(0,+\infty) \mapsto \frac{b(x, s)}{s}$ is strictly decreasing for each $x \in \Omega$ 
	\end{itemize}
	so that the limit problem \eqref{eqz problema limite, ellittico} admits a unique solution $v$ by a Brezis-Oswald result (see e.g. \cite[Lemma A.1]{GaSq24}), then $u(\cdot,t) \to v $ in $C^{1,\theta}(\Omegabar)$, which in particular implies \eqref{eq_stab_par_cond}. 
	Notice that, if (\textit{i}) and (\textit{iv}) hold, then (\textit{iii}) means $\inf_{s>0} \left(\max_{x\in \Omegabar}\frac{b(x,s)}{s}\right) <\lambda_1$.
	
	Some functions that satisfy (\textit{i})--(\textit{iv}) are given by $b(x,s)=a(x) f(s)$ with $f(s)=|s|^{q-1}s$ or $f(s)=-s\log |s|$, %or
	% $f(s)=(1-s)^q \chi_{\tr{[}0,1\tr{]}}(s)$, $q \in (0,1)$ \tor{Non è odd / continua in $0$} (observed that the solutions verify $u(\tb{\Omega}) \subset [0,1]$ \tor{La soluzione dovrebbe dipende anche dal tempo; se $u_0$ non ha immagine in $[0,1]$, però, questa cosa non può essere vera}), 
	and $a \in C(\Omegabar)$ is positive.
	Moreover, also the sum of these functions satisfies (\textit{i})--(\textit{iv}). 
	The nonlinearity $b(s)=(1-s)^q$, $q \in (0,1)$ (see \cite{kenningtonparabolic}), does not verify $(i)$, but nevertheless verifies the parabolic stability condition when $u_0=0$: this can be shown in a standard way, once observed that $0\le u(x,t)\le v(x) \leq \|v\|_{L^\infty(\Omega)}<1$, where $v$ is the unique positive stationary solution, that $u_t \geq 0$ by Corollary \ref{corol_monotonicity}, and that standard parabolic estimates hold.
\end{remark}

Our problems will often involve a nonlinear term $b=b(x,s,t)$; we present here some main assumptions for $b\colon\Omega\times(0,+\infty) \times (0,+\infty)\to \R$.
%We will require $b\colon\Omega\times(0,+\infty) \times (0,+\infty)\to \R$ to satisfy:
	\begin{itemize}
		\item[]\hypertarget{H1}\((H1)\) 
	There exists $T>0$, $M>0$, $k>0$, $q\in [0,1)$, and $\gamma \in [0,1]$,
		such that
		$$b(x,s, t)\geq k t^{\gamma} s^q \quad \hbox{for all $(x,s,t)\in \Omega\times(0,M]\times (0,T]$}.$$
		\item[]\hypertarget{H2}\((H2)\) There exists $T>0$ such that $b(\cdot,s,t)$ is measurable and nonnegative for $(s,t)\in(0,+\infty)\times (0,T)$, and for all $M>0$ there exists $L=L(M,T)>0$ verifying
		$$b(x,s,t)-b(x,r,t)\leq \frac{L}{r}(s-r) \quad \hbox{for all $0<r\leq s\leq M$ and $(x,t) \in \Omega \times (0, T]$}.$$
		\item[]\hypertarget{H3}\((H3)\) The function $b(x,s,\cdot)$ is nondecreasing for each $(x,s) \in \Omega \times (0,+\infty)$.
	\end{itemize} 
 With no loss of generality, whenever assumed together, 
we can choose the constant $T$ in \hyperlink{H1}{\((H1)\)}-\hyperlink{H2}{\((H2)\)} to be the same.
In practical situations, the constant $M$ in \hyperlink{H2}{\((H2)\)} will actually be given by the $L^\infty$-norm of the solution involved. 
Moreover, since we will apply \hyperlink{H1}{\((H1)\)} only for small values of $t$, if $u_0=0$ then for $t$ small $\norm{u(\cdot, t)}_{L^{\infty}(\Omega)}$ can be chosen to be smaller than the fixed $M$ in \hyperlink{H1}{\((H1)\)} (see also condition \hyperlink{H1*}{\( (H1^*) \)}).

\begin{remark}\label{remark comparison principle}
	Observe that the function $b(x,s,t) = a(x,t) s^q$, with $q \in (0,1)$ and $a \in L^{\infty}(\Omega \times \R)$ nonnegative, satisfies \hyperlink{H2}{\((H2)\)}.
	Indeed, for each $0<r\leq s\leq M$ by Lagrange theorem there exists $\xi\in (r,s)$ such that
	\begin{align*}
		b(x,s,t) - b(x,r,t) = a(x, t) \frac{q}{\xi^{1-q}}(s-r)\leq \norm{a}_{\infty} \frac{ q s^q}{r}(s-r)\leq \norm{a}_{\infty} \frac{q M^q}{r}(s-r).
	\end{align*}
%	\tb{In a similar way, we see that also $b(x,s,t)=a(x,t) s^q \log(s)$ satisfies such condition for $q \in [0,1]$ (notice that for $q=0$ $b$ is not continuous in $s=0$).} \tor{M: non è nonnegative, che è richiesto in (H2); dato che non mi sembra venga usato l'esempio, lo toglierei}
Notice finally that \hyperlink{H2}{\((H2)\)} is closed under summation.
\end{remark}

\begin{remark}
	In our arguments, it is possible to substitute \hyperlink{H2}{\((H2)\)} with the following: 
		\begin{itemize}
			\item[]\hypertarget{H2*}\((H2^*)\) 
			There exists $T>0$ and $\omega \in [\frac{1}{2},1]$ such that $b(\cdot,s,t)$ is measurable and nonnegative for $(s,t)\in(0,+\infty)\times (0,T)$, and for all $M>0$ there exists $ L=L(M,T)>0$ verifying
			\[	b(x,s,t)-b(x,r,t)\leq L (s-r)^{\omega} \quad \hbox{for all $0<r\leq s\leq M$, $x\in \Omega$ and $t\in (0,T)$}.\]
		\end{itemize} 
Notice that any $\omega$-Hölder function with $\omega \in [\frac{1}{2},1]$ satisfies \hyperlink{H2*}{\( (H2^*) \)}; moreover $b(s)=\sqrt{s \abs{1-s}}$ satisfies \hyperlink{H2*}{\( (H2^*) \)} but not \hyperlink{H2}{\((H2)\)}.
On the other hand, from a practical point of view related to the applications in the present paper, we see that \hyperlink{H2*}{\( (H2^*) \)} is somehow stronger than \hyperlink{H2}{\((H2)\)}, since $b(s)=s^q$ with $q\in (0,\frac{1}{2})$ satisfies \hyperlink{H2}{\((H2)\)} but not \hyperlink{H2*}{\( (H2^*) \)}. 
See also Remark \ref{rem_u0_nonzero}.
\end{remark}

\begin{remark}\label{rem_troncamento}
We see that fixed $T>0$ and set 
$$b_T(x,s,t):=
\begin{cases}
b(x,s,t) & \quad \hbox{for $t \in [0,T]$}, \\
b(x,s,T) & \quad \hbox{for $t \in [T,+\infty)$},
\end{cases}
$$
then, if $b$ satisfies \hyperlink{H2}{\((H2)\)} (resp. \hyperlink{H3}{\( (H3) \)}), then $b_T$ satisfies \hyperlink{H2}{\((H2)\)} (resp. \hyperlink{H3}{\( (H3) \)}) as well. 
As a consequence of the uniqueness result in Proposition \ref{comparison principle dickstein}, since $b$ and $b_T$ coincide for $t \in [0,T]$, we see that positive solutions of
$$
\begin{cases}
u_t-\Delta u = b(x,u,t)&\hbox{in }\Omega\times(0,+\infty),\\
u=0 &\hbox{on } \partial \Omega\times(0,\infty),\\
u=u_0&\hbox{on } %x\in 
\Omega \times \{0\},
\end{cases}
\quad
\begin{cases}
v_t-\Delta v = b_T(x,v,t)&\hbox{in }\Omega\times(0,+\infty),\\
v=0 &\hbox{on } \partial \Omega\times(0,\infty),\\
v=u_0&\hbox{on } %x\in 
\Omega \times \{0\},
\end{cases}
$$
coincide on $[0,T]$. As a consequence, to study (exact) concavity of the first equation, we can study concavity of the second one and obtain the result thanks to the arbitrariness of $T$. 
Notice that the second problem has the advantage of possibly satisfying the parabolic stability condition even if $b(x,s,t) \to +\infty$ as $t\to +\infty$ (for instance, $b(x,s,t)=t^{\gamma} a(x) s^q$). 
Indeed, consider $v$ solution of the second system and define $\tilde{v}(x,t):=v(x,t+T)$. Then $\tilde{v}$ satisfies
$$	
\begin{cases}
\tilde{v}_t-\Delta \tilde{v} = b_T(x,\tilde{v},t+T) = b(x, \tilde{v}, T) &\hbox{in }\Omega\times(0,+\infty),\\
\tilde{v}=0 &\hbox{on } \partial \Omega\times(0,\infty),\\
\tilde{v}=v(T) \in H^1_0(\Omega) &\hbox{on } % x\in 
\Omega \times \{0\}.
\end{cases}
$$
Thus, if $b(x,s,T)$ satisfies the parabolic stability condition and $\tilde{v}(\cdot, t) \to \bar{v}$ uniformly, $\bar{v}$ solution of a stationary problem $-\Delta \bar{v} = b(x,\bar{v}, T)$, then $v(\cdot,t)\to \bar{v}$ uniformly as well, and thus $b_T(x,s,t)$ satisfies the parabolic stability condition as well. 
\end{remark}

%%%%%%%%%%%%%%%%%%%%%%%%%%%%%%%%%%%%%%%%%%%%%%%%%%%%
%%%%%%%%%%%%%%%%%%%%%%%%%%%%%%%%%%%%%%%%%%%%%%%%%%%%
\subsection{Power concavity and concavity functions}

Let us recall some notions about $\alpha$-concavity, see \cite{kennington} for details. 

\smallskip

\textbf{Notation.} Throughout the paper, whenever $z_1, z_3$ are picked in a convex domain together with a $\lambda \in [0,1]$, we will denote by $z_2$ the convex combination
$$z_2 \equiv \lambda z_3 + (1-\lambda) z_1.$$
Moreover, whenever a function $g$ is in play, we will denote 
$$g_i \equiv g(z_i) \quad \hbox{for $i=1,2,3$},$$
if it creates no ambiguity. Finally if $x\in \R^n$, then $x^i$ denotes the $i$-th component %s 
of $x$.

\medskip

Let $\alpha \in [-\infty, +\infty]$. We recall that a nonnegative function $u\colon\Omega \to \R$ is \emph{$\alpha$-concave} if 
\begin{equation}\label{eq_def_alpha_concav}
	\begin{cases}
		u \text{ is constant} & \text{if } \alpha =+\infty,\\
		u^\alpha \text{ is concave} & \text{if } 0<\alpha<+\infty,\\
		\log(u) \text{ is concave} & \text{if } \alpha =0,\\
		u^\alpha \text{ is convex} & \text{if } -\infty<\alpha<0,\\
		\left\{ x\in \Omega: u(x)>t\right\} \text{ is %are 
		convex }\forall t\in \R & \text{if }\alpha=-\infty;
	\end{cases}
\end{equation}
if $-\infty<\alpha\leq 0$ we additionally require $u>0$. 
Moreover, we say that $u$ is \emph{log-concave} if it is $0$-concave and \emph{quasi-concave} if it is $(-\infty)$-concave.

 With the aim of treating perturbations of concavity, we introduce two operators which allow to quantify the concavity of a function.

\begin{definition}\label{concavity function definition}
	Let $u\colon\Omega\times (0,+\infty) \to \R$, then for each $x_1,x_3\in \Omega, \ t_1, \ t_3\in (0,+\infty)$ and $\lambda \in [0,1]$, %then 
	the \emph{concavity function} $\mathcal{C}_u$ is defined as 
	\begin{equation}\label{defn concavity function}
		\mathcal{C}_u(x_1,x_3,t_1,t_3,\lambda):= u(x_2,t_2)-\lambda u(x_3,t_3)-(1-\lambda)u(x_1,t_1)
	\end{equation}
	where we recall $x_2=\lambda x_3+(1-\lambda)x_1$ and $t_2=\lambda t_3+(1-\lambda)t_1$. 
%\tr{Similarly, if $g\colon\Omega\times (0,+\infty)\to \R$, we write
%$$		\mathcal{C}_g(x_1,x_3,s_1,s_3,\lambda):= g(x_2,s_2)-\lambda g(x_3,s_3)-(1-\lambda)g(x_1,s_1).$$
%}
	
	Let $g\colon\Omega\times (0,+\infty)\to \R$, then 
we introduce
\begin{align*}
\Dom(\mathcal{HC}_g)&:= \Big\{ (x_1,x_3,s_1,s_3,\lambda) \in \Omega\times\Omega\times(0,+\infty)\times(0,+\infty)\times[0,1] \mid \\
& \qquad \qquad \lambda g(x_1,s_1)+(1-\lambda)g(x_3,s_3)>0 \quad \hbox{or} \quad g(x_1,s_1)=g(x_3,s_3)=0\Big\}
\end{align*}
and the \emph{harmonic concavity function} $\mathcal{HC}_g$ defined on $\Dom(\mathcal{HC}_g)$ by % is defined a

	\begin{equation*}
		\mathcal{HC}_g\big((x_1,s_1),(x_3,s_3),\lambda\big):=
		\begin{dcases}
			\mathmakebox[4em][l]{g(x_2,s_2)-\dfrac{g(x_1,s_1)g(x_3,s_3)}{\lambda g(x_1,s_1)+(1-\lambda)g(x_3,s_3)}} \\[1ex]
			\quad\quad\qquad\qquad\text{if } \lambda g(x_1,s_1)+(1-\lambda)g(x_3,s_3)>0,
			\\[1ex]
			g(x_2,s_2)
			\qquad\text{ if } 
			g(x_1,s_1)=g(x_3,s_3)=0.
		\end{dcases}
	\end{equation*}
We say that $g$ is \emph{harmonic concave} if $\mathcal{HC}_g \geq 0$ on its domain.
	Notice that, if $g \geq 0$, %$g\colon\Omega\times (0,+\infty)\to [0,+\infty)$, 
then 
$\Dom(\mathcal{HC}_g)=\Omega\times\Omega\times(0,+\infty)\times(0,+\infty)\times[0,1]$; moreover, if $f\geq g$, then $\Dom(\mc{HC}_g) \subseteq \Dom(\mc{HC}_f)$. Finally, if $g>0$, then the harmonic concavity is equivalent to the $(-1)$-concavity in \eqref{eq_def_alpha_concav}, while if $g<0$ then $g$ is always harmonic concave being $\Dom(\mc{HC}_g)=\emptyset$.
\begin{remark}
In the rest of the paper, we will implicitly assume that the computations involving $\mc{HC}_g$ are made in the points of its domain.
\end{remark}
%for each $x_1,x_3\in \Omega,\ s_1,s_3\in (0,+\infty),$ $\lambda \in [0,1]$ such that
%	$$\lambda g(x_1,s_1)+(1-\lambda)g(x_3,s_3)>0 \quad \hbox{or} \quad g(x_1,s_1)=g(x_3,s_3)=0$$
%	the \emph{harmonic concavity function} $\mathcal{HC}_g$ is defined as
%	\begin{equation*}
%		\mathcal{HC}_g\big((x_1,s_1),(x_3,s_3),\lambda\big):=
%		\begin{dcases}
%			\mathmakebox[4em][l]{g(x_2,s_2)-\dfrac{g(x_1,s_1)g(x_3,s_3)}{\lambda g(x_1,s_1)+(1-\lambda)g(x_3,s_3)}} \\[1ex]
%			\quad\quad\qquad\qquad\text{if } \lambda g(x_1,s_1)+(1-\lambda)g(x_3,s_3)>0,
%			\\[1ex]
%			g(x_2,s_2)
%			\qquad\text{ if } 
%			g(x_1,s_1)=g(x_3,s_3)=0.
%		\end{dcases}
%	\end{equation*}
%	We say that $g$ is \emph{harmonic concave} if $\mathcal{HC}_g \geq 0$.
%	Notice that, if $g\colon\Omega\times (0,+\infty)\to [0,+\infty)$, then $\mathcal{HC}_g$ is defined everywhere. \tr{Otherwise, if $g\colon\Omega\times (0,+\infty)\to \R$ is negative somewhere, then $\mathcal{HC}_g$ has a domain included in $\Omega\times\Omega\times(0,+\infty)\times(0,+\infty)\times[0,1]$.}
\end{definition}
%It can be easily seen that, whenever $g>0$, the above definition of harmonic concavity coincide with the $(-1)$-concavity.
%Moreover it holds 
We have \cite[page 4]{BuSq19}
\begin{equation}\label{eq_hc_grt_c}
	\mc{HC}_g \geq \mc{C}_g,
\end{equation}
and, for any $\beta \in (0,1]$, (see \cite[equation (2.2)]{IsSa13})
\begin{equation}\label{eq_harmonic_beta}
	\mc{HC}_{u(\cdot, \star^{\beta})}(x_1, x_3, t_1,t_3,\lambda) \geq \mc{HC}_{u(\cdot, \star)}(x_1, x_3,t_1^{\beta},t_3^{\beta},\lambda).
\end{equation}
We highlight that $\mc{C}$ and $\mc{HC}$ measure the concavity also with respect to $t$. When the concavity in time is not taken into account, we consider 
\[ \mathcal{C}^*_u(x_1,x_3,\lambda, t):=u_2(t)-\lambda u_3(t)-(1-\lambda)u_1(t), \] 
where
$$u_i(t):= u(x_i,t).$$
Notice again that $\mc{C}^*$ differs from $\mc{C}$ since there is no variation in time. Straightforward generalizations of the previous definitions will be used.
Additionally, we will use the following notation.

\smallskip

\textbf{Notation.} We will write, for $g\colon\Omega\times (0,+\infty)\times (0,+\infty)\to \R$ and $u\colon\Omega\times(0,+\infty)\to \R$, 
%$$	\mathcal{HC}_{g(\cdot, u(\cdot,\star))}(x_1,x_3,t_1,t_3,\lambda):=g\big(x_2,\lambda u_3+(1-\lambda)u_1\big)-\dfrac{g(x_1,u_1)g(x_3,u_3)}{\lambda g(x_1,u_1)+(1-\lambda)g(x_3,u_3)},$$
%$$	\mc{C}^*_{g(\cdot, u(\cdot,\star), \star)}(x_1,x_3,\lambda,t):=g \big(x_2,\lambda u_3(t)+(1-\lambda)u_1(t),t\big) - \lambda g(x_3, u_3(t),t) - (1-\lambda) g(x_1, u_1(t),t),$$
\begin{align*}
\mc{C}^*_{g(\cdot, u(\cdot,\star), \star)}(x_1,x_3,\lambda,t) \, :=& \mc{C}_{g}(x_1, x_3, u_1(t), u_3(t), \lambda,t) \\ 
=& g \big(x_2,\lambda u_3(t)+(1-\lambda)u_1(t),t\big) - \lambda g(x_3, u_3(t),t) - (1-\lambda) g(x_1, u_1(t),t)%. 
\end{align*}
and
\begin{align*}
\mathcal{HC}_{g(\cdot, u(\cdot,\star), \star)}(x_1,x_3,t_1,t_3,\lambda) \, :=& \mathcal{HC}_g(x_1, x_3, u_1, u_3, t_1, t_3, \lambda) \\
=&g\big(x_2,\lambda u_3+(1-\lambda)u_1, t_2\big)-\dfrac{g(x_1,u_1, t_1)g(x_3,u_3,t_3)}{\lambda g(x_1,u_1,t_1)+(1-\lambda)g(x_3,u_3,t_3)},
\end{align*}
whenever $(x_1, x_3, u_1, u_3, t_1, t_3, \lambda) \in \Dom(\mc{HC}_g)$.

\medskip
In what follows, for the sake of brevity, we adopt the conventions $0^{-1}=+\infty$, $(+\infty)^{-1}=0$ and $0^0=1$.
We recall results from \cite[Section 2 and Lemma A2]{kennington}, \cite[Lemma 3.2]{kenningtonparabolic}.
\begin{proposition}[\cite{kennington, kenningtonparabolic}]
%\protect{\cite[\tr{Section 2 and Lemma A2}]{kennington}}, \protect{\cite[\tr{Lemma 3.2}]{kenningtonparabolic}}] 
\label{alpha concavity property} 
Let $u\colon\Omega \to \R$ be a nonnegative function, and recall notations in \eqref{eq_def_alpha_concav}. 
	\begin{enumerate}[label=\textit{\roman*})]
		\item If $u$ is $\alpha$-concave and $\beta \leq \alpha$, then $u$ is $\beta$-concave. 
		\item Let $u$ be quasi-concave. 
		Set $ \alpha(u):=\sup \left\{\beta \in \R: u \text{ is }\beta\text{-concave}\right\}$, we have that $u$ is $\alpha(u)$-concave.
		\item For $\alpha\geq 1$, the $\alpha$-concave functions are closed under positive addition and positive scalar multiplication.
		\item If $\alpha,\beta \in [0, +\infty]$, 	$u$ is $\alpha$-concave and $v$ is $\beta$-concave, then the product $u \, v$ is $\gamma$-concave for $\gamma^{-1}=\alpha^{-1}+\beta^{-1}$.
		\item If $u$ is harmonic concave, then $u-k$ is harmonic concave for each nonnegative $k \in\R$.
		\item If $u$ is concave and $u>0$, then $x \mapsto \frac{(x^1)^2}{u(x)}$ is convex.
	\end{enumerate}
\end{proposition}

In Appendix \ref{sec_prop_conc_fun} we show some generalizations of Proposition \ref{alpha concavity property} (\textit{iv}), (\textit{v}) and (\textit{vi}) in terms of the concavity and harmonic concavity functions.

Furthermore, we recall the following classical result \cite[Theorem 2]{ulam}.
\begin{proposition}[Hyers-Ulam, \cite{ulam}] %\protect{\cite[\tr{Theorem 2}]{ulam}}]
	\label{ulam}
	Let $ X $ be a space of finite dimension and $ K \subset X $ convex. Assume that $ f\colon K \rightarrow \mathbb{R} $ verifies, for some $\delta>0$,
	%is $ \delta $-concave, i.e. 
	\[ \mathcal{C}_f \geq -\delta \quad \hbox{in $K\times K \times [0,1]$}.\] 
	Then there exists 
	a concave function $g : K \rightarrow \mathbb{R} $ such that 
	$$ \|f-g\|_{L^{\infty}(K%D
)}\leq k_{n} \delta,$$
	 where $ k_{n}>0 $ depends only on $ n=\operatorname{dim}(X) $.
\end{proposition}	

We highlight that the optimal constant $k_n$ has been explicitly estimated, for instance by \cite{ulam} and \cite{Lac99} we know that $ \frac{1}{4}\log_2(\frac{n}{2}) \leq k_n \leq \frac{n(n+3)}{4(n+1)}$. 
Notice that $X$ can be chosen as the minimal hyperplane containing $K$.

\medskip

We finally state the following constant rank theorem, consequence of \cite[Theorem 1.5]{CaffareliGuanMa2007}, \cite[Theorem 1.2]{BianGuan2009}, and \cite[Theorem 1.3]{chen2014}, \cite[Theorem 1.3]{ChenHu2013}. %\cite{BianGuan2009, CaffareliGuanMa2007} and \cite{chen2014, ChenHu2013}.

\begin{proposition}[Constant rank theorem]
\label{prop_const_rank}
Let $\Btilde \in C^{2,1}(\Omega \times %,
 \R \times %,
  \R^n \times %, 
  (0,T])$ and $w\in C^{2,1}(\Omega\times(0,T])$
be a spatially convex solution of
\begin{equation}\label{eq_const_rank}
w_t - \Delta w = \Btilde(x,w,D_x w, t) \quad \hbox{in $\Omega \times (0,T]$}
\end{equation}
and assume $\Btilde(\cdot, \cdot, \tilde{p}, t)$ locally convex 
for each $(\tilde{p},t) \in \R^n \times (0, T)$. Then, for each $t\in (0,T)%>0
$, the rank of $D^2_x w(\cdot, t)$ is constant. Moreover, called $\ell(t)$ such value, we have that $t \in (0,T) %\tr{T]}
 \mapsto \ell(t) \in \N$ is nondecreasing.

If moreover %\tr{$\Btilde \in C^{2,1}(\Omega, \R, \R^N, (0,T])$, $w\in C^{2,1}(\Omega\times(0,T])$}, 
$w$ is spacetime-convex and $\Btilde(\cdot, \cdot, \tilde{p}, \cdot)$ locally convex for each $\tilde{p} \in \R^n$, then the same conclusion as above applies to $D^2 w(\cdot, t)$. %\tr{M: bisogna togliere la $t$?}.
\end{proposition}

\begin{remark}\label{rem_appl_const_rank}
We observe the following facts.
\begin{itemize}
\item The above result will be applied in the following way: let $u$ be a solution of our Dirichlet problem, and $v=\varphi(u)$ its concave transformation. 
Then $w:= -v$ is convex and satisfies an equation as \eqref{eq_const_rank}. By \cite[Proposition 2.6]{Gallo-Mosconi-Squassina} and the Dirichlet boundary condition, for each $t$ the matrix $D^2_x w(\cdot,t)$ has full rank in a point (equal to $n$). Thus, by the above constant rank theorem, it is maximum in the whole $\Omega$, and hence here $w$ is strictly convex for each $t$. That is, $v(\cdot, t)$ is strictly concave for each $t$.

If $w$ is also spacetime-convex, then $D^2 v(x,t)$ has constant rank for each $t$, and it can be equal to $n$ or $n+1$; more precisely, by the monotonicity, it is equal to $n$ in $(0,t_0)$ and equal to $n+1$ in $(t_0,+\infty)$ for some $t_0 \in [0,+\infty]$.

\item
In the elliptic case \cite{korevaar1987convex}, the condition $\Btilde(\cdot, \cdot, \tilde{p}, \cdot)$ convex can be substituted by $\Btilde(\cdot, \cdot, \tilde{p}, \cdot)$ positive and harmonic convex: as highlighted in \cite{BianGuan2010}, this is essentially related to the fact that the equation $0=\Delta w + \Btilde(x,w,D_{x}w,t) =: F(D_x^2 w, x, w, D_{x}w, t)$ can be rewritten as $0=\frac{1}{\Delta w} + \frac{1}{\Btilde(x,w,D_{x}w,t)} =: \bar{F}(D_x^2 w, x, w, D_{x}w, t)$. This is clearly not the case for parabolic equations $w_t =F(D_x^2 w, x, w, D_{x}w, t)$; this is the reason why we obtain the strict concavity in Theorem \ref{thm_general_korevaar} but not in Theorem \ref{thm_main_weighted_le}. A weaker parabolic counterpart of this generalization can anyway be found in \cite[Theorem 2.3]{BianGuan2010}.
\end{itemize}
\end{remark}

 %%%%%%%%%%%%%%%%%%%%%%%%%%%%%%%%%%%%%%%%%%%%%%%%%%%%%%%%%%%%%%%%%%%%%%%%
%%%%%%%%%%%%%%%%%%%%%%%%%%%%%%%%%%%%%%%%%%%%%%%%%%%%%%%%%%%%%%%%%%%%%%%%
\section{Concavity and harmonic concavity maximum principles}
\label{sec_concav_max_princip}

In this section we present two maximum principles in the concavity framework, namely Theorem \ref{thm_conc_princ_eigenf} and Theorem \ref{Concavity perturbative maximum principle for parabolic equations}.

We start by showing Theorem \ref{thm_conc_princ_eigenf}, a maximum principle which allows to relate the interior minimum of the concavity function of a solution $v$ to the concavity function of the source; in applications, $v$ will be the transformation of a solution of a PDE. 
In this section we consider a generalization of the result by Korevaar \cite{Korevaar}, by dealing with equations which may include weighted eigenfunctions $ u_t - \Delta u = a(x,t) u$. 
%\tb{Differently from Theorem \ref{Concavity perturbative maximum principle for parabolic equations}} \tr{Non chiaro, spiega meglio}, 
Differently from \cite{Korevaar}, the corresponding $\mathscr{B}$ is not required to be nonincreasing w.r.t. the solution $v$; actually, we will allow $\partial_s \mathscr{B}$ to be also positive, but bounded, in the spirit of \cite{PorruSerra} (see also \cite{GrecoKawohl}): this permits to study also equations of the type $u_t-\Delta u = a(x,t) u \log(u)$ (see \cite{Gallo-Mosconi-Squassina}). 
In these results, we look at some concavity of $\mathscr{B}$, and we study concavity of $u$ in $x$ for each $t$.

We recall the notation $p=(\tilde{p}, p^{n+1}) \in \R^{n+1}$.

\begin{theorem}[Concavity maximum principle]
\label{thm_conc_princ_eigenf}
Let $T>0$ and $v \in C^2_x(\Omega) \cap C^1_t((0,T])$ be such that
\[ - \sum_{i,j=1}^n a_{ij}(D_x v, t) D^2_{ij} v = \mathscr{B}(x,v,D v, t) \quad \hbox{in $\Omega \times (0,T]$} \]
	where $\mathscr{B}\colon \Omega\times \R\times \R^{n+1} \times (0,T]\to \R $ locally bounded
	% \tor{M: serve richiedere locally bounded?} 
	and $(a_{ij}(\tilde{p},t))_{ij}$ is symmetric and positive definite for all $(\tilde{p},t)\in \R^{n} \times (0,T]$. 
	Assume that $(s, p^{n+1})\in v(\Omega\times [0,T]) \times \R \mapsto \mathscr{B}(x,s,\tilde{p}, p^{n+1}, t)$ is differentiable 
	for all $(x,\tilde{p},t)\in \Omega\times \R^{n} \times (0,T]$. 
	Let $\mu \in [0,+\infty)$ and suppose $e^{-\mu (\star)}\mathcal{C}^*_v$ admits a negative interior minimum at $(x_1,x_3,\lambda,t)\in \Omega \times \Omega\times [0,1] \times (0,T]$. 
	Then
	\begin{equation}\label{eq_def_xi*}
	\tilde{\xi}:=D_xv(x_1,t)=D_xv(x_2,t)=D_xv(x_3,t).
	\end{equation}
Assume moreover 
\begin{equation}\label{eq_ipos_dpnB}
\partial_{p^{n+1}} \mathscr{B} \leq 0
\end{equation}
and that there exists $\sigma >0 $ such that%
\footnote{It is indeed sufficient to assume
%$\big(\partial_s \mathscr{B}+ \mu \partial_{p^{n+1}} \mathscr{B}\big)\big(x_2, e^{\mu t} z, e^{\mu t} \tilde{\xi}, e^{\mu t}( \mu z + \zeta), t\big) \leq \tr{-\sigma}$ %\Lambda$
%for any $z \in [w_2(t), \lambda w_3(t)+(1-\lambda)w_1(t)]$ and $\zeta \in [w_t(x_2,t), \lambda w_t(x_3,t) + (1-\lambda) w_t(x_1,t)]$.
%\tr{$\big(\partial_s \mathscr{B}+ \mu \partial_{p^{n+1}} \mathscr{B}\big)\big(x_2, s,  \bar{\xi}, \mu(z+\zeta), t\big) \leq \tr{-\sigma}$ %\Lambda$
%for any $z \in [v_2(t), \lambda v_3(t)+(1-\lambda)v_1(t)]$ and $\zeta \in [v_t(x_2,t), \lambda v_t(x_3,t) + (1-\lambda) v_t(x_1,t)]$.}
$\big(\partial_s \mathscr{B}+ \mu \partial_{p^{n+1}} \mathscr{B}\big)\big(x_2, z,  \tilde{\xi}, \zeta, t\big) \leq -\sigma$ %\Lambda$
for any $z \in [v_2(t), \lambda v_3(t)+(1-\lambda)v_1(t)]$ and $\zeta \in [v_t(x_2,t), \lambda v_t(x_3,t) + (1-\lambda) v_t(x_1,t)]$.
%\[
%\big(\partial_s \mathscr{B}
%+\mu \partial_{p^{n+1}}\mathscr{B}\big)
%\big(x_2,z,\tilde\xi,\eta,t\big)\le -\sigma
%\]
%for any
%\[
%z\in [v_2(t),\lambda v_3(t)+(1-\lambda)v_1(t)]
%\]
%and any
%\[
%\eta\in
%\left[
%v_t(x_2,t),
%\lambda v_t(x_3,t)+(1-\lambda)v_t(x_1,t)
%\right].
%\]
}
\begin{equation}\label{eq_ipos_sigma_trans}
 \partial_s \mathscr{B} + \mu \partial_{p^{n+1}} \mathscr{B} \leq -\sigma <0.
\end{equation}
Then
$$ \inf_{\Omega \times \Omega \times [0,1] \times (0,T]} \left( e^{-\mu (\star)}\mc{C}^*_v \right) \geq \frac{e^{-\mu t}}{\sigma} \mc{C}^*_{\mathscr{B}(\cdot, v(\cdot, \star), \tilde{\xi}, v_t(\cdot, \star), \star)} (x_1, x_3, \lambda, t).$$
As a consequence, set $\rho:=\min \left\{d(x_1,\partial\Omega),d(x_3,\partial \Omega)\right\}>0$, we have
\begin{equation}\label{eq_first_rho_appear}
\inf_{\Omega \times \Omega \times [0,1] \times (0,T]} \mathcal{C}^*_v \geq - \frac{e^{\mu T}}{\sigma} \sup_{\Omega_{\rho} \times \Omega_{\rho} \times [0,1] \times (0,T]} \left(\mc{C}^*_{\mathscr{B}(\cdot, v(\cdot, \star), \tilde{\xi}, v_t(\cdot, \star), \star)}\right)^- .
\end{equation}
%\tr{why is the sup taken on $\Omega_{\rho}$? In your applications, does one have have any precise information on $\rho$?}
\end{theorem}

\begin{proof}
Let us perform a drift in time and define 
\begin{equation}\label{eq_drift_trick}
w(x,t):= e^{-\mu t} v(x,t)
\end{equation}
which satisfies
\begin{align*}
 - \sum_{i,j=1}^n a_{ij}(e^{\mu t} D_x w, t) D^2_{ij} w 
 &= e^{-\mu t} \mathscr{B}\big(x, e^{\mu t} w, e^{\mu t} D_x w, e^{\mu t} (\mu w + w_t), t\big) \\
&=: \widehat{\mathscr{B}}(x,w, D w, t) \quad \hbox{ in $\Omega \times (0,T]$.}
\end{align*}
Notice that 
%\begin{equation}\label{eq_deriv_eq_B_Btilde_1}
%(\partial_{p^{n+1}}\mathscr{B)}(x,s,p,t) = (\partial_{p^{n+1}} \tr{\widehat{\mathscr{B}}}) \big(x, e^{-\mu t} s, e^{-\mu t} \tilde{p}, e^{-\mu t}(p_{n+1}-\mu s), t\big),
%\end{equation}
%\begin{equation}\label{eq_deriv_eq_B_Btilde_2}
%\begin{aligned}
%(\partial_s\mathscr{B)}(x,s,p,t) =& \, (\partial_s \tr{\widehat{\mathscr{B}}}) \big(x, e^{-\mu t} s, e^{-\mu t} \tilde{p}, e^{-\mu t}(p_{n+1}-\mu s), t\big) -\\
%& - \mu (\partial_{p^{n+1}} \tr{\widehat{\mathscr{B}}}) \big(x, e^{-\mu t} s, e^{-\mu t} \tilde{p}, e^{-\mu t}(p_{n+1}-\mu s), t\big).
%\end{aligned}
%\end{equation}
$$
\begin{aligned}
(\partial_{p^{n+1}}\widehat{\mathscr{B}})(x,s,p,t) = (\partial_{p^{n+1}} \mathscr{B}) \big(x, e^{\mu t} s, e^{\mu t} \tilde{p}, e^{\mu t}(\mu s+p_{n+1}), t\big),
\\
(\partial_s\widehat{\mathscr{B}})(x,s,p,t) = \, (\partial_s \mathscr{B} + \mu \partial_{p^{n+1}} \mathscr{B}) \big(x, e^{\mu t} s, e^{\mu t} \tilde{p}, e^{\mu t}(\mu s+p_{n+1}), t\big),
\end{aligned}
$$
thus assumptions \eqref{eq_ipos_dpnB} and \eqref{eq_ipos_sigma_trans} read
\begin{equation}\label{eq_deriv_eq_B_Btilde_1}
\partial_{p^{n+1}}\widehat{\mathscr{B}} \leq 0
\end{equation}
\begin{equation}\label{eq_deriv_eq_B_Btilde_2}
\partial_s\widehat{\mathscr{B}} \leq -\sigma <0.
\end{equation}
%\tr{Being
%$$\mathscr{B}(x,r,\zeta,t) = e^{\mu t} \tr{\widehat{\mathscr{B}}} \big(x, e^{-\mu t} r, e^{-\mu t} \tilde{\zeta}, e^{-\mu t}(\zeta_{n+1}-\mu r), t\big)$$
%we have
%}
We notice that $(x_1,x_3,\lambda, t)$ is an interior minimum for $\mc{C}^*_w$; since it is an interior critical point (with respect to the spatial variable), we obtain
$$\bar{\xi} := D_x w(x_1,t)=D_x w(x_2,t)=D_x w(x_3,t),$$
and thus \eqref{eq_def_xi*} holds, with $\tilde{\xi} = e^{\mu t} \bar{\xi}$. 
Exploiting that it is an interior minimum (thus the Hessian matrix with respect to the spatial variable is positive semidefinite) and the positive definiteness %definite positiviness 
of $a_{ij}(e^{\mu t} \bar{\xi}, t)$, by \cite[Lemma A.1]{kennington} we obtain 
$$\sum_{i,j=1}^n a_{ij}(e^{\mu t}\bar{\xi}, t) D^2_{ij} w(x_2, t) - \lambda \sum_{i,j=1}^n a_{ij}(e^{\mu t}\bar{\xi}, t) D^2_{ij} w(x_3, t) - (1-\lambda) \sum_{i,j=1}^n a_{ij}(e^{\mu t} \bar{\xi}, t) D^2_{ij} w(x_1, t) \geq 0$$
and thus, by the equation,
$$\widehat{\mathscr{B}}\big(x_2, w_2(t), \bar{\xi}, w_t(x_2,t), t\big) \leq \lambda \widehat{\mathscr{B}}\big(x_3, w_3(t), \bar{\xi}, w_t(x_3,t), t\big)+ (1-\lambda)\widehat{\mathscr{B}}\big(x_1, w_1(t),\bar{\xi}, w_t(x_1,t), t\big) .$$
Subtracting $\widehat{\mathscr{B}}\big(x_2, \lambda w_3(t)+(1-\lambda)w_1(t), \bar{\xi}, \lambda w_t(x_3,t) + (1-\lambda) w_t(x_1,t), t\big)$ on both sides and applying Lagrange theorem, we know that there exists $z\in [w_2(t),\lambda w_3(t)+(1-\lambda)w_1(t)]$ and $\zeta \in [w_t(x_2,t), \lambda w_t(x_3,t) + (1-\lambda) w_t(x_1,t)]$ such that
\begin{equation*}
\nabla_{s,p^{n+1}}\widehat{\mathscr{B}}(x_2, z, \bar{\xi}, \zeta, t) \cdot \big(\mc{C}^*_w(x_1,x_3,\lambda,t), \partial_t \mc{C}^*_w(x_1,x_3,\lambda, t)\big) \leq - \mc{C}^*_{\widehat{\mathscr{B}}(\cdot, w(\cdot, \star), \bar{\xi}, w_t(\cdot, \star), \star)} (x_1, x_3, \lambda, t).
\end{equation*}
Using that $(x_1,x_3,\lambda, t)$ is a minimum (with respect to the time variable), we obtain
$$\partial_t \mc{C}^*_w(x_1,x_3, \lambda, t) \leq 0,$$
which plugged in the previous inequality together with \eqref{eq_deriv_eq_B_Btilde_1} %\eqref{eq_ipos_dpnB} and \eqref{eq_deriv_eq_B_Btilde_1}}, % $\partial_{p^{n+1}} \mathscr{B} \leq 0$, 
implies 
\begin{equation*}
\partial_s \widehat{\mathscr{B}}(x_2, z, \bar{\xi}, \zeta, t) \mc{C}^*_w(x_1,x_3,\lambda,t) \leq - \mc{C}^*_{\widehat{\mathscr{B}}(\cdot, w(\cdot, \star), \bar{\xi}, w_t(\cdot, \star), \star)} (x_1, x_3, \lambda, t).
\end{equation*}
By %the definition of $\tr{\widehat{\mathscr{B}}}$, 
\eqref{eq_deriv_eq_B_Btilde_2} %\tr{\eqref{eq_ipos_sigma_trans}, \eqref{eq_deriv_eq_B_Btilde_2} and} 
and $\mc{C}^*_w(x_1,x_3,\lambda,t) <0$, we achieve
\begin{equation}\label{eq_pre_negative}
-\sigma \mc{C}^*_w(x_1,x_3,\lambda,t) \leq - \mc{C}^*_{\widehat{\mathscr{B}}(\cdot, w(\cdot, \star), \bar{\xi} %\tilde{\xi}
, w_t(\cdot, \star), \star)} (x_1, x_3, \lambda, t)
\end{equation}
and thus
$$ e^{-\mu t}\mc{C}^*_v(x_1,x_3,\lambda,t) \geq \frac{e^{-\mu t}}{\sigma} \mc{C}^*_{\mathscr{B}(\cdot, v(\cdot, \star), \tilde{\xi}, v_t(\cdot, \star), \star)} (x_1, x_3, \lambda, t)$$
which is the first claim. As a consequence
\begin{align*}
 \inf_{\Omega \times \Omega \times [0,1] \times (0,T]} \left( e^{-\mu (\star)}\mc{C}^*_v \right) 
 &\geq - \frac{e^{-\mu t}}{\sigma} \sup_{\Omega_{\rho} \times \Omega_{\rho} \times [0,1] \times (0,T]} \left(\mc{C}^*_{\mathscr{B}(\cdot, v(\cdot, \star), \tilde{\xi}, v_t(\cdot, \star), \star)}\right)^- \\
 &\geq - \frac{1}{\sigma} \sup_{\Omega_{\rho} \times \Omega_{\rho} \times [0,1] \times (0,T]} \left(\mc{C}^*_{\mathscr{B}(\cdot, v(\cdot, \star), \tilde{\xi}, v_t(\cdot, \star), \star)}\right)^-
 \end{align*}
 and hence the second claim, observed that (recall that $ \inf_{\Omega \times \Omega \times [0,1] \times (0,T]} \left( e^{-\mu (\star)}\mc{C}^*_v \right)<0$)
 \[ \inf_{\Omega \times \Omega \times [0,1] \times (0,T]} \mathcal{C}^*_v = \inf_{\Omega \times \Omega \times [0,1] \times (0,T]} \left( e^{\mu (\star)} e^{-\mu (\star)}\mathcal{C}^*_v\right) \geq e^{\mu T} \inf_{\Omega \times \Omega \times [0,1] \times (0,T]} \left(e^{-\mu (\star)}\mathcal{C}^*_v\right) . \qedhere \]
\end{proof}

\begin{remark}
We observe that the quantity $\rho$ given in Theorem \ref{thm_conc_princ_eigenf} and appearing in \eqref{eq_first_rho_appear} is generally not known: indeed, it is related to the localization of the minimum of the concavity function, which is not an easy task even in simple elliptic equations in simple domains. Since, without any additional information on the equations and on the boundary conditions, the function $\mathscr{B}(\cdot, v(\cdot, \star), \tilde{\xi}, v_t(\cdot, \star), \star)$ may explode on the boundary, the presence of such $\rho$ allows to deduce anyway a quantitative estimate. In some models, anyway, the dependence on $\rho$ could be omitted by exploiting the trivial inequality 
$$ \sup_{\Omega_{\rho} \times \Omega_{\rho} \times [0,1] \times (0,T]} \left(\mc{C}^*_{\mathscr{B}(\cdot, v(\cdot, \star), \tilde{\xi}, v_t(\cdot, \star), \star)}\right)^- \leq \sup_{\Omega \times \Omega \times [0,1] \times (0,T]} \left(\mc{C}^*_{\mathscr{B}(\cdot, v(\cdot, \star), \tilde{\xi}, v_t(\cdot, \star), \star)}\right)^- $$
see for instance Corollary \ref{cor_korevaar_perturbed} and Theorem \ref{approximate concavity parabolic application a(x)u introduction}.
\end{remark}

\begin{remark}
\label{rem_heat_eq_case}
If $\mathscr{B}(x,v,D v, t) \equiv \Btilde(x,v,D_x v, t) - v_t$, as in the heat equation, we see that $\partial_{p^{n+1}} \mathscr{B} = - 1 \leq 0$, while 
$$\partial_s \mathscr{B} + \mu \partial_{p^{n+1}} \mathscr{B} \equiv \partial_s\Btilde - \mu \leq \Lambda - \mu \leq -\sigma<0$$
is satisfied if $\partial_s \Btilde$ is bounded from above by some $\Lambda \in \R$ and $\mu$ is sufficiently large ($\mu>\Lambda$).
\end{remark}

We show now a maximum principle which allows to relate the interior minimum of the concavity function of a solution $v$ to the \emph{harmonic} concavity function of the source. 
Notice that here we need the source to enjoy some strict monotonicity in $v$, see \eqref{strictly decreasing eqt} (differently from Theorem \ref{thm_conc_princ_eigenf}, compare with \eqref{eq_ipos_sigma_trans}). 
In these results, we look at some harmonic concavity (rather than concavity) of $\mathscr{B}$, and we study concavity of $u$ in $(x,t)$ (rather than concavity in $x$ for each $t$).

\begin{theorem}[Harmonic concavity maximum principle]
\label{Concavity perturbative maximum principle for parabolic equations} 
	Let $v\in C^2_x(\Omega) \cap C^1_t((0,+\infty))$ be such that 
	\begin{equation*}
- \sum_{i,j=1}^{n}a_{ij}(Dv)D^2_{ij}v = \mathscr{B}(x,v,Dv, t) \quad \hbox{in $\Omega \times (0,+\infty)$,}
	\end{equation*}
	where $\mathscr{B}\colon\Omega\times \R\times \R^{n+1} \times (0,+\infty)\to \R $ and $(a_{ij}(p))_{ij}$ is symmetric and positive definite for all $p\in \R^{n+1}$. 
	Assume that $s \in v(\Omega\times [0,+\infty)) \mapsto \mathscr{B}(x,s,p,t)$ is differentiable 	for all $(x,p,t)\in \Omega\times \R^{n+1} \times (0,+\infty)$.
	Suppose $\mathcal{C}_v$ admits a negative interior minimum at $(x_1,x_3,t_1,t_3,\lambda)\in \Omega \times \Omega \times (0,+\infty)\times (0,+\infty)\times [0,1]$. 
	Then
	\begin{equation}\label{eq_def_xi}
	\xi:=Dv(x_1,t_1)=Dv(x_2,t_2)=Dv(x_3,t_3).
	\end{equation}
	Assume moreover that there exists $\sigma>0$ such that
	\begin{equation}\label{strictly decreasing eqt}
		\sup_{z\in [v(x_2,t_2),\lambda v(x_3,t_3)+(1-\lambda)v(x_1,t_1)]} \partial_s\mathscr{B}(x_2,z,\xi, 	t_2)\leq-\sigma.
	\end{equation}
	Then
\begin{equation}\label{eq_harm_concav_quantit}
\inf_{\Omega\times \Omega \times (0,+\infty)\times (0,+\infty)\times[0,1]}\mathcal{C}_v\geq \dfrac{1}{\sigma} \mathcal{HC}_{\mathscr{B}(\cdot, v(\cdot,\star), \xi, \star)}(x_1,x_3,t_1,t_3,\lambda). 
\end{equation}
\end{theorem}

\begin{proof}
Notice first that $(x_1,t_1)\ne (x_3,t_3)$ and $\lambda\in (0,1)$, otherwise $\mathcal{C}_v(x_1,x_3,t_1,t_3,\lambda)=0$ which is a contradiction. 
Being $(x_1,x_3,t_1,t_3,\lambda)$ in the interior of the domain we obtain
			\[ D\mathcal{C}_v(x_1,x_3,t_1,t_3,\lambda)=0 \]
		and in particular $D_{y_1}\mathcal{C}_v = D_{y_3}\mathcal{C}_v = D_{\tau_1}\mathcal{C}_v = D_{\tau_3}\mathcal{C}_v = 0$ imply \eqref{eq_def_xi}. 
		Set now $A:=(a_{ij}(\xi))_{ij}$ and define the $2n\times 2n$ matrices 
		\[ C:= 
		\begin{bmatrix}
			D^2_{y_1, y_1}\mathcal{C}_v(x_1,x_3, t_1,t_3,\lambda) & 	 D^2_{y_1, y_3}\mathcal{C}_v(x_1,x_3, t_1,t_3,\lambda)\\
			D^2_{y_1, y_3}\mathcal{C}_v(x_1,x_3, t_1,t_3,\lambda) & 	 D^2_{y_3, y_3}\mathcal{C}_v(x_1,x_3, t_1,t_3,\lambda)
		\end{bmatrix}, \]
		which is positive semidefinite since $(x_1,x_3, t_1,t_3,\lambda)$ is a minimum for $\mathcal{C}_v$ and 
		\[ B:= \begin{bmatrix}
			s^2a_{ij}(\xi) & sra_{ij}(\xi) \\
			sra_{ij}(\xi) & r^2a_{ij}(\xi) 
		\end{bmatrix}, \]
		which is positive semidefinite by hypothesis. So by \cite[Lemma A.1]{kennington} we have $\Tr(BC)\geq 0$, that is
		\begin{equation}\label{tr(BC)<0, caso approssimato}
			\alpha s^2 +2\beta sr +\gamma r^2\geq 0 \quad \hbox{for all $s,r\in \R$} 
		\end{equation}
		where 
		\[ \alpha := \Tr (AD^2_{y_1,y_1}\mathcal{C}_v), \quad \beta:= \Tr (AD^2_{y_1,y_3}\mathcal{C}_v),\quad \gamma:= \Tr (AD^2_{y_3,y_3}\mathcal{C}_v),\]
		which rewrite as 
$$
				\alpha =(1-\lambda)^2Q_2-(1-\lambda)Q_1, \qquad \beta= \lambda(1-\lambda)Q_2, \qquad \gamma= \lambda^2Q_2-\lambda Q_3, 
$$
		where we have set for $k\in\left\{1,2,3\right\}$
		\[ Q_k:=\sum_{i,j=1}^na_{ij}D^2_{ij}v(x_k,t_k).\]
		By \eqref{tr(BC)<0, caso approssimato} we obtain $\alpha\geq0, \gamma\geq0$ and $\beta^2-\alpha \gamma\leq0$, i.e. 
		\begin{align}
			Q_2&\geq \dfrac{Q_1}{1-\lambda}\label{relation Q_1 II approssimato},\\
			Q_2&\geq \dfrac{Q_3}{\lambda}\label{relation Q_3 II approssimato},\\
			Q_1Q_3&\geq Q_2((1-\lambda)Q_3+\lambda Q_1)\label{relation Q_1Q_3 II approssimato},
		\end{align}
		where we justify the last by 
		\[
		\begin{split}
			0&\geq \lambda^2(1-\lambda)^2Q_2^2-((1-\lambda)^2Q_2-(1-\lambda)Q_1)(\lambda^2 Q_2-\lambda Q_3)\\
			& =\lambda(1-\lambda)((1-\lambda)Q_2Q_3+\lambda Q_2Q_1-Q_1Q_3).
		\end{split}
		\]
		\noindent\textit{Claim}: $(1-\lambda)Q_3+\lambda Q_1<0$ or $Q_1=Q_3=0$.
		
		If $(1-\lambda)Q_3+\lambda Q_1\geq0$, then by \eqref{relation Q_1 II approssimato} and \eqref{relation Q_1Q_3 II approssimato}
		\begin{align*}
			Q_2 \big( (1-\lambda)Q_3+\lambda Q_1 \big)&\geq \frac{Q_1}{1-\lambda}\big( (1-\lambda)Q_3+\lambda Q_1\big),\\
			Q_2\big( (1-\lambda)Q_3+\lambda Q_1 \big)&\geq Q_1Q_3+\frac{\lambda}{1-\lambda}Q_1^2,\\
			Q_2\big( (1-\lambda)Q_3+\lambda Q_1 \big)&\geq 	Q_2\big( (1-\lambda)Q_3+\lambda Q_1 \big)+\frac{\lambda}{1-\lambda}Q_1^2,
		\end{align*}
		which imply $Q_1=0$. Similarly (exploiting \eqref{relation Q_3 II approssimato} and \eqref{relation Q_1Q_3 II approssimato}) we have $Q_3=0$.
		Then if $(1-\lambda)Q_3+\lambda Q_1<0$ and from \eqref{relation Q_1Q_3 II approssimato}, it holds
		\[ Q_2\geq \dfrac{Q_1Q_3}{(1-\lambda)Q_3+\lambda Q_1}. \]
Being $Q_k = -\mathscr{B}\big(x_k,v_k, \xi, t_k\big)$, we get 
	\begin{align*}
		-\mathscr{B}\big(x_2,v_2,\xi, t_2\big)\geq - \dfrac{\mathscr{B}\big(x_1,v_1,\xi, t_1\big)\mathscr{B}(x_3,v_3,\xi, t_3\big)}{(1-\lambda)\mathscr{B}\big(x_3,v_3,\xi, t_3\big)+\lambda \mathscr{B}\big(x_1,v_1,\xi, t_1\big)}.
	\end{align*}
	Adding on both sides $\mathscr{B}\big(x_2,\lambda v_3+(1-\lambda)v_1,\xi, t_2\big)$ and applying Lagrange theorem, we obtain
$$-\partial_s\mathscr{B}(x_2,z,\xi, t_2)\mathcal{C}_v(x_1,x_3,t_1,t_3,\lambda)\geq \mathcal{HC}_{\mathscr{B}(\cdot, v(\cdot,\star),\xi, \star)}(x_1,x_3,t_1,t_3,\lambda)$$
	for some $z\in [v_2,\lambda v_3+(1-\lambda)v_1]$. 
	Otherwise if $Q_1=Q_3=0$, then $Q_2\geq 0$ and by the same arguments as before it holds
	\[		-\partial_s\mathscr{B}(x_2,z,\xi, t_2)\mathcal{C}_v(x_1,x_3,t_1,t_3,\lambda)\geq \mathcal{HC}_{\mathscr{B}(\cdot, v(\cdot,\star),\xi,\star)}(x_1,x_3,t_1,t_3,\lambda).\]
	Since $\mathcal{C}_v(x_1,x_3,t_1,t_3,\lambda)<0$ and $\partial_s \mathscr{B}(x_2,z,\xi, t_2)\leq-\sigma$, in any case it follows that
	\[	\sigma \mathcal{C}_v(x_1,x_3,t_1,t_3,\lambda)\geq -\partial_s\mathscr{B}(x_2,z,\xi,t_2)\mathcal{C}_v(x_1,x_3,t_1,t_3,\lambda)\geq \mathcal{HC}_{\mathscr{B}(\cdot, v(\cdot,\star),\xi,\star)}(x_1,x_3,t_1,t_3,\lambda), \]
and thus the claim.
\qedhere	
\end{proof}

We see that, in Theorem \ref{Concavity perturbative maximum principle for parabolic equations}, the nonlinear term $\mathscr{B}$ actually includes $v_t$ 
in any form. In the special case of heat type equations, we can remove the dependence on the right side of \eqref{eq_harm_concav_quantit} from $ v_t$ by assuming $v$ monotone.
\begin{corollary}\label{cor_discard_vt}
In the setting of Theorem \ref{Concavity perturbative maximum principle for parabolic equations}, assume that $\mathscr{B}$ has the form
$$\mathscr{B}(x, v, Dv, t) \equiv \Btilde(x,v, D_x v, t) - k(t) v_t$$
with $\mc{HC}_{k} \leq 0$ (for instance, $k$ is constant). 
Assume moreover that $\xi^{n+1} \geq 0$ (recall $\xi^{n+1}= v_t(x_1,t_1)= v_t(x_2,t_2)= v_t(x_3,t_3)$) and $k>0$. Then
		\[ \inf_{\Omega\times \Omega \times (0,+\infty)\times (0,+\infty)\times[0,1]}\mathcal{C}_v\geq \dfrac{1}{\sigma} \mathcal{HC}_{\Btilde(\cdot, v(\cdot,\star), \tilde{\xi},\star)}(x_1,x_3,t_1,t_3,\lambda). \]
\end{corollary}
\begin{proof}
It is sufficient to apply Theorem \ref{Concavity perturbative maximum principle for parabolic equations} and Proposition \ref{proposition f-k and harmonic concavity function} (notice that if $\xi^{n+1}=0$ the claim is obvious).
\end{proof}

\begin{remark}
We notice in Corollary \ref{cor_discard_vt} the essential role of $ v_t$ being constant on $(x_1,t_1)$, $(x_2,t_2)$, $(x_3,t_3)$. This is a consequence of the fact that we study the space-time concavity in $(0,+\infty)$ (and not only space concavity in $(0,T]$).
\end{remark}

%%%%%%%%%%%%%%%%%%%%%%%%%%%%%%%%%%%%%%%%%%%%%%%%%%%%%%%%%%%%%%%%%%%%%%%%
%%%%%%%%%%%%%%%%%%%%%%%%%%%%%%%%%%%%%%%%%%%%%%%%%%%%%%%%%%%%%%%%%%%%%%%%
\section{Maximum principles with boundary conditions}
\label{sec_gen_eigen}

In Section \ref{sec_concav_max_princip} we dealt with the value of the concavity function in the interior of the domain. In this section we focus on the behaviour on the boundary: to this aim, we consider a Dirichlet boundary condition and a Cauchy initial datum. 
For the sake of simplicity, we focus now on equations of heat type, that is with main operator $\partial_t - \Delta$.

%%%%%%%%%%%%%%%%%%%%%%%%%%%%%%%%%%%%%%%%%%%%%%%%%%%%
%%%%%%%%%%%%%%%%%%%%%%%%%%%%%%%%%%%%%%%%%%%%%%%%%%%%
\subsection{Concavity} 

In this subsection we apply Theorem \ref{thm_conc_princ_eigenf} to $v$, transformation of a solution $u$ of a Cauchy-Dirichlet problem.

\begin{theorem}[Estimate on log-concavity]
\label{thm_conc_princ_bound}
	Let $\Omega \subset \R^n$, $n\geq 2$, be a bounded, smooth, strongly convex domain. Let $T>0$, $b\colon \Omega \times (0,+\infty) \times (0,T]\to \R$, $b(x,\cdot, t)$ differentiable for all $(x,t)\in \Omega\times (0,T]$, $u_0\in C^1(\Omegabar)$ 	with $u_0=0$ on $\partial\Omega$ and $u \in C^2_x(\Omegabar) \cap 
	C^1(\Omegabar\times[0,T])$, be such that
$$u_t - \Delta u = b(x,u, t) \quad \hbox{in $\Omega \times (0,T]$}$$ 
with
	\begin{equation}\label{boundary T and initial conditions}
		\begin{cases}
				u>0 & \text{in } \Omega\times (0,T],\\
				u=0 & \text{on }\partial \Omega \times[0,T],\\
					D_x u \cdot \nu >0 & \text{on }\partial \Omega\times [0,T], \\
					u =u_0 & \text{on } \Omegabar \times \{0\}.
		\end{cases}
	\end{equation}
Assume there exists $\Lambda \geq 0$ %\in \R$ 
such that%
\footnote{
If $\Lambda \in [-\frac{1}{T},0]$ the claim holds as well. If $\Lambda<-\frac{1}{T}$, then the coefficient on the right hand side does not depend on $T$ and equals $-\frac{1}{|\Lambda|}$.
}
$$\partial_s \big(e^{-s} b(x,e^{s}, t)\big) \leq \Lambda \quad \hbox{for each $(x,s, t) \in \Omega \times \R \times (0,T]$}.$$
Assume moreover that $u_0>0$ is $\log$-concave in $\Omega$. 
Then there exists $\rho>0$ such that
\begin{equation}\label{eq_conc_princ_bound}
\inf_{\Omega \times \Omega \times [0,1] \times [0,T]} \mc{C}^*_{\log(u)} \geq -T e^{1+\Lambda T} \sup_{\Omega_{\rho} \times \Omega_{\rho} \times [0,1] \times (0,T]} \left(\mc{C}^*_{\frac{b(\cdot, u(\cdot, \star), \star)}{u(\cdot,\star)} }\right)^-.
\end{equation}
\end{theorem}

\begin{proof}
We see that $v:=\log(u)$ solves (exploiting the notation in Remark \ref{rem_heat_eq_case})
$$v_t - \Delta v = e^{-v} b(x, e^v, t) + |D_x v|^2 =: \Btilde(x,v,D_x v, t) \quad \hbox{in $\Omega \times (0,T]$}$$
with $\partial_s \Btilde \leq \Lambda$. 
If $ \mc{C}^*_v \geq 0$ the claim holds. Otherwise, assume that $ \mc{C}^*_v$ is somewhere negative. 
Let $\mu > \Lambda$. 
%By the assumption on the initial condition, and by the assumption on the boundary coupled with \cite[Lemma 2.4]{Korevaar}, the function $e^{-\mu t}\mc{C}_v^*(x_1,x_3,\lambda, t)$ cannot get negative for $t \to 0$ nor $(x_1, x_3) \to \partial(\Omega \times \Omega)$\tr{; notice that \cite{Korevaar} deals with elliptic equations, but the cited result requires only condition on $u$ on the boundary of $\Omega$}. 
By the assumption on the initial condition, the function $e^{-\mu t}\mc{C}_v^*(x_1,x_3,\lambda, t)$ cannot get negative for $(x_1,x_3,\lambda,t) $ approaching $\overline{\Omega} \times \overline{\Omega}\times[0,1]\times\{0\}$. Let now by contradiction $(x_1^n,x_3^n,\lambda^n, t^n)$ be a minimizing sequence for $e^{-\mu t}\mc{C}_v^*$ approaching $\partial (\Omega \times \Omega) \times[0,1]\times(0,T]$; up to a subsequence $t_n \to t \in (0,+\infty)$. By $D_x u \cdot \nu >0 $ we easily obtain $u(x_i^n, t^n)/u(x_i^n, t) \to 1$, i.e. $v(x_i^n, t^n) - v(x_i^n, t) \to 0$; this ensures that $(x_1^n,x_3^n,\lambda^n, t)$ is a minimizing sequence as well. Consider $u(\cdot, t)$: %for each $t>0$, 
by the assumption on the boundary coupled with \cite[Lemma 2.4]{Korevaar} (see also \cite[Proposition 4.6]{GaSq24}), the function $e^{-\mu t}\mc{C}_v^*(x_1^n,x_3^n,\lambda^n, t)$ is positively far from zero, which is an absurd. 
%By continuity, $e^{-\mu t}\mc{C}_v^*(x_1^n,x_3^n,\lambda^n, t^n)$ cannot get negative as well.} \tor{M: C'erano delle imprecisioni e mancanze nella versione precedente. Da ricontrollare. E' corretto? Serve una qualche uniformità rispetto a $t$? Il "by continuity" è vago}
%the same function cannot get negative for $(x_1, x_3) \to \partial(\Omega \times \Omega)$.
%\tr{ATT! Potrei anche avere $(x_1^n, x_3^n, t^n)$ e quindi...}}

Then $e^{-\mu t}\mc{C}_v^*(x_1,x_3,\lambda, t)$ must admit an interior negative minimum. 
Thus, by Theorem \ref{thm_conc_princ_eigenf}, we obtain %\tor{M: ricontrolla legame tra $\sigma$ e $\mu-\Lambda$}
$$
 \inf_{\Omega \times \Omega \times [0,1] \times (0,T]} \mathcal{C}^*_v \geq - \frac{e^{\mu T}}{\mu-\Lambda} %\Lambda-\mu}
 \sup_{\Omega_{\rho} \times \Omega_{\rho} \times [0,1] \times (0,T]} \left(\mc{C}^*_{e^{-v(\cdot, \star)} b(\cdot, e^{v(\cdot, \star)}, \star)}\right)^- .
$$
Since the choice of $\mu>\Lambda$ is arbitrary, we maximize the right hand side (which is the case for $\mu=\frac{1}{T}+\Lambda$) and conclude.
\end{proof}

\begin{remark}\label{rem_u0_nonconcav}
If in Theorem \ref{thm_conc_princ_bound} we do not assume $u_0$ to be $\log$-concave, then we take into account also $ \mathcal{C}_{\log(u_0)}$, obtaining
$$\inf_{\Omega \times \Omega \times [0,1] \times [0,T]} \mc{C}^*_{\log(u)} \geq -\max\left \{T e^{1+\Lambda T} \sup_{\Omega_{\rho} \times \Omega_{\rho} \times [0,1] \times (0,T]} \left(\mc{C}^*_{\frac{b(\cdot, u(\cdot, \star), \star)}{u(\cdot,\star)} }\right)^-, \sup_{\Omega \times \Omega \times [0,1]} \mathcal{C}_{\log(u_0)}^- \right\}.$$
Indeed, as in the previous proof, the negative infimum is far from $\partial(\Omega \times \Omega) \times [0,1]\times (0,T]$; if there is an interior minimum, then the previous reasoning applies and \eqref{eq_conc_princ_bound} holds. If instead the infimum is realized for $(x_1,x_3,\lambda,t) \to \overline{\Omega} \times \overline{\Omega}\times[0,1]\times\{0\}$, then it means
$$\inf_{\Omega \times \Omega \times [0,1] \times [0,T]} \mc{C}^*_{\log(u)} \geq %=
 \inf_{\Omega \times \Omega \times [0,1]} \left(-\mathcal{C}_{\log(u_0)}^- \right).$$
Joining the two possibilities, we obtain the claim. %\tor{M: Ricontrollare}
\end{remark}

\begin{corollary}[Exact log-concavity]
\label{corol_concav_strict}
In the assumptions of Theorem \ref{thm_conc_princ_bound} assume moreover that
\begin{itemize}
\item $(s,x) \mapsto s^{-1} b(x,s,t)$ is concave for each $t>0$.
\end{itemize}
Then $u(\cdot, t)$ is $\log$-concave for each $t\in (0,T)%>0
$. Moreover if $b\in C^{2,1}(\Omega\times(0,+\infty)\times [0,T))$ and $u \in C^{2,1}(\Omega \times [0,T))$, then $u(\cdot, t)$ is strictly $\log$-concave for each $t\in (0,T)%>0
$.
\end{corollary}

\begin{proof}
By Theorem \ref{thm_conc_princ_bound} we obtain that $v=\log(u)$ is concave. 
Arguing as in Remark \ref{rem_appl_const_rank} and applying Proposition \ref{prop_const_rank} to $w=-v$, we obtain the strict concavity.
\end{proof}

%%%%%%%%%%%%%%%%%%%%%%%%%%%%%%%%%%%%%%%%%%%%%%%%%%%%
\subsubsection{Applications}

%We start observing the following.

\begin{proof}[Proof of Theorem \ref{thm_general_korevaar} and Theorem \ref{thm_populat}]
They are a consequence of Corollary \ref{corol_concav_strict} and Proposition \ref{alpha concavity property} (\textit{iv}).
\end{proof}

\begin{remark}\label{rem_direct_strict_conc}
We deduced strict concavity in Theorem \ref{thm_general_korevaar} thanks to the constant rank theorem. 
In the eigenfunction case $b(x,s,t)=a(x,t)s$, by assuming $a(\cdot, t)$ strictly concave for each $t>0$, we can reach the same claim in a more straightforward way.

Indeed, assume by contradiction that there exists $t>0$ and $x_1 \neq x_3$, $\lambda \in (0,1)$ such that $\mc{C}^*_{\log(u)}(x_1, x_3, \lambda, t) = 0$. Recalled that we already proved $\mc{C}^*_{\log(u)}\geq 0$, 
in the notations of the proof of Theorem \ref{thm_conc_princ_eigenf} -- fixed a whatever $\mu>0$ -- we have $w=e^{-\mu t} v = e^{-\mu t} \log(u)$, thus $\mc{C}^*_{w} \geq 0 = \mc{C}^*_w(x_1,x_3,\lambda, t)$, and hence $(x_1,x_3,\lambda, t)$ is a point of minimum for $\mc{C}^*_w$. Thus by \eqref{eq_pre_negative} 
we obtain
$$ 0 \leq - \mc{C}^*_{\Btilde(\cdot, w(\cdot, \star), \xi, \star)} (x_1, x_3, \lambda, t) 
= - e^{-\mu t} \mc{C}^*_{a(x,t) } (x_1,x_3,\lambda, t) <0, $$
thanks to the strict concavity of $a$ and the assumption $x_1 \neq x_3$, $\lambda \in (0,1)$.
 This contradiction concludes the proof. 
\end{remark}

\begin{remark}
		\label{rem_gradient_type}
		We mention that Theorem \ref{thm_conc_princ_eigenf} has interesting applications also in gradient-type equations. For instance, consider
		the evolutionary equation driven by the %Fokker-Planck equation (with 
		Ornstein–Uhlenbeck operator
		$$\partial_t u - \Delta u = -x \cdot \nabla u \quad \hbox{in $\Omega \times (0,+\infty)$}$$
		with $u=0$ on $\partial \Omega \times (0,+\infty)$ and $u=u_0$ on $\Omega \times \{t=0\}$. 
		Recently (after the submission of the present paper) the authors in \cite{CQS26} showed that every solution 
		%of the evolutionary equation driven by the %Fokker-Planck equation (with 
		%Ornstein–Uhlenbeck operator
		%$$\partial_t u - \Delta u = x \cdot \nabla u \quad \hbox{in $\Omega \times \{(0,+\infty)$}$$
		%with $u=0$ on $\partial \Omega \times (0,+\infty)$ and $u=u_0$ on $\Omega \times \{t=0\}$ 
		is log-concave in $x$ for each $t$, as long as $u_0$ is chosen $\log$-concave and $\Omega$ convex. We see that actually we can recover this result if we additionally require $\Omega$ strongly convex and $u_0 \in C^1(\overline{\Omega})$ with $D_x u_0 \cdot \nu >0 $ on $\partial \Omega$ (see \eqref{boundary T and initial conditions} and Lemma \ref{Hopf lemma in parabolic case}). Indeed, reasoning on $[0,T]$ for arbitrary $T$ fixed, following Remark \ref{rem_heat_eq_case},
		 we have
		$\Btilde(x,s,\tilde{p}, t) = - x \cdot \tilde{p}+ |\tilde{p}|^2$, and hence $\partial_s \Btilde=0$. Moreover, $(x,s) \mapsto \Btilde(x,s,\tilde{p}, t)$ is concave. Arguing as in Theorem \ref{thm_conc_princ_bound} we have that the minimum is in the interior, thus Theorem \ref{thm_conc_princ_eigenf} implies the claim. 
		Differently from \cite{CQS26}, we do not rely on an integral representation of the solution, thus the result can be easily extended to %generalized Ornstein–Uhlenbeck operator and 
		more general drift Laplacians, e.g. of the form $\Delta u -V(x,t) \cdot \nabla u$. 
		%https://www.math.uni-bielefeld.de/~grigor/wma.pdf
		%\tor{M: ricontrollare}
	\end{remark}

We focus now on perturbed concavity as consequence of Theorem \ref{thm_conc_princ_bound}. 
Exploiting it together with Propositions \ref{concavità prodotto} and \ref{prop_extram_cases} we obtain the following statement. Notice that $\{\partial_r(e^{-r} f(e^r)) \mid r \in \R\} \equiv \{s \partial_s( s^{-1} f(s)) \mid s>0\} $.

\begin{corollary}\label{cor_korevaar_perturbed}
	Let $\Omega \subset \R^n$, $n\geq 2$, be a bounded, smooth, strongly convex domain. Let $T>0$, $a\colon\Omega\times(0,T]\to \R$ locally bounded, $f\colon(0,+\infty)\to\R$ differentiable, $u_0\in C^1(\Omegabar)$ with $u_0=0$ on $\partial\Omega$ and $u \in C^2_x(\Omegabar) \cap C^1%C
	(\Omegabar\times [0,T])$ be such that
$$u_t - \Delta u = a(x,t) f(u)\quad \hbox{in $\Omega \times (0,T]$}$$
with \eqref{boundary T and initial conditions}. 
Set $\bar{f}(s):= \frac{f(s)}{s} $ for $s>0$ and assume
$$\Lambda:=\sup_{(x,s,t) \in \Omega \times(0,+\infty)\times (0,T]}  %  \sup_{s >0}
 \big(a(x,t) s \bar{f}'(s) \big) < + \infty.$$ %\in \R.$$
Assume moreover that $u_0>0$ is $\log$-concave in $\Omega$. 
Then there exists $\rho>0$ such that the following facts hold.
\begin{enumerate}[label=\textit{\roman*})]
\item If $\bar{f}$ is constant, then 
$$\inf_{\Omega \times \Omega \times [0,1] \times [0,T]} \mc{C}^*_{\log(u)} \geq -T e^{1+\Lambda T} \sup_{\Omega_{\rho} \times \Omega_{\rho} \times [0,1] \times (0,T]} \left(\mc{C}^*_{a(\cdot, \star)}\right)^-.$$

\item If $\bar{f}$ is $\theta'$-concave, with $\theta'>1$, and $0< m\leq a \leq M$ with $m^{\theta} \geq \frac{1}{2} M^{\theta}$, $\frac{1}{\theta} + \frac{1}{\theta'}=1$, then 
$$\inf_{\Omega \times \Omega \times [0,1] \times [0,T]} \mc{C}^*_{\log(u)} 
\geq -T e^{1+\Lambda T}
 \norm{\bar{f}(u)}_{L^{\infty}(\Omega_{\rho}\times (0,T])}
 \sup_{\Omega_{\rho} \times \Omega_{\rho} \times [0,1] \times (0,T]} 
 \left( \left(\mc{C}^*_{a^{\theta}(\cdot, \star)}\right)^-\right)^{\frac{1}{\theta}}.
 $$
 
 \item If $\bar{f}$ is concave and $a$ is bounded and nonnegative, then 
 $$\inf_{\Omega \times \Omega \times [0,1] \times [0,T]} \mc{C}^*_{\log(u)} \geq 
 -T e^{1+\Lambda T} 
 \norm{\bar{f}(u)}_{L^{\infty}(\Omega_{\rho}\times (0,T])}
 \osc(a).
 $$
\end{enumerate}
\end{corollary}
We note that Corollary \ref{cor_korevaar_perturbed} applies to the following sources $f(s)=s$ (case \textit{i}), $f(s)=s\log^q(1+s)$, $f(s)=\frac{s^{q+1}}{1+s^q}$ with $q\in(0,1)$ (case \textit{ii}), $f(s)=\frac{s^2}{1+s}$, $f(s)=s\log(1+s)$, $f(s)=s\log(s)$ (case \textit{iii}).

\begin{proof}[Proof of Theorem \ref{approximate concavity parabolic application a(x)u introduction}]
It is a particular case of Corollary \ref{cor_korevaar_perturbed}.
\end{proof}

\begin{remark}\label{rem_converg_logarithm}
		Consider $b(x,s,t)=s \log(s)$. In \cite{Gallo-Mosconi-Squassina} the authors show the existence of a $\log$-concave solution for the stationary problem; here, in Theorem \ref{thm_general_korevaar}, we show the $\log$-concavity of the evolutionary solution in each $[0,T]$, whenever such solution exists. 
		We discuss the possible (pointwise) convergence of $u(x,t)$ to a stationary solution as $t\to +\infty$ and, thus, the possibility of achieving the result in \cite{Gallo-Mosconi-Squassina} through a parabolic approach. We start noticing that the assumptions in Remark \ref{stability_parabolic condition u^gamma} are not fulfilled.
		Let $u_0$ be a $\log$-concave function. 
%		If $I(u_0):=|\nabla u_0|_2^2 - \int_{\Omega} u_0^2 \log(u_0) <0$, then by \tr{\cite[Theorem 1.1]{HAN2019}} we obtain that $u(x,t) \to +\infty$ as $t \to +\infty$. 
%		If $I(u_0)=0$, then exploiting the logarithmic Sobolev inequality and arguing as in \cite[Lemma 2.3]{wangwang2024}, we obtain $|u_0|_2^2 \geq (2\pi)^{n/2} e^n$; if equality holds, then by \tr{\cite[Theorem 1.2]{CLL15}} we have $u(x,t) \to 0$ as $t \to +\infty$. 
%		The cases $I(u_0)=0$ with $|u_0|_2^2 > (2\pi)^{n/2} e^n$, and $I(u_0)>0$ are -- up to the authors' knowledge -- still open problems. 
If $I(u_0):=|\nabla u_0|_2^2 - \int_{\Omega} u_0^2 \log(u_0) <0$, then by \cite[Theorem 1.1]{HAN2019} we obtain that $u(x,t) \to +\infty$ as $t \to +\infty$. 
If $I(u_0)>0$, and $J(u_0):=I(u_0)+\frac{1}{4} \norm{u_0}_2^2\leq \frac {1}{4} (2\pi)^{n/2} e^n=:M$, then by \cite[Theorems 1.1 and 1.2]{CLL15} we have $u(x,t) \to 0$ as $t \to +\infty$.
If $I(u_0)=0$, then by the logarithmic Sobolev inequality and arguing as in \cite[Lemma 2.3]{wangwang2024}, we obtain $J(u_0) \geq M$; if $J(u_0)=M$, then \cite[Theorem%s 
1.2]{CLL15} implies $u \in L^{\infty}((0,+\infty), H^1_0)$, even if it is not known if the solution decays to zero or not.
The cases $I(u_0)>0$ and $J(u_0)>M$, or $I(u_0)=0$ and $J(u_0)> M$ are -- up to the authors' knowledge -- still open problems. 
		We further observe that, similarly to  Corollary \ref{corol_monotonicity}, if $u_0$ is also a subsolution of the stationary problem (in particular, $I(u_0) \leq 0$), then we have that $u(x, \cdot)$ is nondecreasing, thus a limit must exist.
By the above considerations, we have no examples at hand in which the convergence to the stationary solution can be pursued.

\end{remark}

\begin{example}
As particular cases of concave $a$ we may consider $a(x,t)=t^{\gamma} d_\Omega(x)^\omega $, $\gamma, \omega \in [0,1]$, with $\gamma+\omega=1$ (see Proposition \ref{alpha concavity property}). 
In such a case the solutions of
$$ u_t - \Delta u = t^{\gamma} d(x,\partial \Omega)^{\omega} u$$
are $\log$-concave.
\end{example}

As an additional consequence of Theorem \ref{thm_conc_princ_bound}, we may consider also population models \cite{CaCo89}.
\begin{corollary}
\label{corol_popul_quantit}
	Let $\Omega \subset \R^n$, $n\geq 2$, be a bounded, smooth, strongly convex domain. 
	Let $T>0$, $a\colon\Omega\times(0,T]\to \R$ bounded, $u_0\in C^1(\Omegabar)$ with $u_0=0$ on $\partial\Omega$ and $u \in C^2_x(\Omegabar) \cap C^1_t((0,T]) \cap C(\Omegabar\times[0,T])$ be such that
$$u_t - \Delta u = a(x, t) u - u^2 \quad \hbox{in $\Omega \times (0,T]$}$$
with \eqref{boundary T and initial conditions}. 
Assume moreover that $u_0>0$ is $\log$-concave in $\Omega$. 
Then there exists $\rho>0$ such that
$$\inf_{\Omega \times \Omega \times [0,1] \times [0,T]} \mc{C}^*_{\log(u)} \geq - e T \sup_{\Omega_{\rho} \times \Omega_{\rho} \times [0,1] \times (0,T]} \left(\mc{C}^*_{a(\cdot, \star) }\right)^-.$$
\end{corollary}

\begin{remark}\label{rem_bangbang}
We see that some results in Corollary \ref{corol_popul_quantit} can be extended to piecewise constant weights $a=a(x)$ arising %rising 
 in bang-bang optimization theory (see \cite[Theorem 3.9]{CaCo89}), i.e. equal to $a_1>0$ on some region of $\Omega$ and to $-a_2<0$ on some other region.
We give just a sketch of the proof. 
Let indeed $a$ be as above, and let $a_{\eps} \in L^{\infty}(\R^n)$ be a suitable equibounded regular approximating sequence for $a$. 
Let $u_{\eps}$ be positive solutions of $\partial_t u_{\eps} -\Delta u_{\eps} = a_{\eps}(x) u_{\eps} - u_{\eps}^2$, $u_{\eps} =0$ on $\partial \Omega$, $u_{\eps}(0)=u_0$.
Exploiting that $a_{\eps}(x) t -t^2\leq 0$ for $t \geq \sup_{\eps} \norm{a_{\eps}}_{\infty}$ and the maximum principle we obtain that $\norm{u_{\eps}}_{\infty}$ is equibounded; then standard regularity theory implies that $u_{\eps}$ is also equibounded in some $C^{0,\sigma}$, and hence it converges uniformly up to a subsequence by Ascoli-Arzelà theorem. 
We can thus pass to the limit the weak formulation (see e.g. \cite{CazenaveDicksteinEscobedo}) of the equations and obtain that $u_{\eps} \to u$ solution of the equation with weight $a(x)$. 
Since $u_{\eps}$ are regular, then they 
satisfy the assumptions of Corollary \ref{corol_popul_quantit} and thus, in particular, $\inf \mc{C}^*_{\log(u_{\eps})} \geq - e T \sup\left(\mc{C}^*_{ a_{\eps}(\cdot, \star) }\right)^- \geq - eT \osc(a_{\eps})$. 
 Passing to the limit, we obtain 
 $$\inf \mc{C}^*_{\log(u)} \geq - eT (a_1+a_2).$$
 We leave the details to the interested reader.
\end{remark}

%%%%%%%%%%%%%%%%%%%%%%%%%%%%%%%%%%%%%%%%%%%%%%%%%%%%
%%%%%%%%%%%%%%%%%%%%%%%%%%%%%%%%%%%%%%%%%%%%%%%%%%%%
\subsection{Harmonic concavity} 
\label{subsec_harmon_boundar}

In this subsection we consider applications of Cauchy-Dirichlet type to Theorem \ref{Concavity perturbative maximum principle for parabolic equations}.
The main assumptions %we consider} %assumption we assume 
are the following ones.
\begin{itemize}[wide]
	\item[-]
	Domain:
	\begin{equation}\label{eq_ipost_domain}
	\hbox{$\Omega \subset \R^n$, $n\geq 2$, bounded convex domain which satisfies the interior sphere property.}
	\end{equation}
	\item[-] Regularity: 
	\begin{align}\label{regularity conditions}
	u\in C^2_x(\Omega)\cap C^1_x(\Omegabar)\cap C^1_t((0,+\infty))\cap C(\Omegabar\times[0,+\infty)).
	\end{align} 
	\item[-] Equation:
	\begin{equation}\label{eq_general_bxu}
	u_t-\Delta u=b(x,u, t) \quad\hbox{in }\Omega\times (0,+\infty).
	\end{equation}
	\item[-] Boundary conditions: 
	\begin{equation}\label{boundary and initial conditions}
		\begin{cases}
				u>0 & \text{in } \Omega\times (0,+\infty),\\
					u=0 & \text{on }\partial \Omega\times [0,+\infty),	 \\
					u =0 & \text{on } \Omegabar \times \{0\}. 
		\end{cases}
	\end{equation}
\end{itemize}

Our aim is to deduce a quantitative estimate on the concavity function of a power of $u$ (see Theorem \ref{approximate concavity parabolic with boundary conditions}).
Let us consider a function $u\colon\Omegabar \times[0,+\infty)\to \R$ which satisfies \eqref{boundary and initial conditions} and the equation \eqref{eq_general_bxu}. 
For $\alpha, \beta>0$ %$\alpha\in(0,1)$ and $\beta \in (0,	2]$, 
set 
	\begin{equation}\label{eq_def_v_poweru}
v(x,t):=u^\alpha(x,t^{\beta}) .
\end{equation}
	By some straightforward computations, and recalled that $v>0$, we have that the equation solved by $v$ is 
\begin{equation}\label{eq_solved_by_v}
 -\Delta v =\mathscr{B} (x,v, Dv, t) \quad\text{in }\Omega\times (0,+\infty), 
 \end{equation}
where 
\begin{equation}\label{defn psi}
	\begin{aligned}
	\mathscr{B}(x,s,p,t) &:= \Btilde(x,s,\tilde{p},t) -\frac{1}{\beta} t^{1-\beta }p^{n+1} \\
	&:= \frac{1-\alpha}{\alpha}\frac{|\tilde{p}|^2}{s}+\alpha s^{\frac{\alpha-1}{\alpha}}b\left(x,s^{1/\alpha}, t^{\beta}\right) -\frac{1}{\beta} t^{1-\beta} p^{n+1}.
	\end{aligned}
\end{equation}

%, namely $v(x,t)=u^\alpha(x,t^{\beta})$, for suitable $\alpha$ and $\beta$. 
To apply Corollary \ref{cor_discard_vt} to $v$, we need to ensure that a negative minimum of the concavity function is attained in the interior of $\Omega \times (0,+\infty)$ or at infinity. 
In the following propositions, thus, we rule out the possibility of having a minimum on $\partial \Omega$ or $t=0$.
%\tor{M: dobbiamo allungare ulteriormente il brodo?}

\smallskip

We start with some boundary estimates.
We refer also to Lemma \ref{lem_behav_IshSal} for a version including $f(x_0,t)=0$ for $x_0 \in \partial \Omega$. 
We mention also \cite[Theorem 5]{MaSo00}. 

\begin{lemma}[Boundary estimates]
\label{lem_behav_boundar_omega}
%\label{lem_behav_t=0}
Let $\Omega$ satisfy \eqref{eq_ipost_domain}, and $u$ be a function such that \eqref{regularity conditions}, \eqref{eq_general_bxu} and \eqref{boundary and initial conditions} hold.
Suppose 	
$b\colon\Omega\times(0,+\infty) \times (0,+\infty)\to \R$ satisfies \hyperlink{H1}{\((H1)\)}, and consider $M, T, q,\gamma, k$ therein; we can assume $T$ small enough in such a way $\norm{u}_{L^{\infty}(\Omega\times[0,T])}<M$.
 The following facts hold.

\begin{itemize}

\item[(i)] $\,$ \emph{[}Dealing with $ \Omega \times \{0\}$\emph{]}:
let $\varphi_1$ be the first Dirichlet eigenfunction on $\Omega$ with eigenvalue $\lambda_1$, normalized by $\norm{\varphi_1}_{\infty}=1$. 
Then, $$u(x,t) \geq C e^{-\lambda_1 t} t^{\frac{1+\gamma}{1-q}}\varphi_1(x) \quad \hbox{in $\Omega \times (0,T)$}$$
for some explicit $C=C(\norm{u}_{L^{\infty}(\Omega \times [0,T])}, q, \gamma, T,k)>0$. 

\item[(ii)] $\,$\emph{[}Dealing with $\partial \Omega \times \{0\}$\emph{]}:
for any $x_0 \in \partial \Omega$ and $\beta \in (0,2]$, there exists $\delta=\delta(x_0,T, \beta,q,\gamma,k,\Omega)>0$ such that
$$u(x_0+t\nu,t^{\beta})\geq \delta t^{\frac{2\beta(1+\gamma)+(2-\beta)(1-q) }{2(1-q)}} \quad \hbox{for $t\in (0,T)$ small},$$
where $\nu$ is the interior normal to $\partial \Omega$ at $x_0$. 
\end{itemize}
\end{lemma}

\begin{proof}
From \hyperlink{H1}{\((H1)\)} with $M=\norm{u}_{L^{\infty}(\Omega \times [0,T])}$, we know that there exists $k=k(M, T)>0$ such that $u$ satisfies \eqref{boundary and initial conditions} and 
$$u_t-\Delta u=b(x,u,t)\geq k t^{\gamma} u^q \quad \text{in }\Omega\times (0, T).$$
Set 
$$w(x,t):= C e^{-\lambda_1 t} t^{\frac{1+\gamma}{1-q}} \varphi_1(x)$$
and $C := (\frac{1-q}{1+\gamma} k)^{\frac{1}{1-q}} 
$. 
A straightforward computation shows that $w$ satisfies \eqref{boundary and initial conditions} 
together with
$$w_t-\Delta w = C \frac{1+\gamma}{1-q} e^{-\lambda_1 t} t^{\frac{\gamma+q}{1-q}} \varphi_1(x) \quad \text{in }\Omega\times (0,+\infty)$$
and hence, since $q <1$ and by definition of $C$, 
$$w_t-\Delta w = 
\frac{1}{k} C^{1-q} \frac{1+\gamma}{1-q} e^{-(1-q) \lambda_1 t} \varphi_1^{1-q}(x) (k t^{\gamma} w^q) \leq k t^{\gamma} w^q \quad \text{in }\Omega\times (0,+\infty).$$
Then by Proposition \ref{comparison principle dickstein}, $u\geq w$ in $\Omega\times (0,T)$, which is the first claim.

\medskip

To show the second claim, assume first that $x_0 = 0 \in \partial \Omega$.
Let $\eta \in (0,\min\{1,T^{\beta}\})$ and define $w_\eta$ and $\Omega_\eta$ by 
	\[ w_\eta(x,t) :=\eta^{\frac{1 +\gamma}{1-q}}w(\eta^{-1/2}x,\eta^{-1}t), \quad \Omega_\eta:= \sqrt{\eta} \Omega \]
	which satisfy		
	\begin{equation*}
		\begin{cases}
			\partial_t w_\eta-\Delta w_\eta \leq 
			k t^{\gamma}(w_\eta)^q &\text{in }\Omega_\eta\times (0,+\infty),\\
			w_\eta>0 & \text{in } \Omega_\eta\times (0,+\infty), \\
			w_\eta=0 & \text{on }\partial \Omega_\eta\times (0,+\infty),\\ 
			w_\eta=0& \text{on } \Omega_\eta \times \{0\}.	
			\end{cases}
	\end{equation*}
Noticed that $\Omega_{\eta} \subset \Omega$, by Proposition \ref{comparison principle dickstein} we obtain $u\geq w_\eta$ in $\Omegabar_\eta\times (0,T)$. 
Now let $t=\eta^{\frac{1}{\beta}} \in (0, T)$ and, observed $0 \in \partial \Omega_{\eta}$, we choose $x=t\nu \in \Omega_{\eta}$ (i.e. $t^{1-\frac{\beta}{2}} \nu \in \Omega$ for $t$ small), so that
\begin{equation}\label{eq_dim_wtnutbeta}
 u(t\nu,t^{\beta})\geq t^{\frac{\beta (1+\gamma)
	}{1-q}}w(t^{\frac{2-\beta}{2}}\nu,1). 
	\end{equation}
	If $\beta=2$ we need a correction $rt \nu \in \Omega_{\eta}$ for $r$ small (so that $r \nu \in \Omega$), and the above expression becomes $u(r t \nu, (rt)^{2})\geq  t^{\frac{2 (1+\gamma)
	}{1-q}}w(r\nu,r^2)$ %we obtain $u(t\nu,t^{2})\geq t^{\frac{2 (1+\gamma)
	%}{1-q}}w(\nu,1)$, 
	and the claim follows choosing $\delta := r^{-\frac{2 (1+\gamma)}{1-q}}w(r \nu, r^2)$. %w(\nu,1)$. 
	If $\beta \in (0,2)$, 
by Lagrange theorem there exists $\xi_{t}\in (0,t^{\frac{2-\beta}{2}}\nu)$ such that (recall $w(0,1)=0$)
	$$w(t^{\frac{2-\beta}{2}}\nu,1) = t^{\frac{2-\beta}{2}}\nu\cdot D_{x} w(\xi_t,1).
	$$
	Being $\delta:= \frac{1}{2}\nu\cdot D_{x} w(0,1)>0$ thanks to Lemma \ref{Hopf lemma in parabolic case}, for $t$ small enough we have $\nu\cdot D_{x} w(\xi_t,1)>\delta$ and hence we obtain
	\begin{equation}\label{eq_dim_estim_w}
	 w(t^{\frac{2-\beta}{2}}\nu,1) \geq \delta t^{\frac{2-\beta}{2}}\quad \hbox{for $t$ small}.
	 \end{equation}
	Thus, by \eqref{eq_dim_wtnutbeta} and \eqref{eq_dim_estim_w} we obtain $u(t\nu,t^{\beta})\geq \delta t^{\frac{2\beta(1+\gamma)+(2-\beta)(1-q) }{2(1-q)}}$.

The case of general $x_0$ can be obtained straightforwardly, noticed that the equation related to $u \mapsto k t^{\gamma} u^q$ is $x$-translation invariant.
\end{proof}

In the following proposition, we rule out the possibility of the negative minimum to be on the boundary of $\Omega$ or at the initial time.

\begin{proposition}[Rule out $\partial \Omega$ and $t=0$]
\label{minimum not on the boundary}
	Let $\Omega$ satisfy \eqref{eq_ipost_domain}, and $u$ be such that \eqref{regularity conditions}, \eqref{eq_general_bxu} and \eqref{boundary and initial conditions} hold. Let $v$ be as in \eqref{eq_def_v_poweru}.
Suppose 		
$b\colon\Omega\times(0,+\infty)\times (0,+\infty)\to \R$ is such that $b(x,\cdot,t)$ is differentiable for all $(x,t)\in \Omega\times (0,+\infty)$ and satisfies \hyperlink{H1}{\((H1)\)}, and let $\alpha \in \big(0, \frac{2(1-q)}{2\beta(1+\gamma)+(2-\beta)(1-q)}\big)$
and $\beta \in (0,2]$.
Then $\mathcal{C}_{v}%u^\alpha(\cdot, \star^{\beta})}
$ cannot achieve any negative minimum at $(x_1,x_3,t_1,t_3,\lambda)\in \Omegabar \times \Omegabar \times [0,+\infty)\times [0,+\infty)\times (0,1) $ such that one of $t_1,t_3\in \left\{0\right\}$, or one of $x_1,x_2,x_3\in \partial \Omega$.
\end{proposition}

\begin{proof}
	Assume now by contradiction that $\mathcal{C}_{v}%u^\alpha}
$ admits a negative minimum at $(x_1,x_3,t_1,t_3,\lambda)$ such that one of $t_1,t_3\in \left\{0\right\}$ or one of $x_1,x_2,x_3\in \partial \Omega$. %Set $v(x,t):=u^{\alpha}(x, t^{\beta})$. 
	If $t_1,t_3\in [0,+\infty)$ and $x_1,x_3\in \partial \Omega$, then $\mathcal{C}_v(x_1,x_3,t_1,t_3,\lambda)=v(x_2,t_2) \geq 0$, which is not possible. If $t_1,t_3=0$ and $x_1,x_3\in \Omegabar$, $\mathcal{C}_v(x_1,x_3,t_1,t_3,\lambda)= 	0$ which again is not possible. 
	
		If $t_1,t_3\in (0,+\infty)$, $x_1\in \partial \Omega$ and $x_3\in \Omega$, the parabolic Hopf Lemma \ref{Hopf lemma in parabolic case} implies $D_{\bar{\nu}_1}u (x_1, t_1^{\beta}) >0$, where $\nu_1$ is the interior normal at $(x_1,t_1)$ and $\bar{\nu}_1:=(\tilde{\nu}_1, \beta t_1^{\beta-1}\nu_1^{n+1})$ thus
	\begin{align*}
\lim_{(y_1,y_3,\tau_1,\tau_3,\eta)\to(x_1,x_3,t_1,t_3,\lambda)}D_{\nu_1}\mathcal{C}_v(y_1,y_3,\tau_1&,\tau_3,\eta)\\=\lim_{(y_1,y_3,\tau_1,\tau_3,\eta)\to(x_1,x_3,t_1,t_3,\lambda)}\bigg((1-\eta)&D_{\nu_1}v(y_2,\tau_2)-(1-\eta)\alpha u^{\alpha-1}(y_1,\tau_1^{\beta})
D_{\bar{\nu}_1} u(y_1, \tau_1^{\beta}) \bigg)=-\infty,
	\end{align*}
		and this is a contradiction since it is a minimum point. 
		The same argument holds if $t_1\in (0,+\infty)$, $t_3=0$, $x_1\in \partial \Omega$, $x_3\in \Omega$.
	
	If $t_1=0$, $t_3\in (0,+\infty)$, $x_1\in \Omega$ and $x_3\in \Omega$, by Lemma \ref{lem_behav_boundar_omega} (\textit{i}) we have
	$$\limsup%\lim
	_{\tau \to 0} v_t(x_1,\tau)\geq \limsup%\lim
	_{\tau \to 0} \frac{v(x_1, \tau)}{\tau} \geq 
	C^{\alpha} \lim_{\tau \to 0} \frac{e^{-\alpha \lambda_1 \tau^{\beta}}}{\tau^{1-\frac{1+\gamma}{1-q}\alpha \beta}} (\varphi_{1}(x_1))^{\alpha} = +\infty$$
by the assumption $\alpha < \frac{2(1-q)}{2\beta(1+\gamma)+(2-\beta)(1-q)} \leq \frac{1-q}{\beta(1+\gamma)}$, 
	which implies 
	$$\lim_{(y_1,y_3,\tau_1,\tau_3,\eta)\to(x_1,x_3,t_1,t_3,\lambda)} D_{t_1}\mathcal{C}_v(y_1,y_3,\tau_1,\tau_3,\eta)=-\infty$$ 
	along a subsequence, again a contradiction. 
	
We are left with the case $t_1=0$, $t_3\in (0,+\infty)$, $x_1\in \partial \Omega$, $x_3 \in \Omega$. 
	By Lemma \ref{lem_behav_boundar_omega} (\textit{ii}) there exists $\delta>0$ such that $$\frac{u^\alpha(x_1 +t\nu , t^{\beta})}{t}\geq \delta^\alpha t^{\alpha \frac{2\beta(1+\gamma)+(2-\beta)(1-q) }{2(1-q)} -1}\quad \hbox{for $t>0$ small},$$
	which positively diverges as $t\to 0^+$ since $\alpha < \frac{2(1-q)}{2\beta(1+\gamma)+(2-\beta)(1-q)}$ by assumption. 
	Thus $(\nu, 1) \cdot D_{(x_1, t_1)}\mathcal{C}_v$ negatively diverges while approaching $(x_1,t_1)$ and this implies $(x_1,x_3,t_1,t_3,\lambda)$ cannot %can not 
be a minimum.
\end{proof}

To deal with the case $t=\infty$ we need to be able to compare the evolutive problem with the stationary one, and to this aim we will exploit the stability condition in Definition \ref{defn problema limite}. 
%Let us consider a function $u\colon\Omegabar \to \R$ which satisfies \eqref{boundary and initial conditions} and the equation \eqref{eq_general_bxu}. 
%	For $\alpha\in(0,1)$ and $\beta \in (0,	2]$, set 
%	$$v(x,t):=u^\alpha(x,t^{\beta}) .$$
%	By some straightforward computations, and recalled that $v>0$, we have that the equation solved by $v$ is 
%\begin{equation}\label{eq_solved_by_v}
% -\Delta v =\mathscr{B} (x,v, Dv, t) \quad\text{in }\Omega\times (0,+\infty), 
% \end{equation}
%where 
%\begin{equation}\label{defn psi}
%	\begin{aligned}
%	\mathscr{B}(x,s,p,t) &:= \Btilde(x,s,\tilde{p},t) -\frac{1}{\beta} t^{1-\beta }p^{n+1} \\
%	&:= \frac{1-\alpha}{\alpha}\frac{|\tilde{p}|^2}{s}+\alpha s^{\frac{\alpha-1}{\alpha}}b\left(x,s^{1/\alpha}, t^{\beta}\right) -\frac{1}{\beta} t^{1-\beta} p^{n+1}.
%	\end{aligned}
%\end{equation}
From now on with $\Btilde(x,s,\tilde{p},\infty)$ and $v(x,\infty)$ we mean their pointwise limit in time whenever the stability condition is assumed.

Consider $\alpha\in(0,1)$ and $\beta \in (0,2]$. In the following two remarks we discuss the possible localization of the global minimum of $\mathcal{C}_v$, in particular for $(t_1, t_3) \in [0,+\infty) \times \{+\infty\}$ and $\lambda\in \{0,1\}$.

\begin{remark}[Case $(t_1, t_3) \in [0,+\infty) \times \{+\infty\}$]
\label{rem_min_infinf}
Assuming $v_t\geq 0$ in $\Omega \times (0,+\infty)$ and $v(x,t)\to v(x,\infty)<+\infty$ as $t\to+\infty$, if $(x_1,x_3,t_1,t_3,\lambda)\in \Omegabar\times \Omegabar\times [0,+\infty)\times \left\{+\infty\right\}\times (0,1)$ is a global minimum for $\mathcal{C}_v$, then also $(x_1,x_3,+\infty,+\infty,\lambda)$ is a global minimum for $\mathcal{C}_v$. Indeed 
\begin{align*}
%\begin{split}
		\mathcal{C}_v(x_1,x_3,t_1,t_3,\lambda)&=v(x_2,\infty)-\lambda v(x_3,\infty %t_3
		)-(1-\lambda)v(x_1,t_1 %\infty
		)\\
		&\geq v(x_2,\infty)-\lambda v(x_3,\infty)-(1-\lambda)v(x_1,\infty).
%	\end{split}
\end{align*}
\end{remark}

\begin{remark}[Case $\lambda \in \{0,1\}$]
\label{remark lamba[0,1]}
Let us discuss the case when $\lambda\in\left\{0,1\right\}$. 
Assume $v_t\geq 0$ in $\Omega \times (0,+\infty)$ and $v(x,t)\to v(x,\infty)<+\infty$ as $t\to+\infty$. If $t_1,t_3\in [0,+\infty)$, then $\mathcal{C}_{v}%u^\alpha(\cdot, \star^{\beta})}
= 0$. 
If $t_3=+\infty$ and $\lambda=1$, then $\mathcal{C}_{v}%u^\alpha(\cdot, \star^{\beta})}
= 0$. 
If $\lambda = 0,t_3=+\infty$ let us consider a minimizing sequence $(x_1^n,x_3^n,t_1^n,t_3^n,\lambda^n)$ converging to $(x_1,x_3,t_1,+\infty,0)$ where $(x_1,x_3,t_1)\in\Omegabar\times \Omegabar\times [0,+\infty]$. 
We can assume, up to a subsequence, that $t_2^n$ converges to $t_2\in[t_1,+\infty]$. Since $x_2^n$ tends to $x_1$, then 
\begin{align*}
%\begin{split}
			\lim_{n\to +\infty}\mathcal{C}_v(x_1^n,x_3^n,t_1^n,t_3^n,\lambda^n)&=v(x_1,t_2)-v(x_1,t_1)\geq 0.
%	\end{split}
\end{align*}
\end{remark}
In light of Remarks \ref{rem_min_infinf} and \ref{remark lamba[0,1]}, we will use the following convention.

\smallskip

\textbf{Convention}. 
If $v_t\geq 0$, then by a global point of minimum for $\mathcal{C}_v$ on $\Omegabar\times \Omegabar\times [0,+\infty]\times [0,+\infty]\times [0,1]$ we mean a point picked in $\big(\Omegabar\times \Omegabar\times [0,+\infty)\times [0,+\infty) \times [0,1] \big) \cup \big(\Omegabar\times \Omegabar\times \{+\infty\}\times \{+\infty\}\times (0,1) \big)$. 

\medskip

With the above notations and conventions, and with the parabolic stability condition at hand, we can now include the case $t=+\infty$ in the reasoning of Proposition \ref{minimum not on the boundary}.
\begin{proposition}[On the case $t=+\infty$]
\label{minimum not on the boundary infinity}
	Under the same assumptions of Proposition \ref{minimum not on the boundary}, assume moreover $v_t\geq 0$ and that $b$ satisfies %satisfying 
	the parabolic stability condition (e.g. \textit{i})--\textit{iv}) in Remark \ref{stability_parabolic condition u^gamma} hold). 
	Then $\mathcal{C}_{v} %u^\alpha(\cdot, \star^{\beta})}
$ cannot achieve any negative minimum at $(x_1,x_3,t_1,t_3,\lambda)\in \Omegabar \times \Omegabar \times [0,+\infty]\times [0,+\infty]\times (0,1) $ such that one of $t_1,t_3\in \left\{0\right\}$, or one of $x_1,x_2,x_3\in \partial \Omega$. 
\end{proposition}
\begin{proof}
	In light of Proposition \ref{minimum not on the boundary} and Remark \ref{rem_min_infinf}, we need to exclude only the stationary case $t_1=t_3=+\infty$ and $x_1\in\partial \Omega$, $x_3\in \Omegabar$; this can be done as in the proof of Proposition \ref{minimum not on the boundary}. See also \cite{BuSq19,kennington}.
\end{proof}

By knowing the behaviour of the concavity function on the boundary, we are now able to deduce the main concavity estimate of this section.

\begin{theorem}[Estimate on power-concavity]
\label{approximate concavity parabolic with boundary conditions}
	Let $\Omega$ satisfy \eqref{eq_ipost_domain}, and $u$ be a function such that \eqref{regularity conditions}, \eqref{eq_general_bxu} and \eqref{boundary and initial conditions} hold. 
	Let $\beta \in [1,2]$ and suppose $b\colon\Omega\times(0,+\infty)\times (0,+\infty)\to \R$ is such that $b(x,\cdot, t) \in C^1((0,+\infty))$ %is \tr{$C^1$} % differentiable 
	%in $\R\setminus \left\{0\right\}$ 
	for all $(x,t)\in \Omega \times (0,+\infty)$, \hyperlink{H1}{\((H1)\)}, \hyperlink{H2}{\((H2)\)} (with $T=\infty$) hold and
	\begin{enumerate}[label=\textit{\roman*})]
%	 \item $s \mapsto s^{\alpha-1}b(x,s, t)$ is strictly decreasing for all $(x,t)\in \Omega \times (0,+\infty)$ for some $\alpha \in (0,1)$ as in Proposition \ref{minimum not on the boundary}. 
	 	 \item $\partial_s \left(s^{\alpha-1}b(x,s, t)\right)<0$ for all $(x,s,t)\in \Omega \times (0,+\infty) \times (0,+\infty)$ for some $\alpha \in (0,1)$ as in Proposition \ref{minimum not on the boundary}. 
%\item \tb{for some $\alpha \in (0,1)$ as in Proposition \ref{minimum not on the boundary}, \tr{there exists $\sigma>0$ such that $\partial_s\left( s^{\alpha-1}b(x,s, t)\right)<-\sigma$} for all $(x,s,t)\in \Omega \times u(\Omega) \times (0,+\infty)$};
%\tor{M: Richiedere che fosse solo strettamente decrescente non implicava l'esistenza di un $\sigma>0$, poiché una funzione strettamente decrescente può avere derivata che si annulla in qualche punto}
		\item $b$ satisfies the parabolic stability condition (e.g. \textit{i})--\textit{iv}) in Remark \ref{stability_parabolic condition u^gamma} hold). 
		\item $b(x, s, \cdot)$	is nondecreasing for each $(x,s) \in \Omega \times (0,+\infty)$.
	\end{enumerate}
Then $u_t \geq 0$.
	Let $(x_1,x_3,t_1,t_3,\lambda)\in \Omegabar\times \Omegabar\times [0,+\infty]\times [0,+\infty]\times [0,1]$ be a global minimum for $\mathcal{C}_{u^\alpha(\cdot, \star^{\beta})}%u^\alpha}
$ (with the above convention).
Then \eqref{strictly decreasing eqt} holds true for some $\sigma>0$ and
\begin{equation*}
 \min_{\Omegabar\times \Omegabar\times[0,+\infty]\times [0,+\infty]\times[0,1]}\mathcal{C}_{u^\alpha(\cdot, \star^{\beta})} \geq -\frac{1}{\sigma}\mathcal{HC}^-_{\Btilde(\cdot, u^\alpha(\cdot,\star^{\beta}),D_xu^\alpha(x_1,t_1^{\beta}), \star)}(x_1,x_3,t_1,t_3,\lambda).
 \end{equation*}
\end{theorem}

\begin{proof}
	By \hyperlink{H2}{\((H2)\)} we can apply Corollary \ref{corol_monotonicity} and obtain $u_t\geq 0$ in $ \Omega\times (0,+\infty)$, 
	which in turn implies $v_t\geq 0$ (so we are in the setting of the above convention).
 If the minimum is nonnegative, the statement holds. 
	 Assume $\mathcal{C}_{u^\alpha(\cdot, \star^{\beta})}$ admits a negative minimum. 
	 By Remark \ref{remark lamba[0,1]} and Proposition \ref{minimum not on the boundary infinity}, we know that $(x_1,x_3,t_1,t_3,\lambda)\in \Omega\times \Omega\times (0,+\infty]\times (0,+\infty]\times (0,1)$. 	 
		
	 Let us verify the hypotheses of Corollary \ref{cor_discard_vt} on the equation \eqref{eq_solved_by_v} satisfied by $v$.
%\tor{Scrivere esplicitamente $\mathscr{B}$ chi è} 
The function $\mathscr{B}$ takes the form \eqref{defn psi}.
%$$\mathscr{B}(x,s,p,t) %&:= \Btilde(x,s,\tilde{p},t) -\frac{1}{\beta} t^{1-\beta }p^{n+1} \\
%	= \frac{1-\alpha}{\alpha}\frac{|\tilde{p}|^2}{s}+\alpha s^{\frac{\alpha-1}{\alpha}}b\left(x,s^{1/\alpha}, t^{\beta}\right) -\frac{1}{\beta} t^{1-\beta} p^{n+1}.
%$$
	 Since $\alpha^{-1}(1-\alpha)s^{-1}|\tilde{p}|^2$ is nonincreasing %decreasing 
with respect to $s$%\tb{; \tr{by (i)}, \eqref{strictly decreasing eqt} holds.}
%Then denoting $r$ for $s^{\frac{1}{\alpha}}$, 
%	 the second term of $\mathscr{B}$ becomes $\alpha r^{\alpha-1}b(x,r,t^{\beta})$ which is strictly decreasing for hypothesis with respect to $r$. Hence $\mathscr{B}(x,\cdot,p,t)$ is strictly decreasing for all $p\in \R^n$, $x\in \Omega$ and $t \in (0,+\infty)$, which implies \tor{Falso in generale! Una funzione può essere strettamente decrescente ma avere derivata che si annulla in qualche punto. Forse la forma precisa di B aiuta? Oppure bisogna assumerlo sulla $b$, come rafforzativo di (i)? Nelle applicazioni è poi verificata (credo di sì)?}
% the existence of $\sigma>0$ such that \eqref{strictly decreasing eqt} holds. 
, $(i)$ implies the existence of $\sigma>0$ such that \eqref{strictly decreasing eqt} holds.

	If $(x_1,x_3,t_1,t_3,\lambda)\in \Omega\times \Omega\times (0,+\infty)\times (0,+\infty)\times (0,1)$, then the claim follows by 
	Corollary \ref{cor_discard_vt}. % Theorem \ref{Concavity perturbative maximum principle for parabolic equations}.
	 If $t_1=t_3=+\infty$ the claim follows by \cite[Lemma 2.9]{BuSq19}. 
	We conclude by observing that, being $\beta \in [1,2]$, we have $\mc{HC}_{\star^{1-\beta}} \leq 0$.
\end{proof}

In the following corollary, starting from Theorem \ref{approximate concavity parabolic with boundary conditions}, we deduce a bound in terms of the concavity function of the nonlinearity (rather than on the harmonic function).

\begin{corollary}\label{corollario con concavità}
In the assumptions of Theorem \ref{approximate concavity parabolic with boundary conditions}, we also have	
\[ \min_{\Omegabar\times \Omegabar\times[0,+\infty]\times [0,+\infty]\times[0,1]}\mathcal{C}_{u^\alpha(\cdot, \star^{\beta})} \geq-\frac{\alpha}{\sigma}\frac{\mathcal{C}^-_{\tilde{g}(\cdot, u^{\alpha}(\cdot,\star^{\beta}), \star)}(x_1,x_3,t_1,t_3,\lambda)}{(\lambda u^\alpha(x_3,t_3^{\beta})
+(1-\lambda)u^\alpha(x_1,t_1^{\beta}))^2}\]
where $\tilde{g}(x,s,t):= s^{\frac{3\alpha-1}{\alpha}}b(x,s^{1/\alpha}, t^{\beta})$.
\end{corollary}
\begin{proof}
	Since Theorem \ref{approximate concavity parabolic with boundary conditions} holds, let us give a lower bound for the $\mathcal{HC}^-_{\Btilde}$. 
	Set 
	$$g(x,s, t):=s^2\Btilde(x,s,\tilde{\xi},t)= \frac{1-\alpha}{\alpha}|\tilde{\xi}|^2s+\alpha s^{\frac{3\alpha-1}{\alpha}}b(x,s^{1/\alpha}, t^{\beta})$$ where $\tilde{\xi}=
	D_x u^{\alpha}(x_1, t_1^{\beta})$, and moreover $v(x,t) = u^{\alpha}(x,t^{\beta})$. 
	 By Proposition \ref{prop_prel_gxj} we have
\begin{align*}
	\mathcal{HC}&_{\Btilde(\cdot,v(\cdot,\star),\tilde{\xi}, \star)}(x_1,x_3,t_1,t_3,\lambda)\geq\frac{\mathcal{C}_{g(\cdot,v(\cdot,\star), \star)}(x_1,x_3,t_1,t_3,\lambda)}{(\lambda v_3 	+(1-\lambda)v_1 )^2}.
\end{align*}
Since $\frac{1-\alpha}{\alpha}|\tilde{\xi}|^2s$ is linear we have $\mathcal{C}_g=\alpha\mathcal{C}_{\tilde{g}(\cdot, v(\cdot, \star), \star)}$ and hence %, we obtain
\[ 	\mathcal{HC}_{\Btilde(\cdot,v(\cdot,\star),\tilde{\xi}, \star)}(x_1,x_3,t_1,t_3,\lambda)\geq \alpha\frac{\mathcal{C}_{\tilde{g}(\cdot, v(\cdot, \star), \star) }(x_1,x_3,t_1,t_3,\lambda)}{(\lambda v_3 +(1-\lambda)v_1 )^2} \]
and thus the claim.
\end{proof}

As a consequence of Corollary \ref{corollario con concavità} we obtain the following result.
\begin{corollary}[Exact power-concavity]
\label{concavity parabolic with boundary conditions}
In the assumptions of Theorem \ref{approximate concavity parabolic with boundary conditions}, %\tr{(where $(i)$ can be weakened to $b$ strictly decreasing in $s$)}, 
assume moreover
\begin{itemize}
		\item[iv)] $(x,s,t) \mapsto s^{\frac{3\alpha-1}{\alpha}}b(x,s^{\frac{1}{\alpha}},t^{\beta}) $ is a concave function. 
\end{itemize}
Then $u^\alpha(\cdot, \star^{\beta})$ is a concave function in $\Omegabar\times [0,+\infty]$. 
%\tor{M: Qui potrebbe bastare $b$ (strettamente) decrescente? Forse potremmo anche evitare di scriverlo}
\end{corollary}

%%%%%%%%%%%%%%%%%%%%%%%%%%%%%%%%%%%%%%%%%%%%%%%%%%%%
\subsubsection{Applications}

Let $\Omega$ be a convex bounded domain of $\R^n$ and suppose $u$ is a function such that \eqref{regularity conditions} and \eqref{boundary and initial conditions} hold together with \eqref{eq_general_bxu} in the case $b(x,u,t)=a(x,t) u^q$, namely 
\begin{equation}\label{equazione a(x)u^gamma}
	u_t-\Delta u=a(x,t)u^q \quad\text{in }\Omega\times (0,+\infty),
\end{equation}
where $q\in[0,1)$ and %,
 $a\colon\Omega \times (0,+\infty)\to \R$ is such that $a(\cdot,t)$ is measurable for all $t\in(0,+\infty)$. 

\smallskip

Let us start from some exact results. For the sake of simplicity, we focus on $a(x,t)=a(x) t^{\gamma}$, but more general behaviours in time may be considered.
\begin{proof}[Proof of Theorem \ref{thm_main_weighted_le}]
Let us verify the hypotheses of Corollary \ref{concavity parabolic with boundary conditions} with 
$$\alpha= \frac{\theta(1-q)}{1+2\theta +\beta \gamma \theta}<\frac{2(1-q)}{2\beta(1+\gamma)+(2-\beta)(1-q)}<1-q$$ and $b(x,s, t)=a(x) t^{\gamma} s^q$. Notice $s \mapsto s^{\alpha-1} b(x,s,t) = a(x) t^{\gamma} s^{\alpha-(1-q)}$
is strictly decreasing for all $x\in \Omega$ and $t>0$.
Moreover $s \mapsto s^{\frac{3\alpha -1}{\alpha}}b(x,s^{1/\alpha},t^{\beta})=a(x) t^{\gamma \beta} s^{3-\frac{1-q}{\alpha}} $
is concave as a product since $\frac{1}{\theta}+\beta \gamma + 3-\frac{1-q}{\alpha}=1$, by the assumptions we have $3-\frac{1-q}{\alpha}\geq 0$ and thus the claim in light of Proposition \ref{alpha concavity property} (\textit{iv}). 
To get the parabolic condition, we argue as in%
\footnote{Notice that the solution of the stationary problem of $-\Delta u = a(x) u^q$ is $\frac{\theta(1-q)}{1+2\theta}$-concave, where $\frac{\theta(1-q)}{1+2\theta}\geq \frac{\theta(1-q)}{1+2\theta +\beta \gamma \theta}$.} Remark \ref{rem_troncamento} thanks to Remark \ref{stability_parabolic condition u^gamma}, while we see that \hyperlink{H2}{\((H2)\)} %\hyperlink{H1}{\((H1)\)} 
is satisfied as shown in Remark \ref{remark comparison principle}. Then we can apply Corollary \ref{concavity parabolic with boundary conditions} and get the thesis.
Moreover if $a$ is a constant, then $a$ is $\theta$-concave for all $\theta\geq 1$, thus the claim by Proposition \ref{alpha concavity property} (\textit{ii}). 
Finally, we see that, if $\beta=\frac{1}{\gamma}\leq 2$ and $a$ is constant, we can repeat the previous arguments with $\alpha=\frac{1-q}{2+\beta \gamma }$ and obtain the claim.
\end{proof}

\begin{proof}[Proof of Theorem \ref{thm_sum_powers}]
	Let us verify that $b(x, s)=a(x) s^p+s^q$ satisfies the hypotheses of Corollary \ref{concavity parabolic with boundary conditions} with $\alpha=\frac{1-q}{2}$. 
	Since $s^{\frac{1-q}{2}-1}(a(x) s^p+s^q)=a(x) s^{\frac{-1-q+2p}{2}}+s^{\frac{q-1}{2}}$, it is strictly decreasing by the hypothesis on $p$. 
	Moreover the parabolic stability condition is guaranteed by Remark \ref{stability_parabolic condition u^gamma} and $b(x, s)=a(x) s^p+s^q$ satisfies \hyperlink{H2}{\((H2)\)} %\hyperlink{H1}{\((H1)\)} 
	as shown in Remark \ref{remark comparison principle}.
	Regarding the concavity we observe that $s \mapsto a(x) s^{\frac{1-3q+2p}{1-q}} +s$ is concave by the assumptions. 
	Thus we have the claim.
\end{proof}

\begin{proposition}\label{prop_kenn88}
 	Let $\Omega$ satisfy \eqref{eq_ipost_domain}, and $u$ be a function such that \eqref{regularity conditions} and \eqref{boundary and initial conditions} hold.
Suppose
	\begin{equation*}
			u_t-\Delta u= (1-u)^p \quad\emph{in }\Omega\times (0,+\infty)
	\end{equation*}
	for some $p\in (0,1)$. 
	Then $u(\cdot, \star^{2})$ is $\frac{1}{2}$-concave in $\Omegabar \times [0,+\infty)$.
\end{proposition}

\begin{proof}
Notice that $u(\Omega \times (0,+\infty)) \subset [0,1]$, thus we can substitute the right-hand side with\footnote{%
Notice that the derivative explodes to $-\infty$ in $s=1$; anyway, for exact concavity arguments, it is sufficient a strict monotonicity on the image of $u$.}
 $(1-s)^p \chi_{[0,1]}(s)$ with $s \in \R$. 
Observe that $b(s)=(1-s)^p$ satisfies \hyperlink{H2}{\((H2)\)} (being nonincreasing), and it satisfies \hyperlink{H1}{\((H1)\)} with $q=0$ for $M<1$. %\tr{, for $s \in u(\Omega) \subset [0,1]$ (which is enough)}.
To conclude it is sufficient to observe that $s^{\alpha-1} b(s)$ is strictly decreasing for each $\alpha \in (0,1)$, while $s^{\frac{3\alpha-1}{\alpha}}b(s^{\frac{1}{\alpha}})$ is concave for each $\alpha \in [\frac{1}{3}, \frac{1}{2}]$. 
Noticed that $\alpha<\frac{1}{2} $ satisfies the restriction of Proposition \ref{minimum not on the boundary}, then the claim holds by Corollary \ref{concavity parabolic with boundary conditions} for such $\alpha$; by continuity of power concavity (see Proposition \ref{alpha concavity property}), the claim holds also for $\alpha=\frac{1}{2}$. %\tor{M: ho corretto qua}
%Noticed that $\alpha=\frac{1}{2} $, satisfies the restriction of Proposition \ref{minimum not on the boundary}, then the claim holds by Corollary \ref{concavity parabolic with boundary conditions}.
\end{proof}

\begin{remark}
\label{rem_Ken88}
In \cite{kenningtonparabolic} the author shows that, for $\beta=1$, the solution $u$ is such that $\varphi_p(u)$ is concave, where $\varphi_p(s) := \int_0^s \frac{1}{\sqrt{1-(1-r)^{p+1}}
} dr$. 
Set $\psi(s):=\sqrt{s}$ the transformation found in Proposition \ref{prop_kenn88}, by the fact that $\psi \circ \varphi_p^{-1}$ is (strictly) concave, we see that, generally, the information given by $\varphi_p$ is better (when compared for $\beta=1$). 
This fact is due to the choice of a transformation not necessarily of power type, but tailored on the nonlinearity, that is, up to constants, $\varphi_p(s)=\int_0^{s} \frac{1}{\sqrt{F(r)}} dr$, where $F(s)=\int_0^{s} f(r) dr$, $f(s)=(1-s)^p$. 
This choice is coherent with the approach in \cite{AAGS2023,BMSplaplacian, GaSq24}. 
In the present paper, we decided not to pursue this approach in order to focus on the difficulties given by the parabolic framework and avoid technicalities. On the other hand, similar arguments could be developed, and we leave the details to the interested reader.
\end{remark}

Let us now move to some perturbative result for \eqref{equazione a(x)u^gamma}. For the sake of simplicity, we consider $a=a(x)$ locally bounded (see also Remark \ref{rem_time_dependence}), and $\beta=1$. 

 Let $(x_1,x_3,t_1,t_3,\lambda)\in \Omega\times \Omega \times (0,+\infty]\times (0,+\infty]\times (0,1)$ be a global interior minimum of $\mathcal{C}_{u^{\frac{1-q}{2}}}$, then by \eqref{eq_def_xi} in Theorem \ref{Concavity perturbative maximum principle for parabolic equations} (and \cite{kennington})  %\tr{} QUI ANDREBBE CITATA LA STESSA PROPRIETA' ANCHE PER T=INF, PERCHE' QUELLA CHE USIAMO E' SOLO PER T<INF
 we have %\tor{Ricontrolla}
% $\xi:=D_xu(x_1,t_1)=D_xu(x_2,t_2)=D_xu(x_3,t_3)$.
% $\xi:=D_x\tr{u^{\frac{1-q}{2}}}%u
% (x_1,t_1)=D_x \tr{u^{\frac{1-q}{2}}}%u
% (x_2,t_2)=D_x\tr{u^{\frac{1-q}{2}}}%u
% (x_3,t_3)$.
 $$\tilde{\xi}:=\tfrac{1-q}{2} u^{-\frac{1+q}{2}}(x_1,t_1)D_xu(x_1,t_1)=\tfrac{1-q}{2} u^{-\frac{1+q}{2}}(x_2,t_2)D_xu(x_2,t_2)=\tfrac{1-q}{2} u^{-\frac{1+q}{2}}(x_3,t_3)D_xu(x_3,t_3).$$
 We define, for any $\rho>0$ small enough, 
\begin{equation}\label{definition of nu,M,m}
 \mathfrak{m}_{\rho} :=\frac{2}{1-q} \inf_{x \in \overline{\Omega_{\rho}}}f^{\tilde{\xi}}(x) ,\quad \mathfrak{M}_{\rho} :=\frac{2}{1-q} \sup_{x \in \overline{\Omega_{\rho}}}f^{\tilde{\xi}}(x), 
 \end{equation}
where 
\begin{equation}\label{eq_def_fp}
f^{\tilde{p}}(x):=\dfrac{1+q}{1-q}\abs{\tilde{p}}^2+\dfrac{1-q}{2}a(x), \quad \hbox{for $\tilde{p} \in \R^n$ and $x \in \Omega$}.
\end{equation}
Notice that
\begin{equation}\label{eq_propr_Mrho_mrho}
 \mathfrak{M}_{\rho} - \mathfrak{m}_{\rho} = \sup _{\overline{\Omega_{\rho}} } a - \inf_{\overline{\Omega_{\rho}}} a, \quad \mathfrak{m}_{\rho} \geq \inf_{\overline{\Omega_{\rho}}} a.
 \end{equation}
	
\begin{proposition}[Oscillation estimate]\label{approximate concavity parabolic application a(x)u^gamma}
	Let $\Omega$ satisfy \eqref{eq_ipost_domain}, and $u$ be a function such that \eqref{regularity conditions}, \eqref{boundary and initial conditions} and \eqref{equazione a(x)u^gamma} hold. 
	Suppose moreover $a\colon\Omega \to \R$ measurable, bounded %locally bounded %SERVE PER VERIFICARE H2, COME DA REMARK 3.4
	and $a\geq m>0$.
 Then $u$ is increasing in time and there exists $\rho>0$ such that 
 \begin{equation*}
 \min_{\Omegabar\times \Omegabar\times[0,+\infty]\times [0,+\infty]\times[0,1]}\mathcal{C}_{u^{\frac{1-q}{2}}}\geq 
 - \frac{\norm{u(\cdot, \infty)}^{\frac{1-q}{2}}_{L^{\infty}(\Omegabar)}}{\mathfrak{m}_{\rho} } \left(\inf_{\overline{\Omega_\rho}\times \overline{\Omega_\rho} \times [0,1] } \mathcal{C}_{a}-\frac{ \mathfrak{M}_{\rho} }{ \mathfrak{m}_{\rho} }\varepsilon \right)^{-}.
 \end{equation*}
where $\varepsilon:= \mathfrak{M}_{\rho} - \mathfrak{m}_{\rho} \geq 0$ 
and $ \mathfrak{M}_{\rho} , \mathfrak{m}_{\rho} $ are defined in \eqref{definition of nu,M,m}. 
In particular
$$\min_{\Omegabar\times \Omegabar\times[0,+\infty]\times [0,+\infty]\times[0,1]}\mathcal{C}_{u^{\frac{1-q}{2}}} \geq 
-\norm{u(\cdot, \infty)}^{\frac{1-q}{2}}_{L^{\infty}(\overline{\Omega_{\rho}})} \left(2+ \frac{ \eps }{\inf_{\overline{\Omega_{\rho}}} a} \right) \frac{\eps}{\inf_{\overline{\Omega_{\rho}}} a } .$$
\end{proposition}

\begin{proof}
	We want to apply Theorem \ref{approximate concavity parabolic with boundary conditions} with $\alpha=	\frac{1-q}{2}
	$ and $b(x,s)=a(x)s^q$. Since $a$ is bounded from below, $b(x,s)\geq m s^q$. Moreover $s^{\alpha-1}b(x,s)=s^{\frac{1-q}{2}-1}a(x)s^q=a(x)s^{\frac{q-1}{2}}$ is a strictly decreasing function of $s$, the parabolic condition on $b$ is satisfied by Remark \ref{stability_parabolic condition u^gamma} and \hyperlink{H2}{\((H2)\)} %\hyperlink{H1}{\((H1)\)} 
	is satisfied as shown in Remark \ref{remark comparison principle}. 
	Thus we can apply Theorem \ref{approximate concavity parabolic with boundary conditions}. 
	Notice that $\Btilde$ defined in \eqref{defn psi} takes the form
	\[\Btilde(x,s,\tilde{p})=\dfrac{f^{\tilde{p}}(x)}{s}\]
	where $f^{\tilde{p}}$ is given in \eqref{eq_def_fp}.
	Let $(x_1,x_3,t_1,t_3,\lambda)\in \Omegabar\times \Omegabar\times [0,+\infty]\times [0,+\infty]\times [0,1]$ be a global minimum for $\mathcal{C}_{v}$. 
	By Proposition \ref{minimum not on the boundary infinity}, % \ref{minimum not on the boundary}, 
we know that $(x_1,x_3,t_1,t_3,\lambda)\in \Omega\times \Omega\times (0,+\infty]\times (0,+\infty]\times [0,1]$. 
	If the minimum is nonnegative, %positive, 
	the statement holds. 
	Suppose $\mathcal{C}_{v}$ admits a negative minimum at $(x_1,x_3,t_1,t_3,\lambda)\in \Omega\times \Omega\times (0,+\infty)\times (0,+\infty)\times (0,1)$, then we need to estimate %of 
	the harmonic concavity function. 
	Define 
	\[\rho:=\min \big\{d(x_1,\partial\Omega),d(x_3,\partial \Omega)\big\}>0. \]
	Then by \eqref{definition of nu,M,m}, it holds (recall \eqref{eq_hc_grt_c})
	\begin{align*}
		\mathcal{HC}&_{\Btilde(\cdot,v(\cdot,\star),\tilde{\xi})}(x_1,x_3,t_1,t_3,\lambda)\\
		= & \dfrac{f^{\tilde{\xi}}_2}{\lambda v_3 +(1-\lambda)v_1}-\dfrac{f^{\tilde{\xi}}_1}{v_1}\dfrac{f^{\tilde{\xi}}_3}{ v_3 }\dfrac{v_1 v_3 }{\lambda f^{\tilde{\xi}}_1 v_3 +(1-\lambda)f^{\tilde{\xi}}_3v_1}\\
		=& \dfrac{1}{\lambda v_3 +(1-\lambda)v_1}\left( f^{\tilde{\xi}}_2-f^{\tilde{\xi}}_1f^{\tilde{\xi}}_3\dfrac{\lambda v_3 +(1-\lambda)v_1}{\lambda f^{\tilde{\xi}}_1 v_3 +(1-\lambda)f^{\tilde{\xi}}_3v_1}\right)\\
		\geq & \dfrac{1}{\lambda v_3 +(1-\lambda)v_1} \left(\mathcal{HC}_{f^{\tilde{\xi}}}(x_1,x_3,\lambda%x_1,x_3,t_1,t_3,\lambda
		)+f^{\tilde{\xi}}_1f^{\tilde{\xi}}_3\left(\frac{1}{\lambda f^{\tilde{\xi}}_1+(1-\lambda)f^{\tilde{\xi}}_3}-\frac{2}{(1-q) \mathfrak{m}_{\rho} } \right)\right) \\
				\geq & \dfrac{1}{\lambda v_3 +(1-\lambda)v_1} \left(\mathcal{HC}_{f^{\tilde{\xi}}}(x_1,x_3,\lambda%x_1,x_3,t_1,t_3,\lambda
				)+\frac{1-q}{2} \mathfrak{M}_{\rho} ^2\left(\frac{1}{ \mathfrak{M}_{\rho} }-\frac{1}{ \mathfrak{m}_{\rho} } \right)\right) \\
			\geq & \dfrac{1}{\lambda v_3 +(1-\lambda)v_1} \left(\mathcal{C}_{f^{\tilde{\xi}}}(x_1,x_3,\lambda%x_1,x_3,t_1,t_3,\lambda
			)-\frac{1-q}{2}\frac{ \mathfrak{M}_{\rho} }{ \mathfrak{m}_{\rho} }( \mathfrak{M}_{\rho} - \mathfrak{m}_{\rho} )\right)\\
			= & \frac{1-q}{2}\dfrac{1}{\lambda v_3 +(1-\lambda)v_1} \left(\mathcal{C}_{a}(x_1,x_3,\lambda%x_1,x_3,t_1,t_3,\lambda
			)-\frac{ \mathfrak{M}_{\rho} }{ \mathfrak{m}_{\rho} }\varepsilon\right).
	\end{align*}
	In the case %as 
	$t_1=+\infty$ or $t_3=+\infty$, the previous inequality holds by the same arguments for $\mathcal{HC}_{\Btilde(\cdot,v(\cdot,\star),\tilde{\xi})}(x_1,x_3,+\infty,+\infty,\lambda)$. 
	Thus it holds
	\[ \mathcal{HC}_{\Btilde(\cdot,v(\cdot,\star),\tilde{\xi})}(x_1,x_3,+\infty, +\infty,\lambda) 
	\geq 
	\frac{1-q}{2} \frac{1}{\lambda v_3 +(1-\lambda)v_1} \left(\mathcal{C}_{a}(x_1,x_3,\lambda)-\frac{ \mathfrak{M}_{\rho} }{ \mathfrak{m}_{\rho} }\varepsilon\right).\]
	Let us write explicitly the constant $\sigma$ of \eqref{strictly decreasing eqt} recalling that $ v_2 < \lambda v_3 +(1-\lambda)v_1$:
	\begin{align*}
		\sup_{z\in [ v_2 ,\lambda v_3 +(1-\lambda)v_1]} \partial_s\Btilde(x_2,z,D_x v_2 ) & \leq f^{D_x v_2 }(x_2) 	\sup_{z\in [ v_2 ,\lambda v_3 +(1-\lambda)v_1]}\bigg(-\frac{1}{z^2}\bigg)
		\\&
		\leq -\frac{1-q}{2}\frac{ \mathfrak{m}_{\rho} }{(\lambda v_3 +(1-\lambda)v_1)^2}
	\end{align*}
	which leads to $\sigma:=\frac{1-q}{2}\frac{ \mathfrak{m}_{\rho} }{(\lambda v_3 +(1-\lambda)v_1)^2}$. 
	Finally by Theorem \ref{approximate concavity parabolic with boundary conditions}, it holds
$$\min_{\Omegabar\times \Omegabar\times[0,+\infty]\times [0,+\infty]\times[0,1]}\mathcal{C}_{u^{\frac{1-q}{2}}}\geq 
\min	\left\{0,\frac{\lambda v_3 +(1-\lambda)v_1}{ \mathfrak{m}_{\rho} } \left(\mathcal{C}_{a}(x_1,x_3,\lambda)-\frac{ \mathfrak{M}_{\rho} }{ \mathfrak{m}_{\rho} }\varepsilon\right)\right\}$$
	which implies the claim.
We conclude by observing that $\mathcal{C}_{a} \geq \inf_{\overline{\Omega_{\rho}}}%\Omegabar_\rho} 
a - \sup_{\overline{\Omega_{\rho}}} %\Omegabar_\rho} 
a$ on $\overline{\Omega_{\rho}} %\Omegabar_{\rho}
$, \eqref{eq_propr_Mrho_mrho} and the trivial identity $\frac{ \mathfrak{M}_{\rho} }{ \mathfrak{m}_{\rho}} = \frac{ \mathfrak{M}_{\rho} - \mathfrak{m}_{\rho} }{ \mathfrak{m}_{\rho}} +1$.
\end{proof}

\begin{remark}\label{rem_rough_estimate_osca}
We highlight that, in the previous proof, we exploited $\inf_{\Omega} a>0$. 
On the other hand, in regard of the conclusion, we observe that if $a\to 0$ on $\partial \Omega$ the final expression is still well defined; we refer to Section \ref{sec_torsion} for some generalization for the torsion problem in this direction.

Assumed $a$ bounded and set $M:= \sup_{\Omega} a,\ m:= \inf_{\Omega} a$, the final statement in Proposition \ref{approximate concavity parabolic application a(x)u^gamma} implies 
$$\min_{\Omegabar\times \Omegabar\times[0,+\infty]\times [0,+\infty]\times[0,1]}\mathcal{C}_{u^{\frac{1-q}{2}}} \geq 
- \norm{u(\cdot, \infty)}^{\frac{1-q}{2}}_\infty \left(2+ \frac{ \osc(a) }{m} \right) \frac{\osc(a)}{m}$$
where we recall $\osc(a)=M-m$.
Finally, if $a\in C^1(\Omega)$, it is possible to express $\varepsilon$ in terms of the gradient of $a$. Indeed for some $\bar{x},\tilde{x}\in\overline{ \Omega_\rho}$ and $\bar{z}$ on the segment $[\bar{x},\tilde{x}]$ lying in $\overline{\Omega_\rho}$, it holds
	\begin{align*}
		\eps = \mathfrak{M}_{\rho} - \mathfrak{m}_{\rho} &= \nabla a (\bar{z}) \cdot (\bar{x}-\tilde{x}) 
		 \leq \norm{\nabla a}_{\infty} \diam (\Omega).
	\end{align*}
\end{remark}

As a consequence of Proposition \ref{ulam}, Proposition \ref{approximate concavity parabolic application a(x)u^gamma} and Remark \ref{rem_rough_estimate_osca} we obtain the following result.
\begin{corollary}[Hyers-Ulam approximation]\label{corollario ulam}
Under the assumptions of Proposition \ref{approximate concavity parabolic application a(x)u^gamma}, there exist %a
 $C=C\big(\norm{u(\cdot, \infty)}_{\infty}, \inf_{\Omega} a, \osc(a), q, %\gamma, 
n \big)>0$ (with a linear dependence with respect to $\osc(a)$) and a concave function $w:\Omega\times(0,+\infty)\to \R$ such that
\[ \norm{u^{\frac{1-q}{2}}-w}_{L^\infty(\Omega\times(0,+\infty))}\leq C \osc(a).\]
\end{corollary}

\begin{proposition}[$\theta$-concavity estimate]
\label{prop_theta_concav_estim}
Let $\Omega$ satisfy \eqref{eq_ipost_domain}, $\theta \geq 1$,
and $u$ be a function such that \eqref{regularity conditions}, \eqref{boundary and initial conditions} and \eqref{equazione a(x)u^gamma} hold. 
Suppose moreover $a\colon\Omega \to \R$ measurable, $0< m\leq a\leq M$ and $m^\theta\geq M^{\theta}/2$. 
Let $(x_1,x_3,t_1,t_3,\lambda)\in \Omegabar\times \Omegabar\times [0,+\infty]\times [0,+\infty]\times [0,1]$ be a global minimum for $\mathcal{C}_{u^{\frac{\theta(1-q)}{1+2\theta}}}$. 
Then
% there exists $\rho>0$ such that
\[	\min_{\Omegabar\times \Omegabar\times[0,+\infty]\times [0,+\infty]\times[0,1]}\mathcal{C}_{u^{\frac{\theta(1-q)}{1+2\theta}}}
\geq
%-\frac{2\theta }{1+2\theta}\frac{||u(\cdot, \infty)||_{L^{\infty}(\Omega)}^{\frac{(\theta-1)(1-q)}{1+2\theta}}}{ \mathfrak{m}_{\rho} }(\mathcal{C}^-_{a^{\theta}})^{1/\theta}(x_1,x_3,\lambda).
-\frac{\theta }{1+\theta}\frac{||u(\cdot, \infty)||_{L^{\infty}(\Omega)}^{\frac{\theta(1-q)}{1+2\theta}}}{m}(\mathcal{C}^-_{a^{\theta}})^{1/\theta}(x_1,x_3,\lambda).
\]
\end{proposition}
\begin{proof}
As in Proposition \ref{approximate concavity parabolic application a(x)u^gamma}, the hypotheses imply Corollary \ref{corollario con concavità}. 
Being $b(x,s)=a(x)s^q$, we need to evaluate
 $\mathcal{C}_{\tilde{g}(\cdot, v(\cdot, \star))}$
 with $\tilde{g}(x,s):= a(x) s^{\frac{3\alpha-1+q}{\alpha}}$ and $\alpha=\frac{\theta(1-q)}{1+2\theta}$.
 We see that $s\mapsto s^{\frac{3\alpha-1+q}{\alpha}}$ is $\left(1-\frac{1}{\theta}\right)^{-1}$-concave.
 If $\theta>1$, applying Proposition \ref{concavità prodotto} it holds 
 \begin{align*}
	\mathcal{C}_{\tilde{g}(\cdot, v(\cdot, \star))} (x_1,x_3,t_1,t_3,\lambda)\geq
	& -(\mathcal{C}^-_{a^{\theta}})^{1/\theta}(x_1,x_3,\lambda) \big(\lambda v^{1-1/\theta}_3+(1-\lambda) v^{1-1/\theta}_1\big).
\end{align*}
Instead if $\theta=1$, the previous inequality is obvious. So by Corollary \ref{corollario con concavità}
\begin{align*}
 \mathcal{C}_{u^\alpha}(x_1,x_3,t_1,t_3, \lambda)
 \geq&-\frac{\alpha}{\sigma}\frac{\mathcal{C}_{\tilde{g}(\cdot, v(\cdot, \star))}^-
(x_1,x_3,t_1,t_3,\lambda)}{(\lambda v_3 %v_2
+(1-\lambda)v_1)^2}
 \\\geq& -\frac{\theta(1-q)}{1+2\theta}\frac{1}{\sigma}\frac{(\lambda v^{1-1/\theta}_3+(1-\lambda) v^{1-1/\theta}_1)}{(\lambda v_3+(1-\lambda)v_1)^2}(\mathcal{C}^-_{a^{\theta}})^{1/\theta}(x_1,x_3,\lambda).
\end{align*}
%Define $ \rho:=\min \left\{d(x_1,\partial\Omega),d(x_3,\partial \Omega)\right\}>0.$
Then as in Proposition \ref{approximate concavity parabolic application a(x)u^gamma}, 
%we can take \tb{$\sigma:=\frac{1-q}{2}\frac{ \mathfrak{m}_{\rho} }{(\lambda v_3+(1-\lambda)v_1)^2}$} \tor{NO, andava bene prima, qui no ($\alpha$ è cambiato)} and this leads to
%\begin{align*}
% \mathcal{C}_{u^\alpha}(x_1,x_3,t_1,t_3, \lambda)\geq& -\frac{2\theta}{1+2\theta}\frac{(\lambda v^{1-1/\theta}_3+(1-\lambda) v^{1-1/\theta}_1)}{ \mathfrak{m}_{\rho} }(\mathcal{C}^-_{a^{\theta}})^{1/\theta}(x_1,x_3,\lambda)\\ \geq & -\frac{2\theta}{1+2\theta}\frac{||u||_\infty^{\frac{(\theta-1)(1-q)}{1+2\theta}}}{ \mathfrak{m}_{\rho} }(\mathcal{C}^-_{a^{\theta}})^{1/\theta}(x_1,x_3,\lambda).\qedhere
%\end{align*}
we give an estimate of $\sigma$: since $\Btilde$ in \eqref{defn psi} is given by
$$ \Btilde(x,s,\tilde{p},t) = \frac{1+\theta+\theta q}{\theta(1-q)} %\frac{1-\alpha}{\alpha}
\frac{|\tilde{p}|^2}{s}+\frac{\theta(1-q)}{1+2\theta} a(x) s^{-\frac{1+\theta}{\theta}}%s^{\frac{\alpha-1}{\alpha}}b\left(x,s^{1/\alpha}, t^{\beta}\right) 
$$
we have (for the sake of simplicity, here we straightway get rid %ride 
of the $\tilde{\xi}$-term)
\begin{align*}
\sup_{z\in [v_2,\lambda v_3+(1-\lambda)v_1]} \partial_s\Btilde(x_2,z,\tilde{\xi},t_2) &= \partial_s\Btilde(x_2,\lambda v_3+(1-\lambda)v_1,\tilde{\xi},t_2) \\
&\leq -\frac{(1+\theta)(1-q)}{1+2\theta} a(x_2) \big(\lambda v_3+(1-\lambda)v_1\big)^{-\frac{1+2\theta}{\theta}} 
%\\
%&\leq -\frac{(1+\theta)(1-q)}{1+2\theta} a(x,t) \norm{u}_{\infty}^{-(1-q)} 
\end{align*}
hence
\begin{align*}
 \mathcal{C}_{u^\alpha}(x_1,x_3,t_1,t_3, \lambda) &\geq -\frac{\theta(1-q)}{1+2\theta}\frac{1}{\sigma}\frac{(\lambda v^{1-1/\theta}_3+(1-\lambda) v^{1-1/\theta}_1)}{(\lambda v_3+(1-\lambda)v_1)^2}(\mathcal{C}^-_{a^{\theta}})^{1/\theta}(x_1,x_3,\lambda)\\ 
& \geq -\frac{\theta}{1+\theta} \frac{1}{a(x_2)} \big(\lambda v_3+(1-\lambda)v_1\big)^{\frac{1}{\theta}} (\lambda v^{1-1/\theta}_3+(1-\lambda) v^{1-1/\theta}_1)(\mathcal{C}^-_{a^{\theta}})^{1/\theta}(x_1,x_3,\lambda) \\ 
& \geq -\frac{\theta}{1+\theta} \frac{1}{m} \norm{u}_{\infty}^{\frac{\theta(1-q)}{1+2\theta}} (\mathcal{C}^-_{a^{\theta}})^{1/\theta}(x_1,x_3,\lambda). \qedhere
\end{align*}

\end{proof}

 \begin{proof}[Proof of Theorem \ref{thm_approximate concavity parabolic application a(x)u^gamma}]
 It is a consequence of Proposition \ref{prop_theta_concav_estim} (see also %and 
 Remark \ref{rem_rough_estimate_osca}).
 \end{proof}

\begin{remark}\label{rem_time_dependence}
The above perturbative results can be extended also to time-dependent nonlinearities $a(x,t)$. 
If, for instance, $a(x,t)$ is bounded and $a(x,t)\to \bar{a}(x)$ as $t\to +\infty$, it makes sense to investigate the parabolic stability condition for $b$. 
Moreover, if $a>0$ in $\Omega \times (0,+\infty)$, one can consider the infimum over $\Omega_{\rho} \times [\rho, +\infty)$, where $\rho:=\min \big\{d(x_1,\partial\Omega),d(x_3,\partial \Omega), t_1, t_3\big\}>0$.

If instead $a(x,t) \sim a(x) t^{\gamma}$, then one may consider a truncation argument (see Remark \ref{rem_troncamento}) so that $a$ becomes bounded. 
In such a case, anyway, the perturbative estimate will depend on $T$.
We omit the details, and leave them to the interested reader. 
\end{remark}

\begin{remark}[Comments on $u(\cdot, 0) \neq 0$]
\label{rem_u0_nonzero}
	In the whole Subsection \ref{subsec_harmon_boundar} we considered only zero initial conditions $u(\cdot, 0)=u_0=0$, in accordance to \cite{IsSa13, kenningtonparabolic, IsSa16, ILS20}. 
	Many of our arguments, anyway, could be adapted also to nonzero initial conditions. 
	First, assume, in place of \hyperlink{H1}{\((H1)\)}, the following slightly stronger assumption: 
	\begin{itemize}
		\item[]\hypertarget{H1*}\((H1^*)\) 
	There exist $T>0$, $q\in [0,1)$ and $\gamma \in [0,1]$ such that for all $M>0$ there exists $k=k(M,T)>0$ verifying
		$$b(x,s, t)\geq k t^{\gamma} s^q \quad \hbox{for all $(x,s,t)\in \Omega\times(0,M]\times (0,T]$}.$$
\end{itemize}
As in \hyperlink{H2}{\((H2)\)}, $\norm{u}_{\infty}$ takes the place of $M$. 
Notice that $f(s)=(1-s)^p$ (studied in Proposition \ref{prop_kenn88}) satisfies \hyperlink{H1}{\((H1)\)} but not \hyperlink{H1*}{\( (H1^*) \)}.

Assuming
	\begin{itemize}
		\item $u_0 \in C(\Omegabar)$, $u_0=0$ on $\partial\Omega$,%
		\footnote{We highlight that the requirement \eqref{eq_cond_dOmega} is related to the application of the comparison principles, see Proposition \ref{comparison principle dickstein}. 
		By requiring more regularity on $b$ (see \hyperlink{H2*}{\( (H2^*) \)}), this condition can be dropped (see Proposition \ref{comparison principle dickstein_2}).}
		\begin{equation}\label{eq_cond_dOmega}
			u_0\geq \delta d_{\Omega} \quad \hbox{for some $\delta>0$},
		\end{equation} 
		$u_0$ is $\alpha$-concave for some (suitable) $\alpha$, and 
		$$u(\cdot, t)\geq u_0 \quad \hbox{for each $t>0$}$$
		(for instance, when \hyperlink{H3}{\( (H3) \)} holds, we can require $u_0$ to be a subsolution of the stationary problem, see Corollary \ref{corol_monotonicity}),%
		\footnote{The set of functions satisfying these conditions is nonempty. 
		For example, consider $\Omega = B_1$, $b(x,s,t)=a(x)\in L^\infty(\Omega)$ with $a\geq m >0$. Let $u_0(x)=\delta d_{\Omega}(x) =\delta(1-|x|)$, $\delta$ small enough, on $B_1\setminus B_{\frac{1}{2}}$ and extended smoothly and in a concave way to $B_{\frac{1}{2}}$. Then $-\Delta u_0=\delta \frac{n%N
		-1}{|x|}\leq a(x)$ in $B_1 \setminus B_{\frac{1}{2}}$ and we can require also $-\Delta u_0\leq a(x)$ in $B_{\frac{1}{2}}$. 
		Thus $u_0$ satisfies all the assumptions ($\alpha=1$). Less trivial examples (such as perturbations of the present examples) hold as well.}
	\end{itemize}
	all the arguments still hold, except for the ones in Proposition \ref{minimum not on the boundary}, case $t_1=0$, $t_3 \in (0,+\infty)$ and $x_1, x_3 \in \Omega$. 
	As a matter of facts, in this case we have
	$$\limsup_{\tau \to 0} v_t(\tau) \geq \limsup_{\tau \to 0} \frac{v(x,\tau)-v(x,0)}{\tau},$$
	with $v(x,0)=u_0^{\alpha}(x)>0$, and we are not able to show that this limit explodes to $+\infty$, no matter who $\alpha$ is. More in details, by assuming for simplicity $b(x,s,t)= a(x) s^q$, $q \in [0,1)$, $a(x) \geq m>0$, then an improvement of Proposition \ref{comparison principle dickstein} leads to 
	$$u(x,t) \geq v(x,t) + e^{Lt} S(t) (u(x,0)-v(x,0)),$$
	for any supersolution $u$ and any subsolution $v$; here $L= \frac{q m}{\norm{u}_{\infty}^{1-q}}\geq 0$ and $S(t)$ is the heat semigroup. 
	Exploiting this improvement in the proof of Lemma \ref{lem_behav_boundar_omega} (\textit{i}), we obtain 
	$$u(x,t) \geq C e^{-\lambda_1 t} t^{\frac{1}{1-q}}\varphi_1(x) + e^{Lt} S(t) u_0$$
	with $C=((1-q) m)^{\frac{1}{1-q}}$.
	As a consequence, straightforward computations imply, if $\beta=1$ and $u_0 \in C^2(\Omegabar)$, 
	$$\limsup_{\tau\to 0} v_t(x,\tau) \geq \alpha u_0^{\alpha-1}(x) \Delta u_0(x)+ \alpha L u_0^{\alpha}(x) + C \alpha u_0^{\alpha-1}(x) \varphi_1(x) \limsup_{\tau \to 0} \tau^{\frac{1}{1-q}-1}, $$
	where the right hand side corresponds to $\alpha u_0^{\alpha-1}(x) h(x)$ with
	$$h := \Delta u_0+ m %\frac{m}{\alpha} 
	\varphi_1 \; \hbox{ if $q=0$}, \quad h :=\Delta u_0 + L u_0 \, \hbox{ if $q \in (0,1)$}.$$
	As abovementioned, these are finite quantities (positive under suitable assumptions on $u_0$), and thus do not allow to conclude in Proposition \ref{minimum not on the boundary}. 
	
	With similar computations we see that, if $\beta<1$, then we actually obtain $\limsup_{\tau \to 0} v_t(\tau) = +\infty$: 
	with such an information at disposal, we can conclude Proposition \ref{minimum not on the boundary} and apply it to obtain (arguing as in Theorem \ref{approximate concavity parabolic with boundary conditions} and Corollary \ref{corollario con concavità})
		$$\min_{\Omegabar\times \Omegabar\times[0,+\infty]\times [0,+\infty]\times[0,1]}\mathcal{C}_{u^\alpha(\cdot, \star^{\beta})} \geq -\frac{\alpha}{\sigma}\frac{\mathcal{C}^-_{\tilde{g}(\cdot, u^{\alpha}(\cdot,\star^{\beta}) ) }(x_{1},x_{3},t_{1},t_{3},\lambda)}{(\lambda u^\alpha(x_{3},t_{3}^{\beta})
		+(1-\lambda)u^\alpha(x_{1},t_{1}^{\beta})
		)^2} - \xi^{n+1}\mc{HC}_{\star^{1-\beta}}(t_{1},t_{3},\lambda),$$
	where $\tilde{g}(x,s):=a(x) s^{\frac{3\alpha-1+q}{\alpha}} $. 
		While the first term, as in the case $u_0=0$ (see the proof of Theorem \ref{thm_main_weighted_le}), can be positively estimated under suitable assumptions on $\alpha, q$ and $a(x)$, we see that for any $\beta<1$ the second term remains negative (notice that $t_1$ or $t_3$ may go to zero as $\beta \to 1^-$). 
	Thus we are not able to deduce any exact concavity in such a case.
	
	Overall, we observe that conservation of concavity for nontrivial initial data is a delicate issue: in the case of porous medium equations, 
	indeed, it has been proved that $\frac{1}{2}$-concavity is preserved \cite{lee2008parabolic}, while there are counterexamples in the case of $\alpha$-concavity, $\alpha<\frac{1}{2}$, \cite{ChauWeinkove}; this is in contrast with the general spirit of the present paper, where information on some $\alpha$ implies information also on smaller exponents. 
We leave thus $u(\cdot, 0)\neq 0$ as an open problem.
\end{remark}

%%%%%%%%%%%%%%%%%%%%%%%%%%%%%%%%%%%%%%%%%%%%%%%%%%%%%%%%%%%%%%%%%%%%%%%%
%%%%%%%%%%%%%%%%%%%%%%%%%%%%%%%%%%%%%%%%%%%%%%%%%%%%%%%%%%%%%%%%%%%%%%%%
\section{Further results on the weighted torsion problem}
\label{sec_torsion}

Let us deal with the case $q=0$ in \eqref{equazione a(x)u^gamma}.
First, we can find more explicit constants and comparison functions than %in 
Corollary \ref{corollario ulam}; moreover, we have more freedom in the choice of the sign of $a$.

\begin{proposition}
	Let $\Omega$ satisfy \eqref{eq_ipost_domain}, and $u,v$ be some functions such that \eqref{regularity conditions} and \eqref{boundary and initial conditions} hold. Suppose
	\begin{equation*}
		u_t-\Delta u=a(x, t) \quad \emph{in }\Omega\times (0,+\infty), \qquad 		v_t-\Delta v=K\quad\emph{in }\Omega\times (0,+\infty)
	\end{equation*}
	where $K>0$, $a\colon\Omega \times (0,+\infty) %\R
	 \to \R$ and there exist $m,M\in\R$ such that $m\leq a\leq M.$ 
	Then $v$ is $\frac{1}{2}$-concave and
	\[ \norm{u^{1/2}-v^{1/2}}_\infty\leq e^{R}\sqrt{\max\left\{|m-K|,|M-K|\right\}}, \]
	where $R$ is such that $\Omega\subseteq B(0,R)$.
\end{proposition}
\begin{proof}
	Let us subtract the two equations obtaining that $w:=u-v$ satisfies the problem 
	\begin{equation*}
		\begin{cases}
			w_t-\Delta w=a(x, t)-K &\text{in }\Omega\times (0,+\infty),\\
			w=0 & \text{on }\partial \Omega\times (0,+\infty),\\ 
			w=0& \text{on } \Omega \times \{0\}. \\
		\end{cases}
	\end{equation*}
	Set $\psi:= e^{2R}-e^{x_1+R}$. 
	Then
	\[ \psi_t - \Delta \psi =e^{x_1+R}\geq 1 \quad \text{in }\Omega \times (0,+\infty).\]
	Set $f(x,t):= a(x, t)-K$, hence we have $(\partial_t-\Delta)(w-\psi\sup_{\Omega\times(0,+\infty)} f^+)\leq0$ in $\Omega\times (0,+\infty)$ and $w-\psi \sup_{\Omega\times(0,+\infty)} f^+\leq 0$ on the parabolic boundary (see \eqref{eq_def_parabol_boundary}).
	Then by the maximum principle \cite[Corollary 2.5]{lieberman1996}, it holds $$\sup\limits_{\Omega\times (0,+\infty)}(w-\psi\sup_{\Omega\times(0,+\infty)} f^+)\leq 0,$$ 
	so
	$$\sup_{\Omega\times (0,+\infty)}w\leq e^{2R}\sup_{\Omega\times(0,+\infty)} |f|.$$
	In a similar way we obtain
	$$\inf_{\Omega\times (0,+\infty)}w\geq -e^{2R}\sup_{\Omega\times(0,+\infty)} |f|.$$
	Then it holds
	\begin{equation*}
		\norm{u^{1/2}-v^{1/2}}_\infty\leq	\norm{u-v}^{1/2}_\infty \leq e^{R}\sqrt{\sup_{\Omega\times(0,+\infty)} |f|} \leq %=
		e^R\sqrt{\max\left\{|m-K|,|M-K|\right\}}. 
	\end{equation*}
	Finally, $v$ is $\frac{1}{2}$-concave by Proposition \ref{approximate concavity parabolic application a(x)u^gamma}.
\end{proof}

When dealing with the torsion problem ($q=0$), more general classes of concave weights $a(x)$ can be considered, in particular by allowing $a(x)=0$ on $\partial \Omega$. In particular, instead of \hyperlink{H1}{\((H1)\)} we can consider the following condition
\begin{itemize}
\item[]\hypertarget{H1'}\((H1')\) 
There exists $T>0$, $k>0$, $\omega \geq 0$ and $\gamma \in [0,1)$,
		such that, 	
		$$b(x,t)\geq k d_{\Omega}^\omega(x) t^{\gamma} \quad \hbox{for all $(x,t)\in \Omega\times (0,T]$}.$$
\end{itemize}

\begin{remark}\label{rem_concav_distance}
The presence of $d_{\Omega}$ in assumption \hyperlink{H1'}{\((H1')\)} is due to the following fact: if $a(x)$ is a nonnegative, nontrivial concave function, then it is positive in $\Omega$ (but not necessarily in $\Omegabar$), and there exists $C>0$ such that
$$a(x) \geq C d_{\Omega}(x) \quad \hbox{in $\Omegabar$}.$$
\end{remark}

The fact that, differently from \hyperlink{H1}{\((H1)\)}, we focus on $q=0$ in \hyperlink{H1'}{\((H1')\)}, allows us to compare our solutions, time-scaled by $\beta=2$, with the one of the equation $ U_t -\Delta U + U = k d_{\Omega}^{\omega}(x) t^{\gamma} $ for which a representation formula holds. 
Thus combining \cite[Lemma 3]{IsSa13} and Proposition \ref{comparison principle dickstein}, we obtain an alternative version of Lemma \ref{lem_behav_boundar_omega}.

\begin{lemma}\label{lem_behav_IshSal}
Let $n\geq 2$, $\Omega$ be bounded, convex and smooth, and $u$ be a function such that \eqref{regularity conditions}, \eqref{eq_general_bxu} and \eqref{boundary and initial conditions} hold.
Suppose $b\colon\Omega\times(0,+\infty)\to (0,+\infty)$ satisfies \hyperlink{H1'}{\((H1')\)}, and let $T, \gamma, \omega$ given therein. 
Then
\begin{itemize}
\item[(i)]
% there exists $C>0$ such that \tor{C'è qualcosa che non va, non può valere per $x$ vicino al bordo! Dopo viene applicato per un $x_1$ specifico, quindi dovrebbe bastare localmente. Ma da dove viene fuori questa stima?}
%$$u(x,t) \geq C t^{\frac{2+2\gamma+\omega}{2}} \quad \hbox{in $\Omega \times (0,T)$};$$
for any $K \Subset \Omega$ there exists $C_K>0$ such that
$$u(x,t) \geq C_K %e^{-\lambda_1 t} 
t^{1+\gamma} \quad \hbox{in $K \times (0,T)$};$$
\item[(ii)] for any $x_0 \in \partial \Omega$ there exists $\delta>0$ such that 
$$u(x_0+t\nu,t^2)\geq \delta t^{2+2\gamma+\omega} \quad \hbox{for $t \in (0,T)$ small},$$
where $\nu$ is the interior normal to $\partial \Omega$ at $x_0$.
\end{itemize}
\end{lemma}

\begin{proof}
With no loss of generality we can assume $K=\overline{B}_r \subset B_R \subset \Omega$;
by \hyperlink{H1'}{\((H1')\)}, we have $b(x,t) \geq C'_R t^{\gamma}$ in $B_R\times (0,T)$. Arguing as in the proof of Lemma \ref{lem_behav_boundar_omega} \textit{(i)}, with \(q=0\), let $\varphi_1^R$ be the first Dirichlet eigenfunction in $B_R$, normalized $\norm{\varphi_1^R}_{\infty}=1$; we set
$$w(x,t):= C e^{-\lambda_1^R t} t^{1+\gamma} \varphi_1^R(x)$$
and $C :=\frac{C'_R}{1+\gamma}$. % \min\{\frac{C'_K}{1+\gamma}, \inf_{\partial K \times (0,+\infty)} u\} $. 
A straightforward computation shows that $w$ satisfies \eqref{boundary and initial conditions} 
together with
$$w_t-\Delta w = C (1+\gamma) e^{-\lambda_1^R t} t^{\gamma} \varphi_1^R(x) \quad \text{in }B_R\times (0,T)$$
and hence, by definition of $C$, 
$$w_t-\Delta w = 
\frac{1+\gamma}{C'_R} C e^{- \lambda_1^R t} \varphi_1^R(x) (C'_R t^{\gamma}) \leq C'_R t^{\gamma} \quad \text{in } B_R\times (0,T).$$
Then by Proposition \ref{comparison principle dickstein}, $u\geq w$ on $B_R\times (0,T)$ and hence $u \geq C (\inf_{B_r} \varphi_1^R) t^{1+\gamma} =C_r t^{1+\gamma}$ in $B_r\times (0,T)$.
The second claim follows by \cite[Lemma 3]{IsSa13}.
\end{proof}

Thanks to this result, we are able to show the counterpart of Proposition \ref{minimum not on the boundary} in the case $\beta=2$, that is

\begin{proposition}\label{prop_con_H1'}
In the setting of Proposition \ref{minimum not on the boundary}, 
suppose $b$ satisfies \hyperlink{H1'}{\((H1')\)}, and let $\alpha \in (0, \frac{1}{2+2\gamma+\omega})$.
Then $\mathcal{C}_{u^\alpha(\cdot, \star^{2})}$ cannot achieve any negative minimum at $(x_1,x_3,t_1,t_3,\lambda)\in \Omegabar \times \Omegabar \times [0,+\infty]\times [0,+\infty]\times (0,1) $ such that one of $t_1,t_3\in \left\{0\right\}$, or one of $x_1,x_2,x_3\in \partial \Omega$.
\end{proposition}

\begin{proof}
The proof proceeds similarly to the one of Proposition \ref{minimum not on the boundary}. 
In particular if $t_1=0$, $t_3\in (0,+\infty)$, $x_1\in \Omega$ and $x_3\in \Omega $, by Lemma \ref{lem_behav_IshSal} (\textit{i}),
	$$\limsup%\lim
	_{\tau \to 0} v_t(x_1,\tau)\geq C^{\alpha} \limsup%\lim
	_{\tau \to 0} \frac{1}{\tau^{1-(2+2\gamma% +\omega
	)\alpha 	}} = +\infty$$
	by the assumption $\alpha < \frac{1}{2+2\gamma+\omega}$ which implies the contradiction.
If $t_1=0$, $t_3\in (0,+\infty)$, $x_1\in \partial \Omega$, $x_3 \in \Omega$, by Lemma \ref{lem_behav_IshSal} (\textit{ii})
	$$\frac{u^\alpha(x_1 +t\nu , t^{2})}{t}\geq \delta^\alpha t^{\alpha(2+2\gamma+\omega)-1}\quad \hbox{for $t>0$ small},$$
which positively diverges as $t\to 0^+$ again by assumptions and thus again a contradiction.
\end{proof}

With this tool, we can prove a refinement of Theorem \ref{thm_main_weighted_le} when $q=0$.

\begin{proposition}[Weighted torsion]
\label{prop_weight_tors}
	Let $\gamma \in [0,\frac{1}{2}]$, $\Omega$ satisfy \eqref{eq_ipost_domain}, and $u$ be a function such that \eqref{regularity conditions}, \eqref{boundary and initial conditions} and
	$$	u_t-\Delta u= c(t) a(x) \quad\text{in }\Omega\times (0,+\infty).$$
	We assume $c$ nondecreasing, with $t \mapsto c(t^2)$ being $\frac{1}{2\gamma}$-concave %$\frac{1}{\gamma}$-concave, 
	and
$$c(t) \geq t^{\gamma}.$$
If $\gamma<\frac{1}{2}$ and $a\geq 0$ is $\theta$-concave for some $\theta\geq \frac{1}{1-2\gamma}$, then $u(\cdot, \star^2)$ is $\frac{\theta}{2\theta+2\theta \gamma +1}$-concave. 
If $a$ is constant, then $u(\cdot, \star^{2})$ is $\frac{1}{2+2 \gamma }$-concave.
\end{proposition}

\begin{proof}
By Remark \ref{rem_concav_distance} we may assume $a \geq C d_{\Omega}^{\omega}$ for $\omega=\frac{1}{\theta}$. By Remark \ref{rem_troncamento} we can restrict to study
$$u_t - \Delta u = c_T(t) a(x) \quad \hbox{in $\Omega \times (0,+\infty)$}$$
where $T$ is fixed and $c_T(t)=c(t)$ for $t \in [0,T]$ and $c_T(t)=c(T)$ for $t\geq %\leq 
T$. 
Clearly $b(x,t):=c_T(t) a(x)$ satisfies \hyperlink{H1'}{\((H1')\)} and \hyperlink{H2}{\((H2)\)}, as well as the parabolic stability condition thanks again to Remark \ref{rem_troncamento}.
To apply Proposition \ref{prop_con_H1'} %Corollary \ref{concavity parabolic with boundary conditions} 
with $\alpha \in (0, \frac{1}{2+2\gamma+\omega}) \subset (0,1)$, we observe that $s \mapsto s^{\alpha-1} b(x,t)$ is strictly decreasing and moreover $(x,s,t) \mapsto s^{\frac{3\alpha-1}{\alpha}}b(x,t^2) = s^{\frac{3\alpha-1}{\alpha}} c_{T}(t^2) a(x) $ is concave %\tor{Non vero in generale, non posso dire che $c(t^2)$ è $1/(2\gamma)$-concava dalle ipotesi che ho}
 if $\frac{3\alpha-1}{\alpha} \geq 0$ and $\frac{3\alpha-1}{\alpha} +2 \gamma + \frac{1}{\theta} = 1$ by Proposition \ref{alpha concavity property} (\textit{iv}). 
Both conditions imply that the optimal $\alpha$ is given by $\alpha=\frac{\theta}{2\theta+2 \theta \gamma +1}$ (recall also Proposition \ref{alpha concavity property}). Noticed that the solution of the stationary problem $-\Delta u = a(x)$ is $\frac{\theta}{2\theta+1}$-concave, and $\frac{\theta}{2\theta+2 \theta \gamma +1} \leq \frac{\theta}{2\theta+1}$, then 
we obtain that $u(\cdot, \star^2)$ is $\frac{\theta}{2\theta+2\theta \gamma +1}$-concave.
\end{proof}

\appendix

%%%%%%%%%%%%%%%%%%%%%%%%%%%%%%%%%%%%%%%%%%%%%%%%%%%%%%%%%%%%%%%%%%%%%%%%
%%%%%%%%%%%%%%%%%%%%%%%%%%%%%%%%%%%%%%%%%%%%%%%%%%%%%%%%%%%%%%%%%%%%%%%%
\section{Comparison principles}
\label{sec_max_priniciples}

We state now a comparison principle, which is the nonautonomous counterpart of \cite[Corollary 2.3]{Dickstein} (see also \cite{CazenaveDicksteinEscobedo}). 
We recall that $w=w(x,t)$ is a subsolution (resp. supersolution) of 
\begin{equation} \label{eq_sub_super_sol}
		\begin{cases}
			w_t-\Delta w = b(x,w,t)&\hbox{in }\Omega\times(0,T),\\
			w=0 &\hbox{on } \partial \Omega\times(0,T),\\
			w=u_0&\hbox{on } \Omega \times \{0\},
		\end{cases}
	\end{equation}
if the above three equalities are substituted by $\leq$ (resp. $\geq$).

\begin{proposition}[Comparison and uniqueness, I]
\label{comparison principle dickstein} 
Let $\Omega \subset \R^n$ be a bounded open domain with the interior sphere property. 
Let $T>0$ and $b\colon\Omega\times(0,+\infty) \times (0,T)\to [0,+\infty)$ satisfying \hyperlink{H2}{\((H2)\)} and assume $b(x,\cdot,t)$ continuous for all $(x,t)\in \Omega\times(0,T)$.
	Let $u_0\in C(\Omegabar)$ with $u_0\geq \delta d_\Omega$ for some $\delta>0$. % or $u_0\equiv 0$. 
	Let $u,v\in C^2_x(\Omega)\cap C^1_t((0,T])\cap C(\Omegabar\times[0,T])$ and $(v(\cdot,t)-u(\cdot,t)%u-v
	)^+\in H^1(\Omega)$ for each $t$.
	Assume $u$ is a positive supersolution and $v$ is a subsolution of \eqref{eq_sub_super_sol}.
Then $u(\cdot,t)\geq v(\cdot,t)$ for all $0\leq t\leq T$. 
In particular, there exists at most one positive solution of \eqref{eq_sub_super_sol}.

The thesis still holds with $u_0 \equiv 0$ by additionally assuming \hyperlink{H3}{\( (H3) \)}.
\end{proposition}
\begin{proof}
Consider first the case $u_0\geq \delta d_\Omega$.
	Based on Hardy's inequality, it holds 
	\begin{equation}\label{hardy's inequality}
		\int_{\Omega}\dfrac{\varphi^2}{d^2_\Omega}\leq C\int_\Omega \abs{\nabla \varphi}^2,
	\end{equation}
	for all $\varphi\in H^1_0(\Omega)$. Being $b$ nonnegative, we have $u_t-\Delta u\geq 0$ and thus by comparison and the semigroup representation of the heat equation we obtain $u(\cdot,t)\geq S(t)u(\cdot,0)$, $S(t)=e^{t \Delta}$. 
 In particular there exists $\gamma>0$ such that $u(\cdot,t)\geq \gamma d_\Omega$ for all $t\in[0,T]$. 
 Taking $t\in (0,T)$, testing the inequalities satisfied by $u$ and $v$ with $(v(\cdot,t)-u(\cdot,t))^+$ and taking the difference we obtain
	\begin{align*}
		\dfrac{1}{2}\dfrac{d}{dt}\int_{\Omega}\abs{\big(v(x,t)-u(x,t)\big)^{+}}^2dx& 
		+ \int_{\Omega}\abs{D_x (v(x,t)-u(x,t))^+}^2dx
		\\ \leq& \int_{\left\{v(\cdot,t)>u(\cdot,t)\right\}}\big(b(x,v, t)-b(x,u, t))(v(x,t)-u(x,t)\big)^+dx\\
		\leq& L_M \int_{\left\{v(\cdot,t)>u(\cdot,t)\right\}} \dfrac{\abs{\big(v(x,t)-u(x,t)\big)^{+}}^2}{u(x,t)}dx,
	\end{align*}
	with $M:=\max \left\{||u||_{L^\infty((0,T)\times\Omega)},||v||_{L^\infty((0,T)\times\Omega)}\right\}$. 
	Since 
	\[\dfrac{1}{u(\cdot,t)}\leq \dfrac{1}{\gamma d_\Omega}\leq \dfrac{\varepsilon}{d^2_\Omega}+C(\varepsilon) \quad \text{for all }\varepsilon>0,\]
	we deduce that 
\begin{align*}
	\dfrac{1}{2}\dfrac{d}{dt}\int_{\Omega}\abs{(v(x,t)-u(x,t))^{+}}^2dx&+\int_{\Omega}\abs{D_x(v(x,t)-u(x,t))^+}^2dx\\\leq& \varepsilon \int_{\Omega}\dfrac{\abs{(v(x,t)-u(x,t))^{+}}^2}{d^2_\Omega(x)}dx+C(\varepsilon)\int_\Omega \abs{(v(x,t)-u(x,t))^{+}}^2dx.
\end{align*}
	Applying \eqref{hardy's inequality} and choosing $\varepsilon>0$ sufficiently small, we then obtain 
	\[ \dfrac{1}{2}\dfrac{d}{dt}\int_\Omega \abs{(v(x,t)-u(x,t))^{+}}^2dx\leq C\int_\Omega\abs{(v(x,t)-u(x,t))^{+}}^2dx, \]
	from which the result follows by Gronwall inequality
	$$\int_\Omega \abs{(v(x,t)-u(x,t))^{+}}^2dx \leq e^{2CT} \int_\Omega\abs{(v(x,0)-u(x,0))^{+}}^2dx.$$
	
	Now let us deal with the case $u_0\equiv 0$. 
	Since $u_t-\Delta u\geq 0$, by \cite[Remark 6.1]{CazenaveDicksteinEscobedo} (see also Lemma \ref{Hopf lemma in parabolic case}) we have that $u(\cdot,t)\geq S(t-r)u(\cdot,r)$ for all $0< r < t \leq T$ and hence, being $u(\cdot,r)>0$, we obtain 
	 $u(\cdot,t)\geq \delta(t)d_\Omega$ for all $t\in (0,T)$ with $\delta(t)>0$. 
	 Let us fix $t \in (0,T)$. By the previous arguments, applied to %with initial time $t$ to 
$u(\cdot, \star + t)$ and $v$ (thanks to $b(x,u,s+t)\geq b(x,u,s)$ given by \hyperlink{H3}{\( (H3) \)}), %\tor{Non posso, perché $u(\cdot, \star+t)$ soddisfa l'equazione con $b(x,u, s+t)$, non la stessa equazione di $u$! Servirebbe (H3) o altro}, 
we conclude that $u(\cdot,s+t)\geq v(\cdot,s%t
)$ for all $0<s<T-t$. 
	 The result follows by letting $t$ %$s$
	 tend to $0$. 
\end{proof}

By requiring a more restrictive regularity of $b(x,\cdot,t)$, we show now that the request $u \geq \delta d_{\Omega}$ can be dropped in Proposition \ref{comparison principle dickstein}; notice anyway that such regularity is satisfied by $b(x,s,t)=a(x) s^q$ only for $q \in [\frac{1}{2},1]$.
We refer also to \cite[Section A.1]{ILS20} and references therein for other comparison principles which involve nonincreasing $b(x,\cdot,t)$. 
\begin{proposition}[Comparison and uniqueness, II]
\label{comparison principle dickstein_2} 
Let $\Omega \subset \R^n$ be a bounded open domain with the interior sphere property. Let $T>0$ and $b\colon\Omega\times(0,+\infty) \times (0,T)\to [0,+\infty)$ satisfying \hyperlink{H2*}{\( (H2^*) \)} and $b(x,\cdot,t)$ continuous for all $(x,t)\in \Omega\times(0,T)$.
	Let $u_0\in C(\Omegabar)$. Let $u,v\in C^2_x(\Omega)\cap C^1_t((0,T])\cap C(\Omegabar\times[0,T])$ and $(u-v)^+\in H^1(\Omega)$.
	Assume $u$ is a positive supersolution and $v$ is a subsolution of \eqref{eq_sub_super_sol}.
Then $u(\cdot,t)\geq v(\cdot,t)$ for all $0\leq t\leq T$. 
In particular, there exists at most one positive solution of \eqref{eq_sub_super_sol}.
\end{proposition}

\begin{proof}
The proof follows the lines of the one in Proposition \ref{comparison principle dickstein}, with no need of the Hardy inequality, but by exploiting (recall $2\omega \geq 1$) $D^{1,2}(\Omega) %\hookrightarrow L^{2^*}(\Omega) 
\hookrightarrow L^{2\omega}(\Omega)$ and Hölder inequality.
%, where we highlight that $ L^{2\omega}(\Omega)$ is well defined as Lebesgue space by the assumption $2\omega \geq 1$.
\end{proof}

We state now a condition for monotonicity of solutions (see also \cite[Proposition 6.12]{CazenaveDicksteinEscobedo} and \cite[Lemma 4.1, page %pag 
199]{pao}).
\begin{corollary}[Monotonicity]
\label{corol_monotonicity}
In the assumptions of Proposition \ref{comparison principle dickstein} or Proposition \ref{comparison principle dickstein_2}, assume in addition \hyperlink{H3}{\( (H3) \)}.
Let $u$ be a positive solution to 
	\begin{equation*}
		\begin{cases}
			u_t-\Delta u = b(x,u,t)&\hbox{in } \Omega\times(0,+\infty),\\
			u=0 &\hbox{on } \partial \Omega\times(0,+\infty),\\
			u=u_0&\hbox{on } \Omega \times \{0\}.
		\end{cases}
	\end{equation*}
	Assume moreover that 
\begin{equation}\label{eq_u_grter_u0}
u(\cdot, t) \geq u_0 \quad \hbox{for each $t>0$}.
\end{equation}
Then $u_t \geq 0$ in $\Omega\times (0,+\infty)$. 
Extending the assumptions on $b$ also to $t=0$, a condition to ensure \eqref{eq_u_grter_u0} is that $u_0$ is a subsolution of the stationary problem
	\begin{equation*}
	\begin{cases}
		-\Delta w = b(x,w,0)&\hbox{in } \Omega,\\
		w=0 &\hbox{on } \partial \Omega.\\
	\end{cases}
\end{equation*}
\end{corollary}
\begin{proof}
Fix $\tau>0$ and set $\tilde{u}(x,t):=u(x,t+\tau)$, then $\tilde{u}$ satisfies 
	\[	\begin{cases}
		\tilde{u}_t -\Delta \tilde{u} =b(x,\tilde{u},t+\tau) \geq b(x,\tilde{u}, t) &\text{in }\Omega\times (0,+\infty),\\
		\tilde{u}=0 & \text{on }\partial \Omega\times (0,+\infty),\\ 
		\tilde{u} \geq u_0 & \text{on } \Omega \times \{0\}.
	\end{cases}\]	
	Since \hyperlink{H2}{\((H2)\)} or \hyperlink{H2*}{\( (H2^*) \)} holds, by Proposition \ref{comparison principle dickstein} or \ref{comparison principle dickstein_2} we have $\tilde{u}(x,t)\geq u(x,t)$ for all $(x,t)\in\Omega\times(0,+\infty)$, that is $u(x,t+\tau)\geq u(x,t)$, which is the claim.
To see the last statement, it is sufficient to apply again Proposition \ref{comparison principle dickstein} or \ref{comparison principle dickstein_2} with $v=u_0$, after having observed that $b(x,s,0)\leq b(x,s,t)$ thanks to \hyperlink{H3}{\( (H3) \)}.
\end{proof}

To state the parabolic version of the Hopf lemma, we recall some usual notation for the domain and the boundary. 
Let $\Omega$ be a bounded open set in $\R^n$, $T>0$ and define the \emph{parabolic boundary} as 
\begin{equation}\label{eq_def_parabol_boundary}
\partial_p\Omega_T:=\big(\Omegabar\times[0,T]\big)\setminus\big( \Omega\times(0,T]\big).
\end{equation}
We observe that, if $\Omega$ satisfies the interior sphere condition, then every point $(\bar{x},\bar{t})\in \partial_p \Omega_T$ satisfies the \emph{spherical cap} condition, that is there exists an open ball $B_r(x_0,t_0)$, $x_0\ne \bar{x}$, such that $(\bar{x},\bar{t})\in \partial B_r(x_0, t_0)$ and 
$$C_r(\bar{x},\bar{t}):= B_r(x_0, t_0)\cap \left\{t<\bar{t}\right\}\subset \Omega\times (0,T].$$ 
We recall thus the Hopf lemma for parabolic equations, see e.g. \cite[Theorem 3, page 170]{ProtterWeitenberg1984}.
%\tr{} SE \bar{t}=0 L'AFFERMAZIONE E' BANALE PERCHE' L'INSIEME E' VUOTO (?)

\begin{lemma}[Parabolic Hopf lemma]
\label{Hopf lemma in parabolic case} 	
Let $\Omega$ be an open bounded set in $\R^n$ satisfying the interior sphere condition, $T>0$ and $u\in C^2_x(\Omega)\cap C^1_t((0,T])\cap C(\Omegabar\times[0,T])$.
	Suppose further
	\[ u_t+Lu\geq 0 \quad \text{in }\Omega\times(0,T]\]
	where \[Lu:=-\sum\limits_{i,j=1}^{n}a_{ij} D^2_{ij}u+\sum\limits_{i=1}^{n}b_i %b^i 
	D_{i}u\] 
	with $a_{ij}, b_i \in L^{\infty}(\Omega\times(0,T])\cap C(\Omega\times(0,T])$, and $L$ is uniformly elliptic. 
	Moreover, suppose that there exists a point $(\bar{x},\bar{t})\in \partial_p\Omega_T$ such that
	\[ u(\bar{x},\bar{t})<u(x,t)\quad \forall (x,t)\in C_r(\bar{x},\bar{t}). \]
	Let $\mathbf{e}\in \R^{n+1}$ be a direction such that $(\bar{x},\bar{t})+s\mathbf{e}\in \overline{C_r(\bar{x},\bar{t})}$ for some $s>0$, and assume that $\partial_{\mathbf{e}} u (\bar{x},\bar{t})$ exists. 
	Then $\partial_{\mathbf{e}} u (\bar{x},\bar{t})>0$.
\end{lemma}

%%%%%%%%%%%%%%%%%%%%%%%%%%%%%%%%%%%%%%%%%%%%%%%%%%%%%%%%%%%%%%%%%%%%%%%%
%%%%%%%%%%%%%%%%%%%%%%%%%%%%%%%%%%%%%%%%%%%%%%%%%%%%%%%%%%%%%%%%%%%%%%%%
\section{Properties of concavity functions}
\label{sec_prop_conc_fun}

 \begin{proposition}[Concavity of the product] \label{concavità prodotto}
Consider $\alpha,\beta \in (1%0
,+\infty)$ with $\alpha^{-1}+\beta^{-1}=1$, $K$ a convex set and $f,g\colon K \to (0,+\infty)$. 
Suppose on $f$ one of the following:
\begin{enumerate}[label=\textit{\roman*})]
	\item $m_1\leq f\leq M_1$ for some $m_1^{\alpha}\geq \frac{1}{2} M_1^{\alpha}>0$,
	\item $\mathcal{C}^-_{f^{\alpha}}=0$, 
\end{enumerate}
and on $g$ one of the following:
\begin{enumerate}[label=\textit{\roman*})]
	\item $m_2\leq g\leq M_2$ for some $m_2^{\beta}\geq \frac{1}{2} M_2^{\beta}>0$,
	\item $\mathcal{C}^-_{g^{\beta}}=0$. 
	\end{enumerate}
 Then, for each $x_1,x_3 \in K$ and $\lambda \in [0,1]$, it holds
\begin{align*}
	\mathcal{C}_{fg}(x_1,x_3,\lambda)\geq&
	 -(\mathcal{C}^-_{f^{\alpha}})^{1/\alpha}(x_1,x_3,\lambda) (\lambda g_3+(1-\lambda) g_1)-(\mathcal{C}^-_{g^{\beta}})^{1/\beta}(x_1,x_3,\lambda) (\lambda f_3+(1-\lambda) f_1)+\\&+(\mathcal{C}^-_{f^{\alpha}})^{1/\alpha}(x_1,x_3,\lambda)(\mathcal{C}^-_{g^{\beta}})^{1/\beta}(x_1,x_3,\lambda).
\end{align*}
 \end{proposition}
 
 \begin{proof}
 Suppose $\alpha,\beta\in(1,+\infty)$ and set $\mu:= \frac{1}{\alpha}$. 
 Assume for a moment that
\begin{equation}\label{eq_dim_buonpositura}
f^{1/\mu}_i-\mathcal{C}^-_{f^{1/\mu}}(x_1,x_3,\lambda)\geq 0, \quad g_i^{1/(1-\mu)}-\mathcal{C}^-_{g^{1/(1-\mu)}}(x_1,x_3,\lambda) \geq0 \qquad \hbox{for $i\in \{1,3\}$}.
\end{equation}
 Then
	\begin{align*}
	f_2 g_2=&\left(f_2^{1/\mu}\right)^\mu \left(g_2^{1/(1-\mu)}\right)^{1-\mu}
	= \lim_{p\to 0}\left(\mu f_2^{p/\mu} + (1-\mu)g_2^{p/(1-\mu)}\right)^{1/p}\\
	\geq& \lim_{p\to 0} \bigg(\mu \left(\lambda f_3^{1/\mu}+(1-\lambda)f_1^{1/\mu}-\mathcal{C}^-_{f^{1/\mu}}(x_1,x_3,\lambda)\right)^{p}+
	\\& + (1-\mu)\left(\lambda g_3^{1/(1-\mu)}+(1-\lambda)g_1^{1/(1-\mu)}-\mathcal{C}^-_{g^{1/(1-\mu)}}(x_1,x_3,\lambda)\right)^{p}\bigg)^{1/p}\\
	\geq& \lim_{p\to 0} \lambda \left(\mu (f_3^{1/\mu}-\mathcal{C}^-_{f^{1/\mu}}(x_1,x_3,\lambda))^p+(1-\mu)\left(g_3^{1/(1-\mu)}-\mathcal{C}^-_{g^{1/(1-\mu)}}(x_1,x_3,\lambda)\right)^p \right)^{1/p}+ \\
	 &+(1-\lambda) \left(\mu (f_1^{1/\mu}-\mathcal{C}^-_{f^{1/\mu}}(x_1,x_3,\lambda))^p+(1-\mu)\left(g_1^{1/(1-\mu)}-\mathcal{C}^-_{g^{1/(1-\mu)}}(x_1,x_3,\lambda)\right)^p\right)^{1/p}\\
	=& \lambda\left(f_3^{1/\mu}-\mathcal{C}^-_{f^{1/\mu}}(x_1,x_3,\lambda)\right)^\mu \left(g_3^{1/(1-\mu)}-\mathcal{C}^-_{g^{1/(1-\mu)}}(x_1,x_3,\lambda)\right)^{1-\mu}+
	\\&+(1-\lambda)\left(f_1^{1/\mu}-\mathcal{C}^-_{f^{1/\mu}}(x_1,x_3,\lambda)\right)^\mu \left(g_1^{1/(1-\mu)}-\mathcal{C}^-_{g^{1/(1-\mu)}}(x_1,x_3,\lambda)\right)^{1-\mu},
\end{align*}
where we used the inequality \cite[Property 7]{kennington}
\[ (1-\lambda)\big(f_1^p+g_1^p\big)^{1/p}+\lambda \big(f_3^p+g_3^p\big)^{1/p}\leq \Big(\big((1-\lambda)f_1+\lambda f_3\big)^p+\big((1-\lambda)g_1+\lambda g_3\big)^p\Big)^{1/p}. \]
		Since $\mu\in (0,1]$, then for all $A\geq B$ %$A,B\geq0$ 
		it holds $(A-B)^\mu\geq A^\mu-B^\mu$, and hence
\begin{align*}
	f_2g_2 \geq& \lambda \big(f_3-(\mathcal{C}^-_{f^{1/\mu}})^\mu(x_1,x_3,\lambda)\big)\big(g_3-(\mathcal{C}^-_{g^{1/(1-\mu)}})^{1-\mu}(x_1,x_3,\lambda)\big)+\\ &+(1-\lambda)\big(f_1-(\mathcal{C}^-_{f^{1/\mu}})^\mu(x_1,x_3,\lambda)\big)\big(g_1-(\mathcal{C}^-_{g^{1/(1-\mu)}})^{1-\mu} %\big) 
	(x_1,x_3,\lambda) \big)
\end{align*}
which implies
 the claim, observed that $\frac{1}{\mu}=\alpha$ and $\frac{1}{1-\mu}=\beta$. 
We are left to verify \eqref{eq_dim_buonpositura}.
If $\mathcal{C}^-_{f^{1/\mu}}(x_1,x_3,\lambda)=0$, then the inequality trivially holds true. 
%Otherwise consider $(y_1,y_3,\tr{\lambda'})\in \overline{K}\times \overline{K}\times [0,1]$ the negative minimum of $\mathcal{C}_{f^{1/\mu}}$. By hypothesis on $M,m$: 
%\begin{align*}
%f^{1/\mu}_1\geq m_1^{1/\mu}\geq M_1^{1/\mu}-m_1^{1/\mu}\geq -\mathcal{C}_{f^{1/\mu}}(y_1,y_3,\tr{\lambda'})=\mathcal{C}^-_{f^{1/\mu}}(y_1,y_3,\tr{\lambda'})\geq \mathcal{C}^-_{f^{1/\mu}}(x_1,x_3,\lambda).
%\end{align*}
Otherwise, %consider $(y_1,y_3,\tr{\lambda'})\in \overline{K}\times \overline{K}\times [0,1]$ the negative minimum of $\mathcal{C}_{f^{1/\mu}}$. By 
by the hypothesis on $M,m$: 
\begin{align*}
f^{1/\mu}_i\geq m_1^{1/\mu}\geq M_1^{1/\mu}-m_1^{1/\mu} 
%-\mathcal{C}_{f^{1/\mu}}=\mathcal{C}^-_{f^{1/\mu}}(y_1,y_3,\tr{\lambda'})
\geq \mathcal{C}^-_{f^{1/\mu}}(x_1,x_3,\lambda).
\end{align*}
Similar computations hold for $g$. 
This concludes the proof.
%
%\[
%\mathcal C_h^-(x_1,x_3,\lambda)
%\le M_1^{1/\mu}-m_1^{1/\mu}
%\le m_1^{1/\mu}
%\le h_i,
%\qquad i=1,3.
%\]
%
%If \(\mathcal C^-_{f^{1/\mu}}(x_1,x_3,\lambda)=0\), the claim is trivial. Otherwise, since
%\(m_1^{1/\mu}\le f^{1/\mu}\le M_1^{1/\mu}\), we have
%\[
%\mathcal C^-_{f^{1/\mu}}(x_1,x_3,\lambda)
%\le M_1^{1/\mu}-m_1^{1/\mu}
%\le m_1^{1/\mu}
%\le f_i^{1/\mu},
%\qquad i=1,3.
%\]
 \end{proof}

We deal now with the boundary case $\alpha=\infty$ and $\beta = 1$.

\begin{proposition}\label{prop_extram_cases}
Let $K$ be a convex set and $f,g\colon K \to \R$ with $\osc(f)$ finite. Then
$$\mc{C}_{fg}(x_1, x_3, \lambda) \geq \mc{C}_g (x_1, x_3, \lambda) f_2 - \osc(f) (\lambda |g_3| + (1-\lambda) |g_1|).$$
In particular, if $f$ is nonnegative and $g$ is concave and bounded, then 
$$\mc{C}_{fg}(x_1,x_3,\lambda) \geq - \norm{g}_{\infty} \osc(f).$$
\end{proposition}

\begin{proof}
 We have
\begin{align*}
\mc{C}_{fg}(x_1,x_3,\lambda) &= f_2 g_2 - \lambda f_3 g_3 - (1-\lambda) f_1g_1 \\
& = \mc{C}_g(x_1, x_3, \lambda) f_2 + \lambda g_3 (f_2-f_3) + (1-\lambda) g_1 (f_2-f_1)
\end{align*}
and thus the two claims.
\end{proof}

\begin{proposition}\label{prop_prel_gxj}
%	Let $K\subseteq \R^n$, $n\geq 1$, be a convex set. Let $g\colon K\to \R$ 
%	and $j\in \left\{1,\dots,n\right\}$, and consider $f\colon x \in K \mapsto \frac{g(x)}{(x^j)^{2}} \in \R$. 
Let $n\geq 1$, $j\in \left\{1,\dots,n\right\}$, $K\subseteq \{x \in \R^n \mid x^j \neq 0\}$ be a convex set. Let $g\colon K\to \R$ 
	and consider $f\colon x \in K \mapsto \frac{g(x)}{(x^j)^{2}} \in \R$. 
Then for all $(x_1,x_3,\lambda)\in\Dom(\mc{HC}_{f})$ % $x_1,x_3 %x,y
%\in K$ and $\lambda\in[0,1]$ 
it holds
	\[ 	\mathcal{HC}_{f}(x_1,x_3,\lambda)\geq \frac{\mathcal{C}_g(x_1,x_3,\lambda)}{(x_2^j)^2} .\]
\end{proposition}
\begin{proof}
	For clarity, let us set $\zeta:=x^j$. 
	If $f_1=f_3=0$, then $g_1=g_3=0$ and thus
	$$\mathcal{HC}_{f}(x_1,x_3,\lambda) = f_2 = \frac{g_2}{\zeta_2^2} = \frac{\mathcal{C}_g(x_1,x_3,\lambda)}{\zeta_2^2}.$$
	Assume now $\lambda f_1+(1-\lambda) f_3>0$. Then we have
	\begin{align*}
		\mathcal{HC}_{f}(x_1,x_3,\lambda)=&
		\frac{g_2}{\zeta^2_2}-\frac{g_1g_3}{\lambda g_1 \zeta_3^2+(1-\lambda)g_3 \zeta_1^2}\\=&
		\frac{\lambda g_3+(1-\lambda)g_1 + \mathcal{C}_g(x_1,x_3,\lambda)}{\zeta^2_2}-\frac{g_1g_3}{\lambda g_1 \zeta_3^2+(1-\lambda)g_3 \zeta_1^2}\\=& \lambda(1-\lambda)\frac{(g_1 \zeta_3-g_3 \zeta_1)^2}{\zeta_2^2(\lambda g_1 \zeta_3^2+(1-\lambda)g_3 \zeta_1^2)}+\frac{\mathcal{C}_g(x_1,x_3,\lambda)}{\zeta_2^2}\geq \frac{\mathcal{C}_g(x_1,x_3,\lambda)}{\zeta_2^2}.\qedhere
	\end{align*}
\end{proof}

\begin{proposition}\label{proposition f-k and harmonic concavity function}
	Let $K$ be a convex set and $f,g :K\to \R $ with $g>0$. 
	Then 
	\begin{equation}\label{eq: stime su HC f-g}
		\mathcal{HC}_{f-g}\geq \mathcal{HC}_f - \mathcal{HC}_g,
	\end{equation}
where the inequality is meant on $\Dom(\mc{HC}_{f-g})$. % \cap \Dom(\mc{HC}_f) $.%\cap \Dom(\mc{HC}_g)$.
	In particular, if $\mathcal{HC}_g \leq 0$ (e.g., if $g$ is constant or $g(t)=t^{\gamma}$ with $\gamma \in [-1,0]$), then
		$$\mathcal{HC}_{f-g}\geq \mathcal{HC}_f.$$
\end{proposition}
		\begin{proof}
	Set for simplicity $h:=f-g$.
	Let $x_1,x_3\in K$ and $\lambda \in [0,1]$. 
		Suppose $h_1=h_3=0$, then $f_1=g_1$ and $f_3=g_3$, which imply $\lambda f_1+(1-\lambda)f_3 =\lambda g_1+(1-\lambda)g_3 >0 $. 
		Moreover
		$$		\mathcal{HC}_f(x_1,x_3,\lambda)- \mathcal{HC}_g(x_1,x_3,\lambda)= f_2-g_2 %g_2-f_2 
= h_2 = \mc{HC}_h(x_1,x_3,\lambda).$$
		Otherwise, if $\lambda h_1+(1-\lambda) h_3>0$, then $\lambda f_1+(1-\lambda) f_3>\lambda g_1 + (1-\lambda) g_3 >0$ by the positivity of $g$.
		Moreover a straightforward computation implies
		\begin{align*}
	\MoveEqLeft	\mc{HC}_h(x_1,x_3,\lambda) -\mc{HC}_f(x_1,x_3,\lambda) +\mc{HC}_g(x_1,x_3,\lambda) =& \\
		&= \frac{\lambda(1-\lambda)(f_1 g_3 - f_3 g_1)^2}{\big(\lambda f_1 + (1-\lambda)f_3 \big) \big(\lambda g_1 + (1-\lambda) g_3\big) \big( \lambda (f_1-g_1) + (1-\lambda) (f_3-g_3)\big)} \geq 0. 	\tag*{\qedhere}
		\end{align*}
	\end{proof}
%\tr{\begin{remark}
%		The inequality \eqref{eq: stime su HC f-g} is meant to hold whenever $\mathcal{HC}_{f-g}$ is well defined. 
%\end{remark}}
\medskip

\printbibliography[title={References}]

\end{document}